\documentclass[11pt, letterpaper]{amsart}
\usepackage{amsmath}
\usepackage{amssymb}
\usepackage{amsfonts}
\usepackage{mathrsfs}
\usepackage{tikz}
\usetikzlibrary{decorations.markings}
\usepackage{tikz-cd}
\usepackage[utf8]{inputenc}
\usepackage{fullpage} 
\usepackage{dsfont}
\usepackage{amsthm}
\usepackage{hyperref}
\usepackage[all]{xy}
\usepackage{float}
\usepackage{enumitem}  
\usepackage{mathtools}

\newtheorem{thm}{Theorem}[section]
\newtheorem{prop}[thm]{Proposition}
\newtheorem{cor}[thm]{Corollary}

\newtheorem{lem}[thm]{Lemma}

\theoremstyle{definition}
\newtheorem{defn}[thm]{Definition}

\newtheorem{exmp}[thm]{Example}
\newtheorem{rmk}[thm]{Remark}
\theoremstyle{remark}

\newtheorem{question}[thm]{Question}

\newcommand{\Aug}{\mathrm{Aug}}

\newcommand{\wAug}{\widetilde{\mathrm{Aug}}}

\newcommand{\J}{\mathcal{J}}

\newcommand{\F}{\mathbb{F}}

\newcommand{\Z}{\mathbb{Z}}

\newcommand{\R}{\mathbb{R}}

\newcommand{\SA}{\mathcal{A}}
\newcommand{\e}{\varepsilon}
\newcommand{\seed}{\mathbf{s}}
\newcommand{\rank}{\mathrm{rank}}

\newcommand{\bF}{\mathbb{F}}

\newcommand{\bfA}{\mathbf{A}}

\newcommand{\Spec}{\mathrm{Spec}}
\newcommand{\Gr}{\mathrm{Gr}}
\newcommand{\SL}{\mathrm{SL}}
\newcommand{\st}{\text{st}}

\renewcommand{\tt}{\mathfrak{t}}
\newcommand{\tb}{\mathrm{tb}}

\newcommand{\La}{\Lambda}

\title{Augmentations, Fillings, and Clusters for 2-Bridge Links}
\author{Orsola Capovilla-Searle}
	\address{University of California Davis, Dept. of Mathematics, Shields Avenue, Davis, CA 95616, USA}
	\email{ocapovillasearle@ucdavis.edu}
\author{James Hughes}
	\address{Duke University, Dept. of Mathematics, Durham, NC 27708, USA}
	\email{jhughes@math.duke.edu}
	\author{Daping Weng}
	\address{University of North Carolina, Dept. of Mathematics, Chapel Hill, NC 27599, USA}
	\email{dweng@unc.edu}

	
\begin{document}

\maketitle
\begin{abstract}
    We produce the first examples relating non-orientable exact Lagrangian fillings of Legendrian links to cluster theory, showing that the ungraded augmentation variety of certain max-tb representatives of Legendrian $2$-bridge links is isomorphic to a product of $A_n$-type cluster varieties. As part of this construction, we describe a surjective map from the set of (possibly non-orientable) exact Lagrangian fillings to cluster seeds, producing a product of Catalan numbers of distinct fillings. We also relate the ruling stratification of the ungraded augmentation variety to Lam and Speyer's anticlique stratification of acyclic cluster varieties, showing that the two coincide in this context. As a corollary, we apply a result of Rutherford to show that the cluster-theoretic stratification encodes the information of the lowest $a$-degree term of the Kauffman polynomial of the smooth isotopy class of the $2$-bridge links we study.
\end{abstract}

\setcounter{tocdepth}{1}

\tableofcontents

\section{Introduction}

In this article, we initiate the use of cluster-algebraic methods to study of non-orientable exact Lagrangian fillings of Legendrian links in the standard contact Darboux ball $\R^3_{\st}$. Exact Lagrangian fillings of Legendrian links are central objects in the study of low-dimensional contact and symplectic topology~\cite{ekholm_etnyre_sabloff_2009, sabloff_traynor_2013, hayden_sabloff_2014, bourgeois_sabloff_traynor_2015}. A series of recent results~\cite{CASG, Casals_zaslow,  Gao_shen_weng_2, CasalsNg, Hughes2021, ABL22, Casals_weng, casals2023lagrangian} building off of previous work from the past decade~\cite{ehk, Pan, STWZ, TreumannZaslow} comprise a concerted effort towards a systematic means of producing and distinguishing orientable exact Lagrangian fillings. 
Many of the most recent advances rely heavily on cluster theory, an algebraic tool that connects combinatorics and geometry. These in turn have led to structural results regarding the homology of the augmentation variety~\cite{CGGLSS}, as well as a conjectural classification of orientable exact Lagrangian fillings of Legendrian positive braid closures~\cite[Conjecture 5.1]{CasalsLagSkel}.  

Given the importance of cluster algebras to the classification of exact Lagrangian fillings, an essential question is to identify which Legendrian links admit a cluster structure on their sheaf moduli or augmentation variety. In collaboration with Casals, the third author previously gave a sufficient condition for identifying such Legendrians in terms of the combinatorics of plabic graphs; see~\cite[Section 4.8]{Casals_weng}. Another general class of Legendrians whose augmentation varieties admit a cluster structure are $(-1)$-closures of positive braids of the form $\beta=\beta'\delta(\beta')$ where $\delta(\beta')$ denotes the Demazure product of the positive braid $\beta'$~\cite{CGGLSS}.   

While cluster-algebraic methods have greatly advanced our understanding of orientable exact Lagrangian fillings, relatively little is known about non-orientable exact Lagrangian fillings. 
In contrast with orientable fillings, whose Euler characteristic is known to be fixed by work of Chantraine~\cite[Theorem 1.3]{chantraine}, a Legendrian may admit (finitely many) distinct homeomorphism classes of decomposable non-orientable exact Lagrangian fillings~\cite[Proposition 1.2]{CCPRSY}. In addition, some knots may admit both orientable and non-orientable fillings~\cite[Example 2.4]{CCPRSY}. Recent work of the first author with Casals provides new techniques for distinguishing exact Lagrangian fillings via Newton polytopes. They then apply these techniques to show that there are Legendrian knots with infinitely many non-orientable exact Lagrangian fillings~\cite{orsola_roger}.

 \subsection{Main results}
The main goal of this work is to extend cluster-algebraic techniques for distinguishing exact Lagrangian fillings to the non-orientable setting. We restrict our focus to specific Legendrian $2$-bridge links where our computations are aided by the relative simplicity of the Legendrian contact dg-algebra.
Given a sequence $n_1, \dots, n_k$ with $n_1\geq 1$, $n_k\geq 1$, and $n_i\geq 2$ for all $1<i<k$, we associate to it the max-tb Legendrian $2$-bridge link $\La[n_1, \dots, n_k]\subseteq \R^3_{\st}$ given by the front projection shown in Figure \ref{fig:front}.\footnote{In general, $2$-bridge links are not Legendrian simple; see~\cite{Foldvari_2019} for additional discussion, including a sufficient condition for Legendrian simplicity. Our choice of front projection identifies a particular choice of max-tb representative.} Such Legendrians are referred to as being in Legendrian rational form; see Definition \ref{def: Legendrian rational form} for details.

\begin{figure}[!htb]
	\centering
	\begin{tikzpicture}[scale=1]
		\node[inner sep=0] at (0,0)
		{\includegraphics[width=8.3cm]{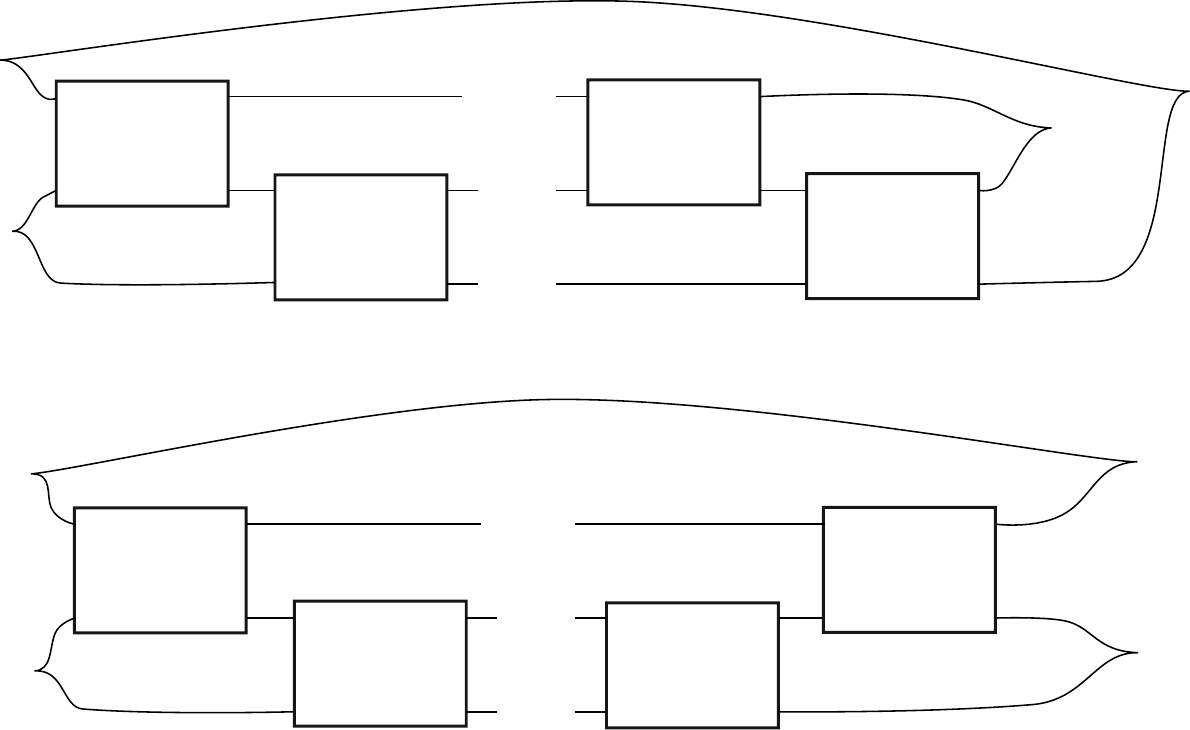}};
	\node at (-3.1,1.5){$n_1$};
	\node at (-1.7,0.9){$n_2$};
	\node at (-0.5,1.5){$\ldots$};
	\node at (-0.5,0.9){$\ldots$};
	\node at (0.6,1.5){$n_{2l-1}$};
	\node at (2.05,0.9){$n_{2l}$};
	\node at (-0.5,-1.5){$\ldots$};
	\node at (-0.5,-2){$\ldots$};
	\node at (-3.1,-1.5){$n_1$};
	\node at (-1.5,-2.1){$n_2$};
	\node at (0.7,-2.1){$n_{2l}$};
	\node at (2.2,-1.5){$n_{2l+1}$};
	\end{tikzpicture}\caption{Front projections of $2$-bridge links $\Lambda[n_1, \ldots, n_{2l}]$ and $\Lambda[n_1, \ldots, n_{2l+1}]$, where each box labelled $n_i$ contains $n_i$ crossings.}\label{fig:front}
	\end{figure}
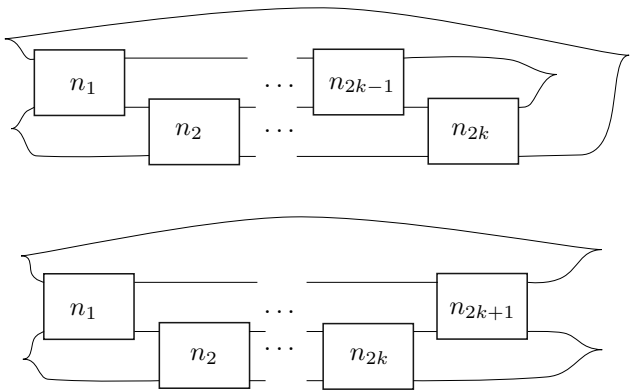

When the rotation number of $\Lambda$ is 0 and all Reeb chords are non-negatively graded, the augmentation variety is often defined as $\Spec( H_0(\SA(\Lambda),\partial))$. However, there are two primary technical obstacles to naively extending the construction of cluster structures from the orientable setting to the case of Legendrian $2$-bridge links $\Lambda=\La[n_1, \dots, n_k]$ when $k\geq 2$: first, the Legendrian dg-algebra $\SA(\Lambda)$ may lack a $\Z$-grading when the rotation number is nonzero; second, even if $\SA(\Lambda)$ admits a $\Z$-grading, $\SA(\Lambda)$ is in general not concentrated in non-negative degrees. We circumvent these obstacles by studying the ungraded augmentation variety $\Aug_u(\La[n_1, \dots, n_k])$ given by 
\[
\Aug_u(\La[n_1, \dots, n_k])=\{\e~|~\e~\text{is an ungraded $\F$-valued augmentation of}~\SA(\Lambda[n_1,\dots, n_k])\}/\sim
\]
where two augmentations are considered to be distinct up to dg-algebra homotopy and $\F$ is an algebraically  closed field of characteristic 2. See Section \ref{sec:aug} for a more detailed description.

The set of (decomposable) exact Lagrangian fillings of $\La[n_1, \ldots, n_k]$ for $n_1\geq 1,~n_k\geq 1$, and $n_i\geq 2$ for $2\leq i\leq k-1$ that we consider are non-orientable except for a few cases. The appearance of non-orientable fillings also requires us to consider a particular enhancement of the Legendrian dg-algebra with group ring coefficients. In the orientable setting, Poincar\'e duality allows for the computation of the dg-algebra with $R[H_1(L; \Z)]$  coefficients for a given ground ring $R$. However, for non-orientable fillings $L$, we lack Poincar\'e duality over $\Z$. Therefore, we consider the group ring $R[H_1(L\backslash \tt, \La \backslash \tt; \Z)]$ where $\tt$ denotes the set of base points on $\La$.

Other than these technical details, our approach mirrors that of~\cite{Gao_shen_weng_2}. Computing the Legendrian dg-algebra of $\La[n_1, \dots, n_k]$ reveals an algebraic isomorphism between the coordinate ring of $\Aug_u(\La[n_1, \dots, n_k])$ and a well-known family of cluster algebras. The cluster variables we obtain are polynomials in the Reeb chords of $\La[n_1, \dots, n_k]$, thought of as regular functions on $\Aug_u(\La[n_1, \dots, n_k])$. We then show that the dg-algebra map induced by (possibly non-orientable) pinching cobordisms agree with cluster localization. As a result, the toric chart induced by an exact Lagrangian filling is a cluster seed. In the case of the $\La[n_1, \dots, n_k]$, the cluster algebras we obtain are particularly well-behaved. 
After a detailed computation of the Legendrian dg-algebra and its dg-algebra homotopies over an algebraically closed field $\mathbb{F}$ of characteristic 2, we show the following:

\begin{thm}[Theorems \ref{thm: cluster structure on augmentation variety} and \ref{thm: cluster torus chart from pinching sequence}]\label{thm: main}
    The coordinate ring $\F[\Aug_u(\La[n_1, \dots, n_k])]$ is isomorphic to a cluster algebra of type $A_{n_1-2}\times A_{n_2-3}\times \dots \times A_{n_{(k-1)}-3}\times A_{n_{k}-2}$. Any exact Lagrangian filling of $\La[n_1, \dots, n_k]$ constructed via an admissible pinching sequence induces an embedding of a cluster chart in $\Aug_u(\La[n_1, \dots, n_k])$. In particular, Hamiltonian isotopic fillings from admissible pinching sequences induce the same cluster chart.
\end{thm}

See Section \ref{sec:cobordism} for more on the specific construction of these decomposable exact Lagrangian fillings and the definition of admissible pinching sequences. In the case of $k=1$, we obtain Legendrian $(2, n)$ torus links, for which the statement is known to hold 
~\cite{Gao_shen_weng_2, Hughes2023}. For $k>1$, work of Lipman and Sabloff establishes an explicit criterion for determining the orientability of fillings of $\La[n_1,\dots, n_k]$ \cite[Theorem 1.1]{LipmanSabloff}. In order to state their result, we first note that every one of the $k$ blocks of crossings in the front projection contains either all negative crossings or all positive crossings so that we may consistently refer to such blocks as positive or negative. Any Legendrian 2-bridge link $\La[n_1, \dots, n_k]$ in Legendrian rational form is then orientably fillable if and only if any negative block contains at most two crossings. All of the exact Lagrangian fillings of links not fulfilling this criterion that we construct are non-orientable.

Following Pan's work in the $k=1$ case~\cite{Pan}, we can index our admissible pinching sequences by a collection of permutations, one for each of block. In each block, we obtain a Catalan number $C_n=\frac{1}{n+1}\binom{2n}{n}$ of admissible pinching sequences inducing distinct cluster charts. Applying~\cite[Theorem 1.1]{Pan} to our setting, we obtain a method for indexing distinct fillings corresponding to different collections of admissible pinching sequences. As a result, every cluster seed in $\Aug_u(\La[n_1, \dots, n_k])$ is induced by an admissible pinching sequence, yielding the following:

\begin{thm}[Propositions \ref{prop:admissiblefilling} and \ref{prop: every cluster chart is achieved}]\label{cor: distinct fillings}
There are at least $C_{n_1-1}C_{n_2-2}\dots C_{n_{(k-1)}-2}C_{n_k-1}$ distinct Hamiltonian isotopy classes of 
exact Lagrangian fillings of $\La[n_1, \dots, n_k]$ constructed from admissible pinching sequences, all of which belong to the same smooth isotopy class.
\end{thm}

Moreover, we show that for any maximal fixed set of Reeb chords $\mathcal{R}$ that form an admissible pinching sequence, 
   there are precisely $C_{n_1-1}C_{n_2-2}\dots C_{n_{(k-1)}-2}C_{n_k-1}$ many exact Lagrangian fillings constructed from pinching sequences of the set $\mathcal{R}$ (Proposition \ref{prop: product of Catalan many fillings}).
We conjecture -- in the spirit of~\cite[Conjecture 5.1]{CasalsLagSkel} -- that the fillings we construct comprise an exhaustive collection of exact Lagrangian fillings of $\La[n_1, \dots, n_k]$; see also Remark 1.4 of loc. cit.

As an additional application of Theorem \ref{thm: main}, we compare the ruling stratification on $\Aug_u(\La[n_1, \dots, n_k])$ defined by Henry and Rutherford~\cite{henry_rutherford} to the anticlique stratification for acyclic cluster varieties described by Lam and Speyer~\cite[Section 3]{LS}. Tracing through the construction of Henry and Rutherford, we show that 
the ruling stratification coincides with Lam and Speyer's cluster-theoretic stratification.  
As part of the construction given in Theorem \ref{thm: main}, we produce a map between mutable quiver vertices and adjacent pairs of crossings of $\La[n_1, \dots, n_k]$, sending anticliques $I$ to normal rulings $R_I$.

\begin{thm}[Theorem 
 \ref{thm: stratification coincide}]\label{thm: stratifications}

Given an anticlique $I$ of quiver vertices of the initial seed $\seed$, the corresponding ruling stratum $W_{R_I}$ equals the anticlique stratum $\mathcal{O}_I$ in $\Aug_u(\La[n_1,\dots, n_k])$.
  
\end{thm}

In earlier work, Rutherford uses ungraded normal rulings to show that the augmentation variety encodes the lowest a-degree term of the Kauffman polynomial~\cite[Theorem 3.1]{Rutherford06}. As a corollary of Theorem \ref{thm: stratifications}, we recover the Kauffman polynomial term of the smooth $2$-bridge link corresponding to $\La[n_1, \dots, n_k]$ from the finite field point count of cluster varieties. 
The appearance of the Kauffman polynomial in this setting can be viewed as an ungraded analogue of the relationship between the point counts of braid varieties and the HOMFLY-PT polynomial obtained in \cite{ShenWeng,CGGS1}; see also \cite{Rutherford06} for a similar statement involving the HOMFLY-PT polynomial coming from normal rulings.

The work of Rutherford cited above can be considered in the context of a somewhat longstanding conjecture regarding topological invariance of the ungraded augmentation variety. Specifically, the ungraded augmentation variety $\Aug_u(\La)$ of a Legendrian $\La$ is conjecturally determined only by $\tb(\La)$ and its topological knot type \cite[Conjecture 4.14]{NgRutherford13}; see also Conjecture \cite[Conjecture 3.14]{Ng01}. Gao and Rutherford also present a similar conjecture for exact Lagrangian fillings \cite[Question 4.13]{GaoRutherford}. While we do not answer either of these questions, Corollary \ref{cor: topological invariant} gives some evidence towards the former conjecture in our setting of 2-bridge links. Moreover, we believe that our main result gives a possible blueprint for tying together these two conjectures by realizing (possibly non-orientable) exact Lagrangian fillings as cluster charts in the ungraded augmentation variety.

\subsection{Further questions}
We collect here a number of questions pertaining to both contact topology and cluster algebras raised by our work. 
From the perspective of contact topology, the isomorphism we construct in Theorem \ref{thm: main} gives a recipe for understanding the connection between cluster charts and fillings algebraically. However, an intrinsically contact-geometric definition of a cluster structure -- in the sense of~\cite[Theorem 1.1]{Casals_weng} -- coming from non-orientable exact Lagrangian fillings lies beyond the scope of this work.
In particular, the lack of an integer-valued Poincar\'e duality in the non-orientable setting presents a significant challenge for geometrically realizing mutation and a cluster ensemble structure.

\begin{question}
    Give a contact-geometric description of cluster variables, quivers, and mutation for the cluster structure on $\Aug_u(\La[n_1, \dots, n_k])$.
\end{question}

Parallel to this question is the problem of understanding which Legendrian links admit cluster structures on their (possibly ungraded) augmentation varieties. 
\begin{question}
    For which Legendrian links $\La$, does $\Aug_u(\La)$ admit a cluster structure with cluster charts induced by (possibly non-orientable) exact Lagrangian fillings of $\La$?
\end{question}
\noindent This problem is thoroughly explored for Legendrian $(-1)$-closures of positive braids by~\cite{CGGLSS} but our work leaves room for a more general collection of links admitting possibly non-orientable exact Lagrangian fillings. 

One can also consider the topology of non-orientable exact Lagrangian fillings. Smoothly, any non-orientable surface is isotopic to an orientable surface with a single non-orientable 1-handle attachment. This prompts the following question,

\begin{question}\label{ques:non-orientable} Is every non-orientable exact Lagrangian filling Hamiltonian isotopic to an exact Lagrangian filling constructed via a single non-orientable 1-handle attachment?
\end{question}

See Example~\ref{ex:[n,2]} for decomposable fillings where this is true and the Hamiltonian isotopy is realized by swapping the order of pinch moves. See Example~\ref{ex:[4,4]} for decomposable fillings where there is no admissible pinching sequence with only one non-orientable pinch move. 

\bigskip

Our approach of computing the ungraded augmentation variety arose partly as a result of the relative complexity of computing the moduli of microlocal rank-one sheaves of Legendrians with nonzero rotation number or negatively graded Reeb chords; see e.g.~\cite[Remark 5.3]{STZ_ConstrSheaves} for a suggested approach in this setting.

\begin{question}
 Formulate a sheaf-theoretic analog to the cluster structure on $\Aug_u(\La[n_1, \dots, n_k])$ constructed in Theorem \ref{thm: main}.
\end{question}
\noindent An adaptation of the Legendrian weave calculus of \cite{Casals_zaslow} could yield a combinatorial tool to aid in this computation.

\bigskip
 
Smooth $2$-bridge links first appeared in a cluster-algebraic context in work of Lee and Schiffler~\cite{LeeSchiffler}, and were later considered in~\cite{NagaiTerashima, yacavone2019cluster, baziermatte2021knot}. However, their approach is significantly different from ours: 
they construct one single cluster algebra for a family of $2$-bridge links and show that the Jones polynomial of each $2$-bridge link in the family is equal to a specialization of an $F$-polynomial of the cluster algebra. Nevertheless, the cluster algebras they obtain resemble the cluster algebras we describe here.
On the other hand, the Kauffman polynomial is also known to specialize to the Jones polynomial. Thus, we raise the following question:
\begin{question}\label{question: cluster Kauffman}
Does the cluster algebra $\Aug_u(\La[n_1, \dots, n_k])$ encode the entire Kauffman polynomial of the topological link type of $\La[n_1, \dots, n_k]$?   
\end{question}

\noindent Further exploration in this vein might also involve relating explicit exact Lagrangian fillings to the snake graphs of Lee and Schiffler in order to provide a contact-geometric interpretation of the $F$-polynomial they obtain. Their work also prompts the question of whether, for other augmentation varieties $\Aug(\La)$ admitting a cluster structure, there exists an $F$-polynomial that specializes to the Jones polynomial of the underlying topological link type of $\Lambda$. See also \cite{baziermatte2021knot} for more constructions of cluster algebras related to topological links, as well as \cite{ShenWeng, galashin2022plabic} for additional cluster theoretic applications to polynomial invariants of knots.


\subsection*{Acknowledgements} 
O.~Capovilla-Searle is supported by the NSF Postdoctoral Research Fellowship DMS-2103188. J. Hughes was partially supported by the NSF grant DMS-1942363. We thank Lenhard Ng for providing key insight into how to define the graded dg-algebra map induced by a non-orientable pinch move and for suggesting the use of co-oriented base points. Thanks also to Eugene Gorsky, Wenyuan Li, Linhui Shen, and Eric Zaslow for useful discussions. Finally, special thanks to Roger Casals, Honghao Gao, Dan Rutherford, and Joshua Sabloff for helpful comments on a draft of this manuscript. \hfill$\Box$\\

\section{Background on Legendrian Links}

In this section we give the necessary background on Legendrian links, their exact Lagrangian fillings, and related invariants. We start by reviewing the Legendrian dg-algebra and then describe the ungraded augmentation variety. The final subsection contains information on exact Lagrangian fillings and their induced dg-algebra maps.
Let $\R^3_{\st}$ denote the standard contact Darboux ball $(\R^3, \xi=ker(dz-ydx))$. A link $\Lambda\subset \R^3_{\st}$ is \emph{Legendrian} if $T_p\Lambda\subset \xi_p$ for all $p\in \Lambda.$ There are two useful projections of a Legendrian in $\R^3_{\st}$: the \emph{front projection} given by $\Pi_{xz}(x,y,z)=(x,z)$, and the \emph{Lagrangian projection} $\Pi_{xy}(x,y,z)=(x,y)$.

\subsection{The Legendrian dg-algebra} Here we give the key background on the Legendrian differential graded algebra, also known as the Legendrian contact homology dg-algebra or the Chekanov-Eliashberg dg-algebra. For more details see~\cite{Etnyre_ng} for a recent survey.

\medskip

\textbf{The Generators of $\SA(\Lambda;R).$} 
Let $\Lambda$ be a oriented Legendrian link in $\R^3_{\st}$ and let $R$ be an algebra over a field of characteristic $2$. We decorate $\Lambda$ with a generic set of \emph{co-oriented} base points $\tt$, whose orientation can either agree or disagree with the orientation of $\Lambda$. A \emph{Reeb chord} $x$ of $\Lambda\subset \R^3_{\st}$ is a trajectory of the Reeb vector field $\partial_z$ of $\R^3_{\st}$ that begins and ends on $\Lambda.$ By genericity, we assume the base points do not intersect any of the Reeb chords. Let $\mathcal{R}(\Lambda)$ denote the set of Reeb chords of $\Lambda.$ As an algebra, the Legendrian dg-algebra $\SA(\Lambda;R)$ is a commutative\footnote{Typically, the Legendrian dg-algebra is defined as a non-commutative algebra, but for the purposes of obtaining a cluster algebra structure on the augmentation variety, we allow all generators to commute with each other.} algebra over $R$ with the generating set $\mathcal{R}(\Lambda)\cup \{s_i,s_i^{-1}\mid s_i\in \tt\}$, modulo the relation $s_is_i^{-1}=s_i^{-1}s_i=1$ for all $s_i\in \tt$. When $R=\mathbb{Z}_2$, we denote the Legendrian dg-algebra by $\SA(\Lambda)$.

\begin{rmk} As we will be working with the ungraded augmentation variety, we ignore the grading of elements in $(\SA(\Lambda;R),\partial)$ throughout this article. In particular, all augmentations, dg-algebra homotopies, normal rulings, and Morse complex sequences are ungraded. 
\end{rmk}

\begin{rmk} Sometimes it is convenient to think of a base point $s\in \tt$ not as an invertible generator for $\SA(\Lambda;R)$ but as an invertible element in $R$. Note that these two interpretations bear no difference algebraically.
\end{rmk}

\textbf{The Differential $\partial.$} The differential $\partial$ counts rigid $\mathcal{J}$-holomorphic disks in the symplectization $(\R_t\times \R^3, d(e^t\alpha))$ with boundary on $\R \times \Lambda$ and punctures that are asymptotic to Reeb chords of $\Lambda$. See~\cite[Section 3.5]{Etnyre_ng} for more details. We will now briefly describe how to combinatorially compute the differential of the Legendrian dg-algebra from a Lagrangian projection of $\Lambda$.

Observe first that every crossing of the Lagrangian projection $\Pi_{xy}(\Lambda)$ corresponds to a Reeb chord of $\Lambda.$ The four quadrants of every crossing in $\Pi_{xy}(\Lambda)$ are decorated with a Reeb sign. The two quadrants where if one traverses the boundary of a quadrant near the crossing in the counterclockwise direction one moves from an understrand to an overstrand are decorated with a positive Reeb sign. The other two quadrants are decorated with a negative Reeb sign. Each rigid $\J$-holomorphic disk with $m+1$ boundary punctures labeled by $p, q_1, \dots, q_m$
is projected to an immersed disk in $\R^2$, 
$$\Delta: (\mathbf{D}_{m+1}, \partial \mathbf{D}_{m+1}) \rightarrow (\R^2, \Pi_{xy}(\Lambda))$$
such that
\begin{itemize}
	\item a neighborhood of each puncture is mapped to one of the four quadrants of a crossing of $\Pi_{xy}(\Lambda)$;
	\item $\Delta(p)=a$ and the neighborhood of $a$ is mapped to a quadrant of the crossing $a$ decorated with a positive Reeb sign;
	\item for $i=1,\ldots, m$, $\Delta$ sends $q_i$ to $b_i$ and a neighborhood of $q_i$ is mapped to a quadrant of the crossing $b_i$ decorated with a negative Reeb sign.
\end{itemize}

 Let $\Delta(a)$ denote the moduli space of such rigid holomorphic disks $\Delta$. Define $w(\Delta)$ to be the product of Reeb chords at the negative Reeb sign corners and base points on the boundary of the disk as we go around the disk in the counterclockwise direction starting and ending at the corner at $a$. A co-oriented basepoint $ s_{i}$ contributes $s_{i}$ if the orientation on the boundary of the disk agrees with the co-orientation of $s_i$ and $s_{i}^{-1}$ otherwise.

For each Reeb chord $a$ of $\Lambda$ given by a crossing of $\Pi_{xy}(\Lambda)$, the differential is given by:
$$\partial(a)= \sum_{\Delta\in \Delta(a)}w(\Delta).$$

The Legendrian dg-algebra $(\SA(\Lambda),\partial)$ is well-known to be a Legendrian invariant in the following sense.

\begin{thm}[{\cite[Theorem 3.3]{Chekanov}}]\label{thm: stable tame isomorphism} If $\Lambda$ and $\Lambda'$ are Legendrian isotopic, then their Legendrian dg-algebras $(\SA(\Lambda),\partial)$ and $(\SA(\Lambda'),\partial')$ are stable tame isomorphic.
\end{thm}

In this theorem, ``stable'' means a composition of adding two free generators $e_1$ and $e_2$ and a differential relation $\partial e_1=e_2$; ``tame'' means a composition of elementary isomorphisms that send one of the free generators $a_i$ to $a_i+w$ where $w$ does not involve $a_i$ or to $ua_iv$ where $u,v$ are invertible elements in $\SA(\Lambda)$, while sending all other free generators to themselves.

\medskip

Let $I$ be the two-sided ideal in $\SA(\Lambda)$ generated by $\mathrm{Im} (\partial)$. The quotient $\mathcal{C}(\Lambda):=\SA(\Lambda)/I$ is called the \emph{characteristic algebra} of $\Lambda$ \cite[Definition 3.1]{ng_2003}. Following Theorem \ref{thm: stable tame isomorphism}, Ng proves the following:

\begin{thm}[{\cite[Theorem 3.4]{ng_2003}}]\label{thm: characteristic algebra isomorphism} If $\Lambda$ and $\Lambda'$ are Legendrian isotopic, then $\mathcal{C}(\Lambda)$ and $\mathcal{C}(\Lambda')$ are tame isomorphic after possibly adding free generators.
\end{thm}

Note that the free generators are precisely the $e_1$'s in the stable part of Theorem \ref{thm: stable tame isomorphism}.

\subsection{Ungraded augmentations of the Legendrian dg-algebra}\label{sec:aug}

As it can be difficult to compute and compare stable tame isomorphism classes of Legendrian dg-algebras it is often easier to work with Legendrian invariants arising from augmentations.

\begin{defn}\label{defn:augmentation}
Let $\bF$ be an algebraically closed field of characteristic $2$. An ($\bF$-valued) \emph{ungraded augmentation} of $\Lambda$ is a $\mathbb{Z}_2$-differential algebra homomorphism $\e:(\SA(\Lambda;R),\partial)\longrightarrow (\bF,0)$. 
Let $\wAug_u(\Lambda)$ denote the set of ungraded augmentations of $\Lambda.$ 
\end{defn}

 The spectrum of the abelianization of the characteristic algebra $\mathcal{C}(\Lambda)$ is an affine scheme $X(\Lambda)$ over $\Z_2$. The set $\wAug_u(\Lambda)$ can then be understood algebraically as $X(\Lambda)(\bF)$, the set of $\bF$-points in $X(\Lambda)$. Since $\bF$ is assumed to be algebraically closed, $\wAug_u(\Lambda)=X(\Lambda)(\bF)$ is automatically an affine variety. The following is a direct consequence of Theorem \ref{thm: characteristic algebra isomorphism}.

\begin{cor}[{\cite[Corollary 3.15]{ng_2003}}]\label{cor: isomorphism before dgahtpy} If $\Lambda$ and $\Lambda'$ are Legendrian isotopic, then $\wAug(\Lambda)$ is isomorphic to $\wAug(\Lambda')$ up to multiplying by affine space factors.
\end{cor}

\begin{defn}\label{dg-algebra-homotopy}
	Let $\Lambda$ be a Legendrian link in $\R^3_{\st}$. Two ungraded augmentations $\e_1, \e_2:(\mathcal{A}(\Lambda;R), \partial)\rightarrow(\bF,0)$  are \emph{dg-algebra homotopic} if there exists a $\mathbb{Z}_2$-linear map
	$h:\mathcal{A}(\Lambda;R)\rightarrow  \bF$ 
 such that
	\begin{enumerate}
		\item  $h(x\cdot y)=h(x)\cdot \e_2(y)+\e_1(x)\cdot h(y)$ for all $x,y\in \mathcal{A}(\Lambda;R)$; 
		\item For any Reeb chord $a$ of $\Lambda$, $\e_1(a)-\e_2(a)=h\circ\partial(a).$
	\end{enumerate}	
\end{defn}

\begin{defn}\label{defn:augvariety}
Let $\Lambda\subset \R^3_{\st}$. The \emph{ungraded augmentation moduli space (variety)} of $\Lambda$, $\Aug_{u}(\Lambda)$, is the quotient space $\wAug_{u}(\Lambda)/\sim$, where two augmentations are considered to be equivalent if they are dg-algebra homotopic.
\end{defn}

\begin{rmk} Note that by Definitions \ref{defn:augmentation} and \ref{defn:augvariety}, $\wAug_u(\Lambda)$ and $\Aug_u(\Lambda)$ only depend on $(\SA(\Lambda),\partial)$ as a differential algebra. Changing the choice of generating set of $\SA(\Lambda)$ does not change $\wAug_u(\Lambda)$ or $\Aug_u(\Lambda)$; thus, $\wAug_u(\Lambda)$ and $\Aug_u(\Lambda)$ are automatically invariant under tame isomorphisms.
\end{rmk}

Although $\wAug_u(\Lambda)$ is always an affine variety, it is not clear in general that $\Aug_u(\Lambda)$ is still an affine variety after quotienting by dg-algebra homotopies. Nevertheless, the following proposition remains true.

\begin{prop}\label{prop: set bijection} If $\Lambda$ and $\Lambda'$ are Legendrian isotopic, then the isomorphism in Corollary \ref{cor: isomorphism before dgahtpy} descends to a natural bijection between $\Aug_u(\Lambda)$ and $\Aug_u(\Lambda')$ as sets.
\end{prop}
\begin{proof} Note that the extra affine space factors in Corollary \ref{cor: isomorphism before dgahtpy} correspond to the free generators $e_1$'s in the stable part of Theorem \ref{thm: stable tame isomorphism}. If two augmentations $\e_1$ and $\e_2$ differ only at such a free generator $e_1$, then we can define $h(e_2):=\e_1(e_1)-\e_2(e_1)$ and deduce that $\e_1$ and $\e_2$ are dg-algebra homotopic. This shows that the affine space factors in $\wAug_u(\Lambda)$ and $\wAug_u(\Lambda')$ do get quotiented out under dg-algebra homotopies.
\end{proof}

We will prove in Proposition \ref{prop:dgahtpy} that $\Aug_u(\Lambda[n_1,\dots, n_k])$ is an affine variety for any Legendrian rational form $\Lambda[n_1,\dots, n_k]$ (see Definition~\ref{def: Legendrian rational form}). Thus, we can use this Legendrian rational form and Proposition \ref{prop: set bijection} to endow $\Aug_u(\Lambda)$ with an affine variety structure for any $\Lambda$ that is Legendrian isotopic to $\Lambda[n_1,\dots, n_k]$. Thus, we can call $\Aug_u(\Lambda)$ the ungraded augmentation variety with no ambiguity for the set of links we consider in this article. Moreover, we will show in Corollary \ref{cor: topological invariant} that $\Aug_u(\Lambda)$ of Legendrian $2$-bridge links $\Lambda$ that admit Legendrian rational forms do not depend on the choice of Legendrian rational forms of $\Lambda$.

\begin{rmk}\label{rmk:splitdgahmtpy}
By~\cite[Proposition 5.19]{aug_sheaves}, augmentations up to dg-algebra homotopy are in bijection with objects of the augmentation category $\Aug_+$ for Legendrian knots. For links, one should consider split-dg-algebra homotopies (see~\cite[Definition 5.3]{wig}) instead of dg-algebra homotopies to obtain an equivalence between the set of augmentations of $\Lambda$ up to split dg-algebra homotopy with $\Aug_+(\Lambda)$~\cite[Proposition 5.5]{wig}. As split dg-algebra homotopies are a refinement of dg-algebra homotopies that do not greatly affect our results for $\Lambda[n_1, \ldots, n_k]$ we omit them. Note that these equivalences are stated for graded augmentations but carry over for ungraded augmentations. It is also worth noting that for Legendrian knots (and graded augmentations) there is an $A_{\infty}$ equivalence between the augmentation category $\Aug_+(\Lambda)$ and the category of microlocal rank-one sheaves with singular support on $\Lambda$~\cite[Theorem 7.1]{aug_sheaves}. This result is expected to extend for links.  It is still not known if there is a similar equivalence between an $\Aug_+$ category constructed from ungraded augmentations and a ``suitable" moduli space of sheaves.
\end{rmk}

\subsection{Exact Lagrangian fillings}\label{sec:cobordism}

\begin{defn}\label{defn:cobord} 
	Let $\Lambda_{+}$ and $\Lambda_{-}$ be Legendrian links in $\R^{3}_{\st}$.  
	An  \emph{exact Lagrangian cobordism} $L$ from $\Lambda_{-}$ to $\Lambda_{+}$ is an embedded Lagrangian surface in the symplectization $(\R_t\times \R^{3}, d(e^t(dz-y dx)))$
	   that  has  cylindrical ends  and is exact in the following sense:  
for some $N>0$, 
	\begin{enumerate}
		\item  $L \cap ([-N,N]\times Y)$ is compact,
		\item  $L \cap ([N,\infty)\times Y)=[N,\infty)\times \Lambda_{+}$, 
		\item  $L\cap ((-\infty,-N]\times Y)=(-\infty,-N]\times \Lambda_{-}$, and 
	\item there exists  a function $f: L \rightarrow \R$  and constants $\mathfrak c_\pm$ such that 
		$e^t\alpha|_{L} = df$, where $f|_{(-\infty, -N] \times \Lambda_{-}} = \mathfrak c_{-}$, and $f|_{[N, \infty) \times \Lambda_{+}} = \mathfrak c_{+}$.
	\end{enumerate}
	An  \emph{exact Lagrangian filling}  $L$ of a Legendrian link $\Lambda$ is an exact Lagrangian cobordism from $\emptyset$ to $\Lambda$. 
\end{defn}

Exact Lagrangian fillings give rise to $k$-systems of augmentations of the Legendrian dg-algebra.

\begin{defn}\label{defn:systemsaug}\cite[Definition 3.9]{CasalsNg}
A \emph{$k$-system of ungraded augmentations} of a Legendrian link $\Lambda$ is a $\mathbb{Z}_2$-differential algebra homomorphism
$$\e: (\mathcal{A}(\Lambda; R), \partial) \longrightarrow (\bF[s_1^{\pm1}, \ldots, s_k^{\pm1}],0).$$
Two $k$-systems of ungraded augmentations 
$$ \e: \mathcal{A}(\Lambda;R) \longrightarrow \bF[s_1^{\pm1}, \ldots, s_k^{\pm1}]~\text{and}~ \e ': \mathcal{A}(\Lambda;R) \longrightarrow \bF[s_1'^{\pm1}, \ldots, s_k'^{\pm1}]$$
are equivalent if there exists an $\bF$-algebra automorphism $$\phi: \bF[s_1^{\pm1}, \ldots, s_k^{\pm1}] \rightarrow \bF[s_1'^{\pm1}, \ldots, s_k'^{\pm1}]$$ such that $\e '=\phi \circ \e$.
\end{defn}

A $k$-system of ungraded augmentations $\e: (\mathcal{A}(\Lambda; R), \partial) \longrightarrow (R[s_1^{\pm1}, \ldots, s_k^{\pm1}],0),$ can be thought of as a family of augmentations. By post-composing with any $\bF$-algebra homomorphism $\eta:\bF[s_1^{\pm, 1} \ldots, s_k^{\pm 1} ]\longrightarrow \bF$, we can obtain an ungraded augmentation 
$$\eta\circ \e: (\mathcal{A}(\Lambda; R), \partial_{\Lambda}) \xrightarrow{\e} 
(R[s_1^{\pm1}, \ldots, s_k^{\pm1}],0)\xrightarrow{\eta} (\F, 0).$$

\begin{thm}[\cite{ehk}]~\label{thm:induceaug}
Let $\Lambda$ be a Legendrian link in $\R^3_{\st}$ with a set of base points $\tt$ (with at least one base point per link component). Suppose that $L$ is an embedded, Maslov-$1$ exact Lagrangian filling of the Legendrian link $\Lambda\subset \R^3_{\st}$. Then $L$ induces a $k$-system of ungraded augmentations
$$\e_{L}: (\SA(\Lambda;  \bF[H_1(L\backslash \tt, \Lambda\backslash \tt;\Z)]), \partial)\rightarrow (\bF[H_1(L\backslash \tt, \Lambda\backslash \tt;\Z)], 0)$$
where 
$k=\rank(H_1(L\backslash \tt,\Lambda\backslash \tt;\Z))$. If $L$ and $L'$ are exact Lagrangian fillings of $\Lambda$ such that there exists a Hamiltonian isotopy from $L$ to $L'$ through exact Lagrangian fillings of $\Lambda$ fixing the boundary, there exists an invertible map $\phi: H_1(L'\backslash \mathfrak{t}, \Lambda\backslash \mathfrak{t};\mathbb{Z})\rightarrow H_1(L\backslash \mathfrak{t}, \Lambda\backslash \tt;\Z)$ such that $\eta\circ\e_{L}$ and $\eta \circ \phi\circ \e_{L'}$ are dg-algebra homotopic for any $\eta:\mathbb{F}[H_1(L\backslash \mathfrak{t}, \Lambda\backslash \mathfrak{t};\mathbb{Z})]\longrightarrow \mathbb{F}$."
\end{thm}

The proof of Theorem~\ref{thm:induceaug} is stated for graded augmentations and orientable fillings but holds in general.

\begin{rmk} Since $L$ is a surface with boundary, as a CW-complex, it deformation retracts onto its 1-skeleton. Thus, $H_1(L\backslash \tt,\Lambda\backslash \tt;\Z)$ has no torsion, which implies that $\bF[H_1(L\backslash \tt, \Lambda\backslash \tt;\Z)]$ is a Laurent polynomial ring, as required by Definition \ref{defn:systemsaug}.
\end{rmk}

\begin{rmk} In \cite{Pan}, Pan applied Poincar\'{e} duality and used $H_1(L;\Z)$ as coefficients in her study of orientable exact Lagrangian fillings. However, in the non-orientable setting, we do not have Poincar\'{e} duality at our disposal, so we have to go with the -- in a sense -- more direct choice $H_1(L\backslash \tt, \Lambda\backslash \tt;\Z)$. This choice of homology coefficients was also used in the construction of a cluster structure on decorated sheaf moduli spaces~\cite{Casals_weng}. 
\end{rmk}

An exact Lagrangian filling is called \emph{decomposable} if it is constructed with the following elementary cobordisms~\cite{ehk}.

\begin{itemize}
    \item \textbf{Cobordisms induced by traces of Legendrian isotopies}.  If $\Lambda_-$ is related to $\Lambda_+$ by a Legendrian isotopy, then there exists a cylindrical exact Lagrangian from $\Lambda_-$ to $\Lambda_+$ which is $\epsilon$ close (in the $C^0$ metric) to the trace of the Legendrian isotopy. For Legendrian links in $\R^3_{\st}$ there exists a set of Legendrian Reidemeister moves in both the Lagrangian projection and the front projection.
    \item \textbf{Minimum cobordism}. A minimum cobordism is the unique exact Lagrangian disk filling of the max-tb unknot $U$ up to Hamiltonian isotopy~\cite{eliashberg_polterovich}. Consider the disjoint union of a Legendrian $\Lambda_-$ and the max-tb unknot $U$. There exists an exact Lagrangian cobordism from $\Lambda_-$ to $\Lambda_-\sqcup U$ given by the union of the exact Lagrangian disk filling of $U$ and the cylinder $\R\times \Lambda_-.$
    \item \textbf{Saddle cobordisms}. Suppose that $a$ is a Reeb chord of $\Lambda_+$ that is both contractible and proper. A Reeb chord is contractible if there is a Legendrian isotopy of $\Lambda_+$ inducing a planar isotopy of $\Pi_{xy}(\Lambda_+)$ and ending in a
Legendrian where the height of the Reeb chord is arbitrarily small. See~\cite[Definition 4.3]{CasalsNg} for a definition of a proper Reeb chord. We know that $a$ corresponds to a crossing in $\Pi_{xy}(\Lambda_+).$ Replace said crossing with its $0$-resolution to obtain $\Pi_{xy}(\Lambda_-).$ Then, there exists an exact Lagrangian saddle cobordism $L_a$ from $\Lambda_-$ to $\Lambda_+.$ Such a procedure is referred to as a \emph{pinch move}. 
Topologically, a pinch move corresponds to adding a $1$-handle.
\end{itemize}

\begin{figure}[h]
	\centering
	\begin{tikzpicture}[scale=1.5]
		\node[inner sep=0] at (0,0) {\includegraphics[width=5 cm]{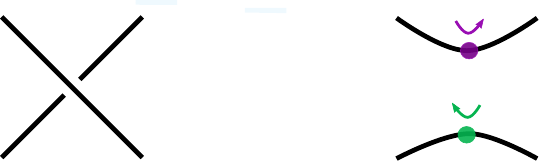}};
		\node at (-1.2,0.3){$a$};
		\node at (1.2,0.035){$s$};
		\node at (1.2,-0.5){$s$};
		\node at (0,0){$\boldsymbol{\rightarrow}$};
  \end{tikzpicture}
	\caption{A pinch move at a Reeb chord $a$, where the crossing $a$ in the Lagrangian projection is replaced with its $0$ resolution and two co-oriented base points $s$.} 
	\label{fig:pinch}
\end{figure}
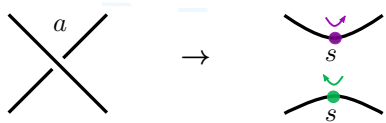

We refer to a pinch move as  \emph{orientable} or \emph{non-orientable} according to the orientability of its saddle surface. Note that if $\Lambda_-$ is obtained from an orientable pinch move on an oriented Legendrian $\Lambda_+$ then $\Lambda_-$ has an induced orientation. In contrast, if $\Lambda_-$ is obtained from a non-orientable pinch move then $\Lambda_-$ has no canonical induced orientation. We now provide a streamlined combinatorial description of the dg-algebra map induced by a saddle cobordism using co-oriented base points. Under minor modification, this combinatorial treatment recovers previous descriptions such as \cite[Section 4]{CasalsNg} in the orientable case and \cite[Section 4.2]{orsola_roger} in the non-orientable case.

The co-orientation at a base point $s$ gives a direction perpendicular to the trace of the base point $s$ on the cobordism. Moreover, each base point
$s$ coming from a pinch move at a Reeb chord $a$,
 in fact, describes a relative homology class in $H_1(L\backslash \tt, \Lambda\backslash \tt;\Z)$. To see this, note that the trace of the Reeb chord $a$ gives a $\J$-holomorphic disk $D_a$ in the complement of $L_a$, and its boundary $\partial D_a$ can be viewed as a $1$-cycle on $L_a$ relative to the boundary $\partial L_a$. Note that the underlying manifold of $\partial D_a$ is the ``unstable submanifold'' of the saddle cobordism. On the other hand, the trace of the newly created pair of base points is the ``stable submanifold'' $S$ of the saddle cobordism (see Figure \ref{fig: saddle cobordism}). For each point $s\in S$ that is not the saddle critical point, we can homotope the relative $1$-cycle $\partial D_a$ along $S$ so that $\partial D_a$ crosses $S$ in exactly the direction given by the orientation of the base point $s$. We consider $s$ as representing this homotoped relative 1-cycle.

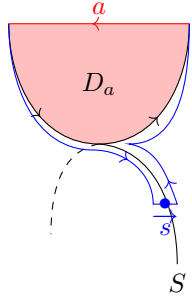
\begin{figure}[H]
    \centering
    \begin{tikzpicture}[scale=0.8]
        \draw [pink, fill=pink] (-1.5,2) to [out=-90,in=180] (0,0) to [out=0,in=-90] (1.5,2)--cycle;
        \node at (0,1) [] {$D_a$};
        \draw [decoration={markings,mark=at position 0.3 with {\arrow{>}}, mark=at position 0.7 with {\arrow{>}}},postaction={decorate}] (-1.5,2) to [out=-90,in=180] (0,0) to [out=0,in=-90] (1.5,2);
        \draw (0,0) to [out=-5,in=90] (1.3,-2) node [below] {$S$};
        \draw [dashed] (0,0)  to [out=-170,in=90] (-0.8,-1.5);
        \draw [red, decoration={markings,mark=at position 0.5 with {\arrow{<}}},postaction={decorate}] (-1.5,2) -- node [above] {$a$} (1.5,2);
        \draw [blue, decoration={markings,mark=at position 0.4 with {\arrow{>}},mark=at position 0.6 with {\arrow{>}}},postaction={decorate}] (-1.5,2) to [out=-90,in=180] (-0.1,-0.1) to [out=0,in=90] (0.9,-1)to [out=0,in=180] (1.3,-1) to [out=110,in=-10] (0.5,0) to [out=0,in=-90] (1.5,2);
        \node [blue]  at (1.1,-1) [] {$\bullet$};
        \node [blue]  at (1.1,-1) [below] {$\overrightarrow{s}$};
    \end{tikzpicture}
    \caption{Topological interpretation of co-oriented base points as relative cycles in a saddle cobordism.}
    \label{fig: saddle cobordism}
\end{figure}

Suppose that $\Lambda_-$ is obtained from $\Lambda_+$ by a single pinch move at a contractible and proper Reeb chord $a$ of $\Lambda_+$ and let $L_a:\Lambda_-\longrightarrow \Lambda_+$ denote the corresponding Lagrangian cobordism. Then, $\Lambda_-$ has two additional \textbf{co-oriented} base points $s$ as shown in Figure~\ref{fig:pinch}. The dg-algebra homomorphism $\Phi_{L_a}:\SA(\Lambda_+;R)\longrightarrow \SA(\Lambda_-;R[s^{\pm 1}])$ induced by the saddle cobordism $L_a$ is defined as follows. Let $\Delta_a(a_i),$ denote the set of immersed disks for $\Lambda_+$ where all corners are convex and there are only two positive corners: one at $a$ and one at $a_i$. For $\Delta\in \Delta_a(a_i)$, let $w_1(\Delta)$ denote the product of the negative corners and base points that we encounter as we traverse $\partial \Delta$ counterclockwise from $a_i$ to $a$. Let $w_2(\Delta)$ denote the product of the negative corners and base points that we encounter as we traverse $\partial \Delta$ counterclockwise from $a$ to $a_i$. Then, 
\[
\Phi_{L_a}(a)=s, \quad \quad \Phi_{L_a}(t)=t \quad \text{for any co-oriented base point $t$, and} 
\]
\[
\Phi_{L_a}(b)=b+\sum_{\Delta\in \Delta_a(b)} \Phi_{L_a}(w_1(\Delta))s^{-1}w_2(\Delta) \quad \text{for any Reeb chord $b\neq a$}.
\]
Note that $\Phi_{L_a}$ is well defined: if we order the Reeb chords $a_1, \ldots, a_r$ in increasing order of height then any disk $\Delta\in \Delta_a(a_j)$ can only have negative punctures at $a_1, \ldots, a_{j-1}.$ Observe also that in comparison to the combinatorial pinch map given in~\cite[Definition 4.6]{CasalsNg} we do not keep track of the orientation of the disks $\Delta$ relative to $\Lambda_+$ or signs as we consider both $\Lambda_+$ and $\Lambda_-$ as unoriented links, and work over a field of characteristic 2.

\begin{exmp}\label{ex:pinch} Consider the Lagrangian projection of a Legendrian tangle shown in Figure~\ref{fig:pinching ex} with four crossings corresponding to Reeb chords $a_1, a_2,a_3,$ and $a_4$ from left to right. If we pinch at the Reeb chords $a_2$ and $a_3$ in order, then the combinatorial dg-algebra map is the identity for any Reeb chord of $\Lambda$ not shown in this tangle and on the Reeb chords $a_1, \ldots, a_4$ it is given by:
\begin{align*}
\Phi_{L_{a_3}}\circ \Phi_{L_{a_2}}(a_1)&=\Phi_{L_{a_3}}(a_1+s_1^{-1})=a_1+s_1^{-1}+s_1^{-2}s_2^{-1},\\
\Phi_{L_{a_3}}\circ \Phi_{L_{a_2}}(a_2)&=\Phi_{L_{a_3}}(s_1)=s_1,\\
\Phi_{L_{a_3}}\circ\Phi_{L_{a_2}}(a_3)&=\Phi_{L_{a_3}}(a_3+s_1^{-1})=s_2+s_1^{-1},~\text{and}\\
\Phi_{L_{a_3}}\circ\Phi_{L_{a_2}}(a_4)&=\Phi_{L_{a_3}}(a_4)=a_4+s_2^{-1}.
\end{align*}
The disk $\Delta\in \Delta_{a_3}(a_1)$ that contributes the term $s_1^{-2}s_2^{-1}$ has $w_1(\Delta)=w_2(\Delta)=s_1^{-1}.$ 
\end{exmp}

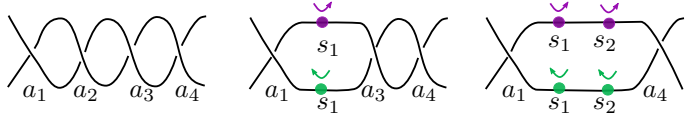
\begin{figure}[H]
	\centering
	\begin{tikzpicture}[scale=1.5]
		\node[inner sep=0] at (0,0) {\includegraphics[width=9 cm]{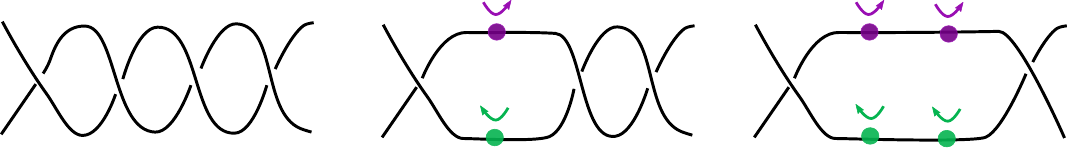}};
		\node at (-2.75,-0.4){$a_1$};
  	\node at (-2.3,-0.4){$a_2$};
   	\node at (-1.8,-0.4){$a_3$};
    	\node at (-1.4,-0.4){$a_4$};
        \node at (-0.6,-0.4){$a_1$};
  	\node at (-0.15,-0.5){$s_1$};
        \node at (-0.15,0){$s_1$};
   	\node at (0.25,-0.4){$a_3$};
    	\node at (0.7,-0.4){$a_4$};
             \node at (1.5,-0.4){$a_1$};
  	\node at (1.9,-0.5){$s_1$};
        \node at (1.9,0.05){$s_1$};
   	\node at (2.3,-0.5){$s_2$};
        \node at (2.3,0.05){$s_2$};
    	\node at (2.75,-0.4){$a_4$};
  \end{tikzpicture}
	\caption{A pinching sequence: pinch $a_2$ and then $a_3.$} 
	\label{fig:pinching ex}
\end{figure}

\section{\texorpdfstring{The Ungraded Augmentation Variety $\Aug_u(\La[n_1, \dots, n_k])$}{}}

In this section, we compute the Legendrian dg-algebra $\SA(\La[n_1,\dots, n_k])$ of the max-tb Legendrian $2$-bridge link $\La[n_1, \dots, n_k]$ in Legendrian rational form and use these computations to describe its ungraded augmentation variety $\Aug_u(\La[n_1, \dots, n_k])$. We first compute the differential of $\SA(\La[n_1, \dots, n_k])$ via a family of recursively defined polynomials known as continuants. We then describe explicit continuant polynomials that appear in the equations cutting out $\Aug_u(\La[n_1, \dots, n_k])$. Finally, we compute all possible dg-algebra homotopies of the ungraded augmentations of $\mathcal{A}(\La[n_1, \dots, n_k])$. We begin by defining continuants, a family of polynomials that greatly simplify our computations. The appearance of continuants in the computation of the Legendrian dg-algebra of $2$-bridge links below can be thought of as a reflection of the $A_n$-type cluster algebras we produce. Similar matricial computations also appear in defining the augmentation variety of positive braid closures; see for example~\cite{Kalman-braid, Gao_shen_weng_2, CGGS1}. We end by showing in Corollary~\ref{cor: topological invariant} that the ungraded augmentation variety of $2$-bridge links is a topological invariant.

\subsection{Continuants and Legendrian rational forms}\label{sub: continuants}Smoothly, a $2$-bridge link is a nontrivial link that admits a projection with exactly four vertical tangencies, two minima, and two maxima. They are known to be classified by equivalence classes of rational numbers. We begin by reviewing continuants, a family of recursively defined polynomials that were introduced by Euler in his study of continued fractions. We then review some basic properties of $2$-bridge links, and their relation to continuants which will allow us to prove Corollary~\ref{cor: topological invariant}. We end by defining the Legendrian representatives of $2$-bridge links that we consider in this article. In this subsection, we temporarily allow \textbf{any characteristic} for the coefficient field and review some basic properties of continuants. In the following section, we will return to characteristic $2$ to compute the differentials for Legendrian dg-algebras for Legendrian $2$-bridge links.  

\begin{defn}\label{def of K_n} \emph{Continuants} are polynomials defined recursively by
\[
 K_n(x_1, \dots x_n)=x_1K_{n-1}(x_2, \dots x_n)- K_{n-2}(x_3, \dots x_n) \quad \text{with} \quad K_0()=1,~\text{and}~K_1(x_1)=x_1.
\]
\end{defn}

\begin{rmk} The formulation of continuants differs from an alternative formulation in some literature which has a $+$ sign instead of a $-$ sign in front of $K_{n-2}(x_3,\dots, x_n)$ (e.g. \cite[Section 1]{Frame49}).
\end{rmk}

\begin{lem}\label{lem: braid matrix prod}Let $B(x)=\begin{pmatrix}
  x & -1\\ 1 & 0  
\end{pmatrix}$. Then 
\[
\begin{pmatrix}K_{n}(x_1, \dots, x_{n})& -K_{n-1}(x_1, \dots, x_{n-1})\\
K_{n-1}(x_2, \dots x_{n}) & -K_{n-2}(x_2, \dots, x_{n-1}) \end{pmatrix}=B(x_1)\dots B(x_n).
\]
\end{lem}
\begin{proof}This follows inductively from applying the recursion relation in computing the matrix product
\[
\begin{pmatrix} x_1 & -1 \\ 1 & 0\end{pmatrix}\begin{pmatrix}K_{n-1}(x_2, \dots, x_{n})& -K_{n-2}(x_2, \dots, x_{n-1})\\
K_{n-2}(x_3, \dots x_{n}) & -K_{n-3}(x_3, \dots, x_{n-1}) \end{pmatrix}=\begin{pmatrix}K_{n}(x_1, \dots, x_{n})& -K_{n-1}(x_1, \dots, x_{n-1})\\
K_{n-1}(x_2, \dots x_{n}) & -K_{n-2}(x_2, \dots, x_{n-1}) \end{pmatrix}.\qedhere\]
\end{proof}

\begin{lem}\label{lem: alternative recursion} The continuants also satisfy the following recursion relation:
\[
K_n(x_1,\dots, x_n)=K_{n-1}(x_1,\dots, x_{n-1})x_n-K_{n-2}(x_1,\dots, x_{n-2}).
\]
\end{lem}
\begin{proof} By Lemma \ref{lem: braid matrix prod}, we know that 
\begin{align*}
\begin{pmatrix}K_{n}(x_1, \dots, x_{n})& -K_{n-1}(x_1, \dots, x_{n-1})\\
K_{n-1}(x_2, \dots x_{n}) & -K_{n-2}(x_2, \dots, x_{n-1}) \end{pmatrix}=&B(x_1)\dots B(x_n)\\
=& \begin{pmatrix}K_{n-1}(x_1, \dots, x_{n-1})& -K_{n-2}(x_1, \dots, x_{n-2})\\
K_{n-2}(x_2, \dots x_{n-1}) & -K_{n-3}(x_2, \dots, x_{n-2}) \end{pmatrix}\begin{pmatrix} x_n & -1 \\ 1 & 0 \end{pmatrix}.
\end{align*}
The upper left corner entry of the above matrix equation proves the claim.
\end{proof}

\begin{lem}\label{lem: palindromic} The continuants $K_n$ are symmetric under the action of $x_i\mapsto x_{n-i+1}$, i.e., $K_n(x_1,\dots, x_n)=K_n(x_n,\dots, x_1)$.
\end{lem}
\begin{proof} We can do an induction on $n$. The claim is trivial for the base case $n=0,1$. Inductively, by Lemma \ref{lem: alternative recursion}, we have
\begin{align*}
K_n(x_1,\dots, x_n)=&K_{n-1}(x_1,\dots, x_{n-1})x_n-K_{n-2}(x_1,\dots, x_{n-2})\\
=& x_nK_{n-1}(x_{n-1},\dots, x_1)-K_{n-2}(x_{n-2},\dots, x_1)\\
=&K_n(x_n,\dots, x_1). \qedhere
\end{align*}
\end{proof}

\begin{lem}\label{lem: Braid Matrix Det} The continuants satisfy $K_{n-1}(x_1,\dots, x_{n-1})K_{n-1}(x_2,\dots, x_n)-K_n(x_1,\dots, x_n)K_{n-2}(x_2,\dots, x_{n-1})=1$.
 Consequently, if $x_1,\dots, x_n\in \F$ satisfy
\begin{enumerate}
    \item $K_n(x_1,\dots, x_n)=0$, then $K_{n-1}(x_1,\dots, x_{n-1})\neq 0$ and $K_{n-1}(x_2,\dots, x_n)\neq 0$;
    \item $K_{n-1}(x_1, \dots, x_{n-1})=0$ or $K_{n-1}(x_2, \dots, x_{n})=0$, then $K_{n-2}(x_2, \dots, x_{n-1})\neq 0$ and $K_{n}(x_1, \dots, x_n)\neq 0$.
\end{enumerate}
\end{lem}
\begin{proof} Note that $\det B(x_i)=1$ for all $i$. Therefore the product $B(x_1)\cdots B(x_n)$ also has determinant $1$. According to Lemma \ref{lem: braid matrix prod}, this determinant is precisely 
\[
K_{n-1}(x_1,\dots, x_{n-1})K_{n-1}(x_2,\dots, x_n)-K_n(x_1,\dots, x_n)K_{n-2}(x_2,\dots, x_{n-1}).\qedhere\]
\end{proof}

\begin{cor}\label{cor: relatively prime} For any $x_1,\dots, x_n\in \Z$, 
\[
\gcd(K_n(x_1,\dots, x_n),K_{n-1}(x_1,\dots, x_{n-1}))=\gcd(K_n(x_1,\dots, x_n),K_{n-1}(x_2,\dots, x_{n}))=1.
\]
\end{cor}
\begin{proof} Note that Lemma \ref{lem: Braid Matrix Det} gives a way to write $1$ as a linear combination of continuants.
\end{proof}

We now review some basic properties of $2$-bridge links and the Legendrian representatives we consider in this article.

\begin{thm}[\cite{Schubert,Conway}, see also {\cite[Theorem 9.3.3]{Murasugi}}]\label{thm: classification of rational links} Each $2$-bridge link can be represented by a rational number $p/q$ with $p>0$ and $\gcd(p,q)=1$. Two such rational numbers $p/q$ and $p'/q'$ represent the same unoriented $2$-bridge link if and only if $p=p'$ and $q\equiv (q')^{\pm 1} \mod p$.
\end{thm}

To specify a projection of a $2$-bridge link associated with a rational number, we introduce the following notation of continued fractions (assuming $n_i>0$ for all $i$):

\begin{equation}\label{eq: cont fraction}
    [n_1, \dots, n_k]:=n_1-\cfrac{1}{n_2-\cfrac{1}{\ddots-\cfrac{\ddots}{n_{(k-1)}- \cfrac{1}{n_k}}}}.
\end{equation}

The $2$-bridge link associated with the rational number $p/q=[n_1,\dots, n_k]$ is then the underlying smooth link of the Legendrian link depicted in Figure \ref{fig:front}\footnote{We can temporarily drop the $n_i\geq 2$ for all $1<i<k$ condition when drawing smooth $2$-bridge links; however, we will need this condition for our discussion on Legendrian $2$-bridge links in the rest of the article.}.

\begin{lem}[{\cite[Lemma 3]{Ng01}}] Any $2$-bridge link can be represented by $[n_1,\dots,n_k]$ with $n_i>0$ for all $i$.
\end{lem}
\begin{proof} By adding sufficiently many multiples of $p$ to $q$, we can assume without loss of generality that $p/q>0$. Then the numbers $n_1,\dots, n_k$ can be uniquely determined by an alternating Euclidean division algorithm.
\end{proof}

Our next goal is to investigate how Theorem \ref{thm: classification of rational links} is reflected on the continued fraction representatives $[n_1,\dots, n_k]$. To do so, let us first relate continued fractions with continuants.

\begin{lem}\label{lem: continuant identity} For $n_i>0$ we have that \begin{enumerate} 
\item $[n_1, \dots, n_k]=\dfrac{K_k(n_1, \dots, n_k)}{K_{k-1}(n_2, \dots, n_k)}$, and
\item $[n_k, \dots, n_1]=\dfrac{K_k(n_1, \dots, n_k)}{K_{k-1}(n_1, \dots, n_{k-1})}$.
\end{enumerate}
\end{lem}
\begin{proof} Note that (2) follows from (1) and Lemma \ref{lem: palindromic}. We will prove (1) by an induction on $k$. For the base case $k=1$, we have $[n_1]=\frac{n_1}{1}$. Inductively, we have
\begin{align*}
    [n_1,\dots, n_k]=& n_1-\frac{1}{[n_2,\dots, n_{k}]}=n_1-\frac{K_{k-2}(n_3,\dots,n_k)}{K_{k-1}(n_2,\dots, n_k)}\\
    =&\frac{n_1K_{k-1}(n_2,\dots, n_k)-K_{k-2}(n_3,\dots, n_k)}{K_{k-1}(n_2,\dots, n_k)}=\frac{K_k(n_1, \dots, n_k)}{K_{k-1}(n_2, \dots, n_k)}. \qedhere
\end{align*}
\end{proof}

\begin{prop}\label{prop: isotopic rational links} Two $2$-bridge links represented by $[n_1,\dots, n_k]$ and $[n'_1,\dots, n'_{k'}]$ are isotopic if and only if $[n'_1,\dots, n'_{k'}]$ can be obtained from $[n_1,\dots, n_k]$ via a sequence of the following moves:
\begin{enumerate}
    \item $[n_1,\dots, n_{k-1},n_k]\sim [n_1,\dots, n_{k-1},n_k+1,1]$;
    \item $[n_1,n_2,\dots, n_k]\sim [1,n_1+1,n_2,\dots, n_k]$;
    \item $[n_1,n_2,\dots, n_k]\sim [n_k,\dots, n_2,n_1]$.
\end{enumerate}
\end{prop}
\begin{proof} For the backwards direction, it suffices to show that the above three types of moves all satisfy the isotopy conditions stated in Theorem \ref{thm: classification of rational links}. Note that by Corollary \ref{cor: relatively prime}, the right hand sides of the identities in Lemma \ref{lem: continuant identity} are automatically reduced.
Move $(1)$ gives the same rational number $[n_1,\dots, n_{k-1},n_k]= [n_1,\dots, n_{k-1},n_k+1,1]$. For Move $(2)$, by Lemma \ref{lem: continuant identity}, we have
\begin{align*}
[n_1,n_2,\dots, n_k]=&\frac{K_k(n_1,\dots, n_k)}{K_{k-1}(n_2,\dots, n_k)},~\text{while}\\
[1,n_1+1,n_2,\dots, n_k]=&1-\cfrac{1}{n_1+1-\cfrac{1}{[n_2,\dots, n_k]}}=\cfrac{n_1-\cfrac{1}{[n_2,\dots, n_k]}}{n_1+1-\frac{1}{[n_2,\dots, n_k]}}=\cfrac{\cfrac{K_k(n_1,\dots, n_k)}{K_{k-1}(n_2,\dots, n_k)}}{\cfrac{K_k(n_1,\dots, n_k)}{K_{k-1}(n_2,\dots, n_k)}+1}\\
=&\frac{K_k(n_1,\dots, n_k)}{K_k(n_1,\dots, n_k)+K_{k-1}(n_2,\dots, n_k)}.
\end{align*}
Note that obviously we have $K_k(n_1,\dots, n_k)+K_{k-1}(n_2,\dots, n_k)\equiv K_{k-1}(n_2,\dots, n_k) \mod K_k(n_1,\dots, n_k)$. Lastly, for the third type, by Lemmas \ref{lem: Braid Matrix Det} and \ref{lem: continuant identity}, these two continued fractions have the same numerator, and their denominators satisfy
\[
K_{k-1}(n_1,\dots, n_{k-1})K_{k-1}(n_2,\dots, n_k) \equiv 1 \mod K_k(n_1,\dots, n_k).
\]
Thus, by Theorem \ref{thm: classification of rational links}, these three types of moves all induce isotopies on $2$-bridge links.

Conversely, we note that the second type of move allows us to change between $p/q$ and $p/(q+p)$. This takes care of the condition $q\equiv q'\mod p$. On the other hand, we already observed that $p/q=[n_1,\dots, n_k]$ and $p/q'=[n_k,\dots, n_1]$ satisfy $qq'\equiv 1\mod p$. Since any other $q''$ that satisfies $qq''\equiv 1 \mod p$ must also satisfy $q''\equiv q \mod p$, it follows that we can first use the second type of move to change $q''$ to $q'$, and then use the third type to change it to $q$.
\end{proof}

Let us now describe the Legendrian representatives of $2$-bridge links that we consider in this work. 

\begin{defn}\label{def: Legendrian rational form} Following the terminology of~\cite{Ng01}, when\footnote{Note that this definition differs slightly from the definition given in~\cite{Ng01} in that we allow $n_1=1$ and $n_k=1$.}
\begin{itemize}
    \item $n_1 \geq 1$,
    \item $n_i\geq 2$ for all $1<i <k$, and
    \item $n_k\geq 1$,
\end{itemize}
we say that the link with front projection given in Figure \ref{fig:front} is in \emph{Legendrian rational form} and denote it by $\La[n_1, \dots, n_k]$.
\end{defn}
The conditions above are necessary for exact Lagrangian fillability, as for any $1<i<n$, allowing $n_i=1$ yields a stabilized Legendrian link. Unless otherwise indicated, all the Legendrian $2$-bridge knots $\Lambda[n_1, \ldots, n_k]$ we consider are in Legendrian rational form. By applying Ng's resolution to the front projections in Figure \ref{fig:front}, we obtain the following Lagrangian projections of Legendrian $2$-bridge links $\Lambda[n_1,\dots, n_k]$. Furthermore, we decorate $\Lambda[n_1,\dots, n_k]$ with co-oriented base points $t_1$ and $t_2$ as given in Figure \ref{fig:2bridge}.

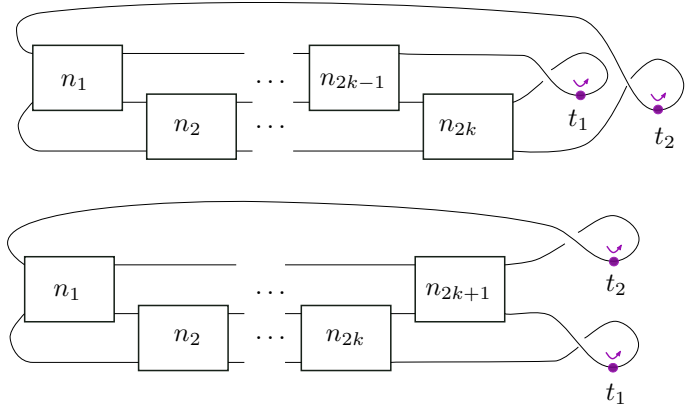
\begin{figure}[!htb]
	\centering
	\begin{tikzpicture}[scale=1]
		\node[inner sep=0] at (0,0)
		{\includegraphics[width=9cm]{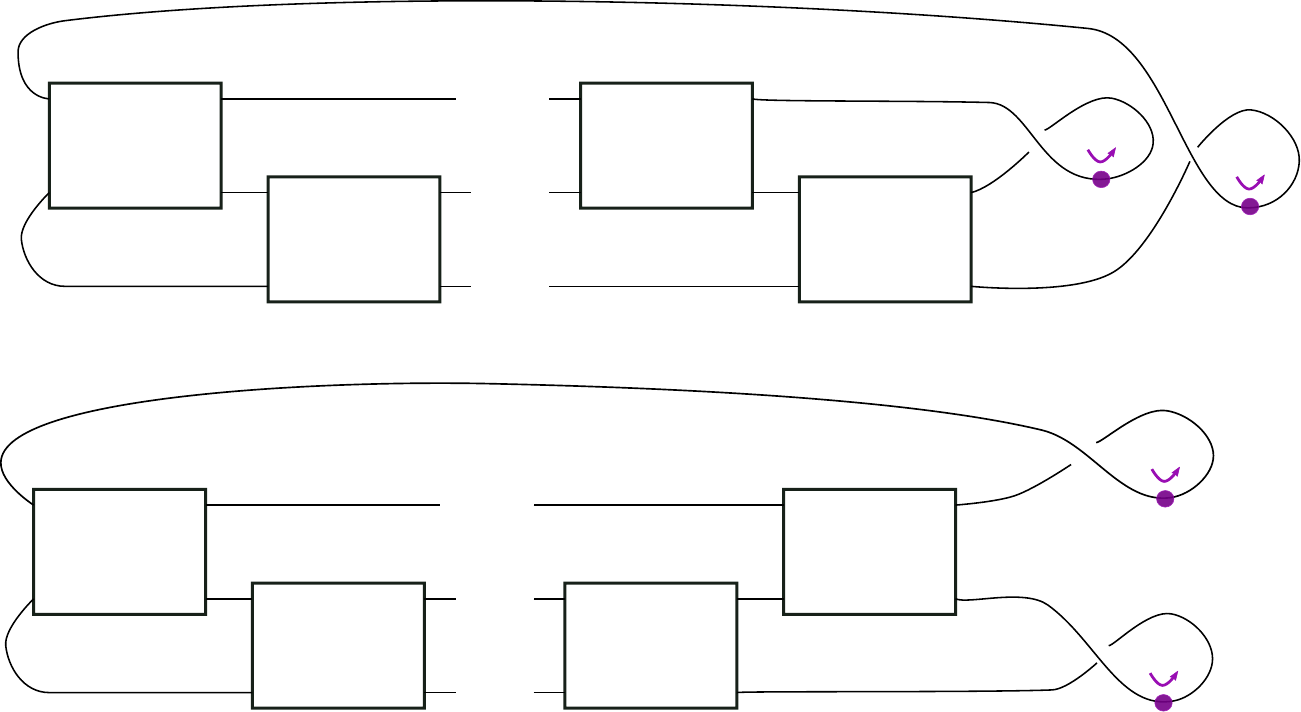}};
	\node at (-3.6,1.4){$n_1$};
	\node at (-2.1,0.8){$n_2$};
	\node at (-1,1.4){$\ldots$};
	\node at (-1,0.8){$\ldots$};
	\node at (0.1,1.4){$n_{2k-1}$};
	\node at (1.5,0.8){$n_{2k}$};
	\node at (-1,-1.4){$\ldots$};
	\node at (-1,-2){$\ldots$};
	\node at (-3.7,-1.4){$n_1$};
	\node at (-2.1,-2){$n_2$};
	\node at (0,-2){$n_{2k}$};
	\node at (1.5,-1.4){$n_{2k+1}$};
    \node at (3.1,0.9){$t_1$};
    \node at (4.25,0.65){$t_2$};
    \node at (3.6,-1.3){$t_2$};
    \node at (3.6,-2.75){$t_1$};
	\end{tikzpicture}\caption{Lagrangian projections of $2$-bridge knots $\Lambda[n_1, \ldots, n_{2k}]$ and $\Lambda[n_1, \ldots, n_{2k+1}]$, where each box labeled $n_i$ contains $n_i$ crossings where the strand with more negative slope is the overstrand. Both co-oriented base points $t_1$ and $t_2$ are oriented counterclockwise.}\label{fig:2bridge}
	\end{figure}

\begin{rmk} By a result of Ng \cite[Theorem 1]{Ng01}, any Legendrian link in Legendrian rational form is a max-tb representative of its respective smooth isotopy class. On the other hand, it is not hard to see that the rotation number of a link in Legendrian rational form is always $0$ or $\pm 1$.
\end{rmk}

\subsection{Computing the differential of $\SA(\La[n_1, \dots, n_k]; R)$}

In this subsection, we return to \textbf{characteristic $2$} and inductively compute the Legendrian dg-algebra of $\La[n_1, \dots, n_k]$ using specific matrices to keep track of the partial holomorphic disks involved in the differential. We first introduce some necessary notation: After applying Ng's resolution, we decorate $\La[n_1, \dots, n_k]$ with two co-oriented base points $t_1$ and $t_2$ as in Figure \ref{fig:2bridge}.  
We label the crossings originally appearing in the front projection from left to right by $a_i$. We label the two new crossings appearing in the resolution by $b_1$ and $b_2$ so that the indices agree with those of the base points $t_1$ and $t_2$. We label the strands from bottom to top with $1,2,3,4$. 

Throughout this article, we denote by $m_i$ the total number of crossings in the first $i$ blocks of $\La[n_1, \dots, n_k]$: 
\[
m_i:=\sum_{j=1}^i n_j.
\]
For $a_{m_i+1}$ and $b_1$, we will express the differential in terms of products of continuants. We condense our notation from Subsection \ref{sub: continuants} to streamline our computations. Denote by \begin{align*}
    K_{n_i}&=K_{n_i}(a_{m_{(i-1)}+1}, \dots, a_{m_{i}}) &
    K_{n_{i}-2}^M&=K_{n_i-2}(a_{m_{(i-1)}+2}, \dots, a_{m_{i}-1})\\
    K_{n_{i}-1}^L&=K_{n_i-1}(a_{m_{(i-1)}+1}, \dots, a_{m_{i}-1}) &
    K_{n_{i}-1}^R&=K_{n_i-1}(a_{m_{(i-1)}+2}, \dots, a_{m_{i}})
\end{align*}
where the $L$ and $R$ superscripts can be understood as denoting the leftmost or rightmost $n_{i}-1$ crossings of the $i$th block, respectively. Similarly, we use $M$ to denote the middle $n_{i}-2$ crossings of the $i$th block. 

\begin{figure}[H]
	\centering
	\begin{tikzpicture}[scale=1]
		\node[inner sep=0] at (0,0) {\includegraphics[width=15 cm]{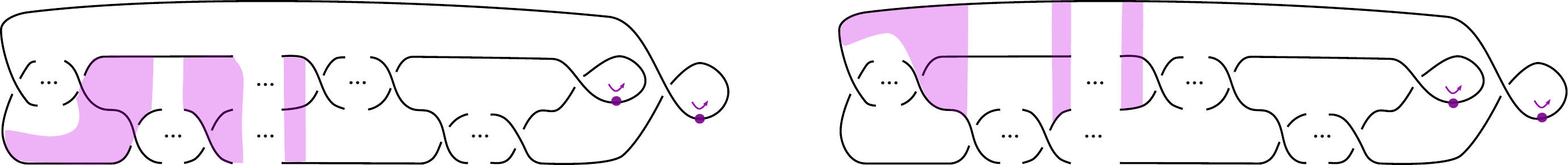}};
  	\node at (-7.2,-0.4){$a_1$};
        \node at (-6.6,-0.4){$a_{n_1}$};
        \node at (-6.2,-1){$a_{n_1+1}$};
        \node at (-5.3,-1){$a_{m_2}$};
        \node at (-3.8,-0.5){$a_{m_{k-1}}$};
        \node at (-3.2,-1){$a_{m_{k-1}+1}$};
        \node at (-2.2,-1){$a_{m_k}$};
        \node at (-1.0,-0.4){$t_2$};
        \node at (-1.7,-0.4){$t_1$};
        
        \node at (0.8,-0.4){$a_1$};
        \node at (1.4,-0.4){$a_{n_1}$};
        \node at (2,-1){$a_{n_1+1}$};
        \node at (2.8,-1){$a_{m_2}$};
        \node at (4.3, -0.5){$a_{m_{k-1}}$};
        \node at (4.8,-1){$a_{m_{k-1}+1}$};
        \node at (5.8,-1.0){$a_{m_k}$};
        \node at (7,-0.4){$t_2$};
        \node at (6.35,-0.4){$t_1$};
  \end{tikzpicture}
	\caption{Schematics of the disks with boundary  contributing to the polynomials $D_{13}(k-2)$ (left) and $D_{24}(k-2)$ (right) for $\Lambda[n_1, \ldots, n_{k}]$ with $k$ even. The boundary is depicted as the far right portion of both of the disks. } 
	\label{fig: partialdisks}
\end{figure}

The differential of $b_2$ is more complicated, as we must consider immersed disks when $k$ is even. 
To express $\partial(b_2)$, we will recursively define a collection of polynomials that count partial rigid holomorphic disks with negative punctures at Reeb chords. These disks start at the left of the Lagrangian projection $\Pi_{xy}(\La[n_1, \ldots, n_k])$ and have boundary (in addition to the usual punctures at Reeb chords) depicted as a vertical interval lying in a region between a pair of strands; see Figure \ref{fig: partialdisks} for schematic depictions of disks contributing to $D_{13}(k-2)$ and $D_{24}(k-2)$. For $i\in 2\Z$, we recursively define the polynomials $D_{ab}(i)$ as follows:  
\begin{align*}
    D_{13}(i)&=K_{n_i-2}^MK_{n_{(i-1)}-2}^MD_{13}(i-2)\\
    D_{14}(i)&=K_{n_i-1}^L\left(K_{n_{(i-1)}-1}^LD_{34}(i-2)+K_{n_{(i-1)}-2}^M D_{24}(i-2)\right)+K_{{n_i}-2}^M D_{14}(i-2)\\
    D_{23}(i)&=K_{{n_i}-1}^RK_{n_i-2}^M D_{13}(i-2)\\
    D_{24}(i)&=K_{n_{i}-1}^RD_{14}(i-2)+K_{n_i}\left(K_{n_{(i-1)}-1}^LD_{34}(i-2)+K_{n_{(i-1)}-2}^M D_{24}(i-2)\right)\\
    D_{34}(i)&=K_{n_{(i-1)}}D_{34}(i-2)+K_{n_{(i-1)}-1}^RD_{24}(i-2).
\end{align*}

For $i=2,$ we define 
\[
D_{13}(2)=K_{n_2-2}^MK_{n_{1}-1}^L, \quad \quad   D_{14}(2)=K_{n_2-1}^LK_{n_{1}-1}^L, \quad \quad
 D_{23}(2)=K_{{n_2}-1}^RK_{n_1-1}^L,
 \]
 \[
 \quad D_{24}(2)=K_{n_2}K_{n_{1}-1}^L,~\text{and}~ \quad \quad D_{34}(2)=K_{n_{1}}.
\]

\begin{lem}\label{lem: partial disk recursion}
   The polynomials $D_{ab}(i)$ for $(a,b)\in\{(1,3), (1,4), (2,3), (2,4), (3,4)\}$ and $1\leq i\leq k, i \in 2\Z$ recursively count partial rigid holomorphic disks including blocks 1 through $i$ without positive punctures and with boundary lying on 
   strands $a$ and $b$ of $\Pi_{xy}(\Lambda[n_1, \ldots, n_k])$.
\end{lem}

\begin{proof}
Work of~\cite[Theorem 4.4]{Kalman-braid}, (see also~\cite[Theorem 2.30]{CGGS1}) together with Lemma \ref{lem: braid matrix prod} implies that the continuant polynomials $K_{j}(a_i, \dots, a_{i+j})$ count the partial rigid holomorphic disks starting at $a_i$ and ending at $a_{i+j}$ where $a_{i}$ and $a_{i+j}$ lie in the same block. 
The base case of $D_{ab}(2)$ is straightforward to verify by inspection. Therefore, it remains to verify that the recursive relations defined above capture all possible partial rigid holomorphic disks appearing between the indicated strands. We write these disks as products of continuants by studying how disks behave at the first and last crossings of previous blocks. We say that a disk passes through an overstrand (resp. understrand) at a crossing if the disk covers two consecutive quadrants of the crossing and the overstrand (resp. understand) at the crossing is transverse to the boundary of the disk.\footnote{We assume that we have perturbed the Legendrian as necessary to ensure that all crossings are transverse.} Fix $i\in 2\Z$ and consider $D_{ab}(i)$:
\begin{itemize}
    \item[$D_{13}(i)$:] All partial disks with boundary lying between strands 1 and 3 must pass through the understrand of the rightmost crossing and the overstrand of the leftmost crossing 
of block $i$. Continuing on to block $(i-1)$, the disks must pass through the understrand of the leftmost crossing and the overstrand of the rightmost crossing   
in order to avoid any positive punctures. By~\cite[Theorem 4.4]{Kalman-braid}, the corresponding terms are $K_{n_i-2}^M$ and $K_{n_{(i-1)}-2}^M$, respectively. After passing through these two blocks, the boundary of the disk still lies on strands 1 and 3 immediately to the right of block $(i-2)$, establishing the recursion in this case.
    \item[$D_{14}(i)$:] All partial disks with boundary lying between strands 1 and 4 must pass through the understrand of the rightmost crossing of block $i$, labeled $a_{m_i}$. A disk also passes through the overstrand of the leftmost crossing $a_{m_{(i-1)}+1}$, yielding the term $K_{n_i-2}^MD_{14}(i-2)$ as any possible negative punctures from block $(i-1)$ lie between strands 2 and 3. In the case where the disk either has a negative puncture at the leftmost crossing of block $i$ or passes through the understrand, the disk must also pass through the understrand of the rightmost crossing of block $(i-1)$, labeled $a_{m_{(i-1)}}$. This yields the terms $K_{n_i-1}^LK_{n_{(i-1)}-2}^MD_{24}(i-2)$ and  $K_{n_i-1}^LK_{n_{(i-1)}-1}^LD_{34}(i-2)$ depending on whether or not the disk passes through the overstrand of the leftmost crossing $a_{m_{i-3}+1}$. Adding each of these terms together establishes the desired recursion.
     \item[$D_{23}(i)$:] All partial disks with boundary lying between strands 2 and 3 must pass through the overstrand of the leftmost crossing 
of block $i$, as well as the overstrand of the rightmost crossing and understrand of the leftmost crossing of block $(i-1)$ in order to avoid any positive punctures. The corresponding terms are $K_{n_i-1}^R$ and $K_{n_{(i-1)}-2}^M$, respectively. After passing through these crossings, the boundary of the disk lies on strands 1 and 3, establishing the desired recursion.
  \item[$D_{24}(i)$:] If a partial disk with boundary lying between strands 2 and 4 passes through the overstrand of the leftmost crossing of block $i$, then it cannot pick up any additional negative punctures from block $(i-1)$. This disk then lies between strands 1 and 4, which yields the term $K_{n_{i}-1}^RD_{14}(i-2)$. If the disk has a negative puncture or passes through the understrand at the leftmost crossing of block $i$, then it must pass through the understrand at the rightmost crossing of block $(i-1).$ Similar to the $D_{14}$, case, this yields the terms $K_{n_i}K_{n_{(i-1)}-2}^MD_{24}(i-2)$ and  $K_{n_i}K_{n_{(i-1)}-1}^LD_{34}(i-2)$ depending on whether or not the disk passes through the overstrand of $a_{m_{i-3}+1}$.
  \item[$D_{34}(i)$:] Any partial disk with boundary lying between strands 3 and 4 cannot have any negative punctures at crossings in block $i$. Therefore, the two terms that appear in the recursion are $K_{n_{(i-1)}-1}^RD_{24}(i-2)$ for disks passing through the overstrand of the leftmost crossing of block $(i-1)$ and $K_{n_{(i-1)}}D_{34}(i-2)$ otherwise.
\end{itemize}
\end{proof}

The differential of $\mathcal{A}_\La[n_1, \dots, n_k]$ is then described in terms of these recursively defined polynomials by the following proposition.

\begin{prop}
    \label{prop: differential}
    The differential $\partial_{\La[n_1, \dots n_k]}$ (for $k\geq 2)$ is given by 
\begin{align*}
    \partial(a_{m_{i}+j})&=\begin{cases} 
    K_{n_1} & i=j=1\\ 
     K_{n_1-1}^LK^M_{n_2-2}\dots K_{n_{(i-1)}-2}^MK_{n_i-1}^R & 2\leq i \leq k-1, j=1\\ 
       0 & 2\leq j\leq n_i \end {cases}\\
 \partial(b_1)&= K_{n_1-1}^LK_{n_2-2}^M\dots K_{n_{(k-1)}-2}^MK_{n_k-1}^R
 +t_1\\ 
    \partial(b_2)&=\begin{cases} K_{n_k}D_{34}(k-1)+K_{n_k-1}^RD_{24}(k-1) + t_2 & k \text{ odd}\\
D_{14}(k)+D_{34}(k)b_1t_1^{-1}D_{13}(k)+D_{24}(k)t_1^{-1}D_{13}(k) + t_2  
& k \text{ even}\end{cases}
\end{align*}
\end{prop}

\begin{proof}
  By inspection, there are no rigid holomorphic disks with a single puncture at $a_{m_i+j}$ for $2\leq j\leq n_{i+1}$, implying that this term vanishes under the differential.

    \textbf{Differential of $a_{m_i+1}$:} 
By Lemma \ref{lem: partial disk recursion}, we can see that all rigid holomorphic disks contributing to $\partial(a_{m_i+1})$ are counted by the disks $D_{23}(i)$ when $(i+1)$ is odd or $K_{n_{i}-1}^RD_{13}(i-1)$ when $i+1$ is even.\footnote{Note that $a_{m_i +1}$ belongs to the $(i+1)$ block.} Upon inspection, the recursion relation for each of these partial disks yields the desired formula.

    \textbf{Differential of $b_1$:}  
    The $t_1$ term in the formula given for $\partial(b_1)$ comes from the rigid holomorphic disk corresponding to the small loop to the right of $b_1$. The remaining disks involved in the differential of $b_1$ all pass through alternating pairs of overstrands and understrands starting with the crossings $a_{m_{k-1}+1}$ and $a_{m_{k-1}}$, and ending with the crossings $a_{n_1+1}$ and $a_{n_1}$. By our argument for the $a_{m_i+1}$ case, these disks are precisely the ones encoded by the product of continuant polynomials appearing in the statement of the proposition.

     \textbf{Differential of $b_2$:}
Consider the crossing labeled $b_2$ and suppose that $k$ is odd. As before, the $t_2$ term comes from the small loop to the right of $b_2$. Following the inductive definition of the $D_{ab}(i)$ polynomials, we can see that $K_{n_k}D_{34}(k-1)$ captures all partial rigid holomorphic disks that have either a negative puncture or pass through the understrand at the leftmost crossing $a_{n_{(k-1)}+1}$ of the $k$th block, while the term $K_{n_k-1}^RD_{24}(k-1)$ captures all such disks that pass through the overstrand at $a_{n_{(k-1)}+1}$. Therefore, the sum of these three terms yields the desired expression for $\partial(b_2)$.

Now consider the case when $k$ is even. As before, the $t_2$ term comes for the small loop to the right of $b_2$. The remaining rigid holomorphic disks contributing to the differential are either embedded or immersed. The embedded disks are given by the $D_{14}(k)$ term, while the immersed disks are given by the terms $D_{34}(k)b_1t_1^{-1}D_{13}(k)$ and $D_{24}(k)t_1^{-1}D_{13}(k)$. A schematic depiction of these disks is given in Figure \ref{fig: b2disks}. 
\end{proof}

\begin{figure}[H]
	\centering
	\begin{tikzpicture}[scale=1]
		\node[inner sep=0] at (0,0) {\includegraphics[width=15 cm]{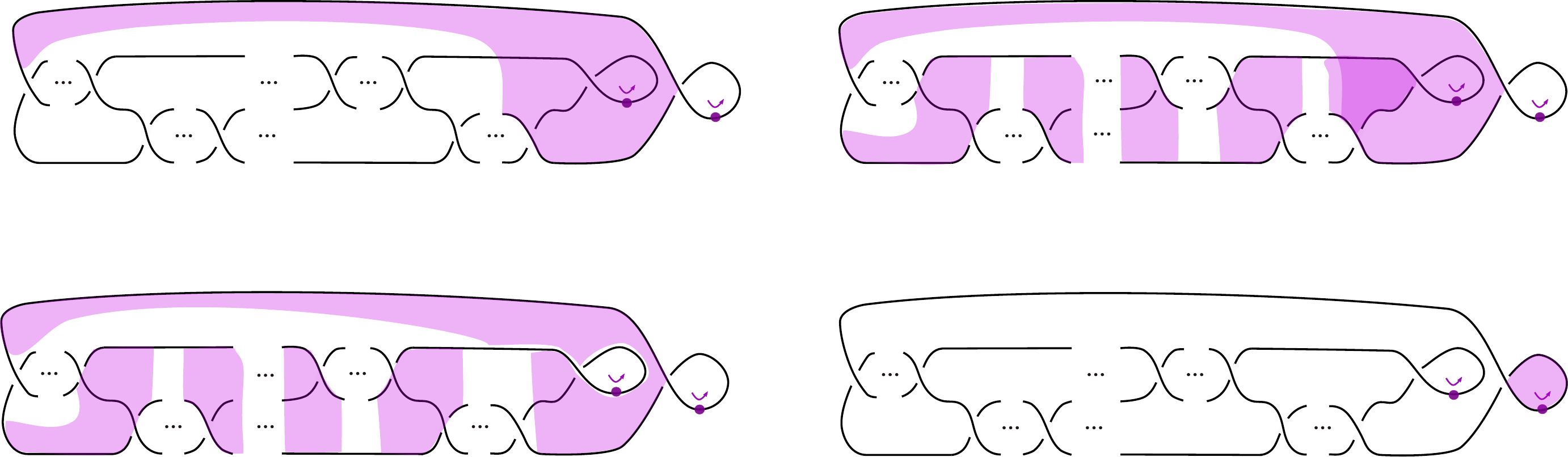}};
  	\node at (-7.1,0.9){$a_1$};
        \node at (-6.4,0.9){$a_{n_1}$};
        \node at (-6,0.4){$a_{n_1+1}$};
        \node at (-5.1,0.4){$a_{m_2}$};
        \node at (-3.6,0.9){$a_{m_{k-1}}$};
        \node at (-3,0.4){$a_{m_{k-1}+1}$};
        \node at (-2,0.4){$a_{m_k}$};
        \node at (-0.8,0.9){$t_2$};
        \node at (-1.7,0.9){$t_1$};
        \node at (-1,1.8){$b_2$};
        \node at (-7.2,-1.9){$a_1$};
        \node at (-6.5,-1.9){$a_{n_1}$};
        \node at (-6,-2.5){$a_{n_1+1}$};
        \node at (-5.1,-2.5){$a_{m_2}$};
        \node at (-3.7,-1.9){$a_{m_{k-1}}$};
        \node at (-3,-2.5){$a_{m_{k-1}+1}$};
        \node at (-2,-2.5){$a_{m_k}$};
        \node at (-1.7,-1.8){$t_1$};
        \node at (-0.9,-2){$t_2$};
        \node at (-1.9,-1){$b_1$};
        \node at (-1,-1){$b_2$};
        \node at (0.8,0.9){$a_1$};
        \node at (1.4,0.9){$a_{n_1}$};
        \node at (2,0.4){$a_{n_1+1}$};
        \node at (2.8,0.4){$a_{m_2}$};
        \node at (4.3,0.9){$a_{m_{k-1}}$};
        \node at (4.8,0.4){$a_{m_{k-1}+1}$};
        \node at (5.8,0.4){$a_{m_k}$};
        \node at (7.2,0.9){$t_2$};
        \node at (6.35,0.9){$t_1$};
        \node at (7,1.8){$b_2$};
        \node at (7.2,-2){$t_2$};
        \node at (6.35,-1.8){$t_1$};
  \end{tikzpicture}
	\caption{Schematics of the embedded and immersed disks contributing to the differential of $b_2$ for $\Lambda[n_1, \ldots, n_{k}]$ with $k$ even. The upper left picture shows the disks contributing the term $D_{14}(k)$, the upper right picture shows the disks contributing the term $D_{24}(k)t_1^{-1}D_{13}(k)$, the lower left picture shows the disks contributing the term $D_{34}(k)b_1t_1^{-1}D_{13}(k)$, and the lower right picture shows the disk contributing the term $t_2$. } 
	\label{fig: b2disks}
\end{figure}

\subsection{Equations describing the ungraded augmentation variety}\label{sub: equations}

We now describe a collection of equations cutting out $\Aug_u(\La[n_1, \dots, n_k])$ (which at first depend on our choice of Lagrangian projection but after considering dg-algebra homotopies determines a Legendrian invariant affine variety). We show that the equation $\partial (a_{m_i+1})=0$ can be given by the vanishing of a specific continuant polynomial, rather than the product of continuants given in Proposition \ref{prop: differential}. We prove a similar statement for the equation $\partial(b_1)=0$ and then show that the vanishing of the expression $\partial(b_2)$ is implied by the previous equations.

\begin{lem}\label{lem: Kni-1R=0}
 For all $i=1, \ldots, k-1$, the equations $\partial(a_{m_i+1})=0$ are simultaneously satisfied if and only if $K_{n_1}=0$, $K_{n_i-1}^R=0$ for all $2\leq i\leq k-1$. The additional equation $\partial(b_1)=0$ is satisfied if and only if $K_{n_k-1}^R\neq 0$ as well.
\end{lem}
\begin{proof}
    To prove the first part of Lemma \ref{lem: Kni-1R=0}, consider $\partial(a_{m_i+1})= K_{n_1-1}^LK_{n_2-2}^M\dots K_{n_{(i-1)}-2}^MK_{n_{i}-1}^R $. By Lemma \ref{lem: Braid Matrix Det}, the vanishing of $\partial(a_{{n_1}+1})=K_{n_1}$ implies that $K_{n_1-1}^L$ and $K_{n_1-1}^R$ are both nonzero. Similarly, the vanishing of $K_{{n_i}-1}^R$ implies that $K_{n_i}$ and $K_{n_i-2}^M$ are both nonzero. It follows that every factor in $\partial(a_{m_i+1})=K_{n_1-1}^LK^M_{n_2-2}\dots K_{n_{(i-1)}-2}^M K_{n_i-1}^R$ except for the final factor $K_{n_i-1}^R$ is nonzero. Thus, the vanishing of $\partial(a_{m_i+1})$ is equivalent to the vanishing of $K_{n_i-1}^R$. The non-vanishing of $t_1$ together with the above argument immediately implies that $K_{n_k-1}^R\neq 0$ if and only if $\partial(b_1)=0$ in addition to $\partial(a_{m_i}+1)=0$.
\end{proof}

In the following subsection, we show that Lemma \ref{lem: Kni-1R=0} gives the equations and inequalities needed to describe $\Aug_u(\La[n_1, \dots, n_k])$ as an affine variety and that the equation $\partial(b_2)=0$ is redundant. To verify this, it suffices to show that the equations and the inequality in Lemma \ref{lem: Kni-1R=0} imply that $\partial(b_2)+t_2$ is always non-zero. 
We require the following lemma in our proof of this fact.

\begin{lem}\label{lem: Aug Vty Recursion}
    As regular functions on the ungraded augmentation variety $\Aug_u(\La[n_1, \dots, n_k])$, the polynomials $D_{ab}(i)$ satisfy
    \begin{itemize}
        \item $D_{34}(i)=0$ for all $2\leq i\leq k,~i\in 2\Z,$
        \item $D_{13}(i)\neq 0,~\text{and}~ D_{24}(i)\neq 0$ for all $2\leq i \leq k-2,~i\in 2\Z.$  
    \end{itemize}
\end{lem}

\begin{proof}
    Expand the expression for $D_{34}(i)$ via the defining recursion to get \begin{equation*}
        D_{34}(i)= K_{n_{(i-1)}}D_{34}(i-2)+K_{n_{(i-1)}-1}^RD_{24}(i-2).
    \end{equation*}

By Lemma \ref{lem: Kni-1R=0}, $K_{n_{(i-1)}-1}^R=0$ so the second term in this expansion vanishes. Moreover, $K_{n_1}=0$, inductively yielding the vanishing of the $K_{n_{(i-1)}}D_{34}(i-2)$ term. Thus $D_{34}(i)=0$ for all $2\leq i \leq k$.

     For $D_{24}(i)$, we expand via the recursion and simplify using the fact that $D_{34}(i-2)=0$ to get
     \begin{equation*}
         D_{24}(i)=K_{n_{i}-1}^RD_{14}(i-2)+K_{n_i}K_{n_{(i-1)}-2}^M D_{24}(i-2).
     \end{equation*}
   By Lemma \ref{lem: Kni-1R=0}, we can eliminate the $K^R_{n_{i}-1}$ term and are left with $D_{24}(i)=K_{n_i}K_{n_{(i-1)}-2}^M D_{24}(i-2)$. Since $D_{24}(2)=K_{n_2}K_{n_{1}-1}^L\neq 0$, we conclude that $D_{24}(i)\neq 0$ inductively.  

     Finally, the non-vanishing of $D_{13}(i)$ follows from an analogous inductive argument. To verify the base case, observe that $D_{13}(2)\neq 0$ follows from the fact that $K_{n_1}=0$ and $K_{n_2-1}^R=0$ by Lemma~\ref{lem: Kni-1R=0}. Applying Lemma~\ref{lem: Braid Matrix Det}, then yields that $K_{n_1-1}^L\neq 0$ and $K^M_{n_2-2}\neq 0$, so that their product $K_{n_1-1}^LK^M_{n_2-2}=D_{13}(2)\neq 0$ as well. 
\end{proof}

\begin{prop}\label{prop: partial b_2 redundant}
Suppose that $\partial(b_1)=\partial(a_{m_1+1})=\cdots=\partial(a_{m_{k-1}+1})=0$. Then $\partial(b_2)+t_2\neq 0$.
\end{prop}
\begin{proof}
Assume that $\partial(a_{m_i+1})=\partial(b_1)=0$ for $1\leq i\leq k-1$. First, consider the case that $k$ is odd. 
By Proposition \ref{prop: differential} and Lemma \ref{lem: Aug Vty Recursion}, the vanishing of $\partial(b_2)$ is equivalent to the vanishing of $K_{n_k-1}^RD_{24}(k-1)+t_2$. Moreover, by Lemma \ref{lem: Kni-1R=0}, we know that $\partial(b_1)=0$ implies that $K_{n_{k}-1}^R\neq 0$. Finally, Lemma \ref{lem: Aug Vty Recursion} implies that $D_{24}(k-1)$ is necessarily nonzero, so that the product $K_{n_{k}-1}^RD_{24}(k-1)$ is nonzero as well.

In the case that $k$ is even, we again apply Lemma \ref{lem: Aug Vty Recursion} to conclude that the term $D_{34}(k)b_1t_1^{-1}D_{13}(k)$ of $\partial(b_2)$ vanishes. 
All that remains to be shown is that $D_{14}(k)+D_{24}(k)t_1^{-1}D_{13}(k)\neq 0$. First observe that by Proposition \ref{prop: differential} and Lemma \ref{lem: partial disk recursion}, the equation $\partial(b_1)=0$ is equivalent to 
\begin{equation*}
   t_1= K_{n_k-1}^R K_{n_{(k-1)}-2}^MD_{13}(k-2).
\end{equation*}
After also expanding $D_{13}(k)$ by the recursion of Lemma~\ref{lem: partial disk recursion}, we have that

\begin{align*}t_1^{-1}D_{13}(k)&= \frac{K_{n_k-2}^M K_{n_{(k-1)}-2}^M D_{13}(k-2)}{K_{n_k-1}^R K_{n_{(k-1)}-2}^M D_{13}(k-2)}\\
&=\frac{K_{n_k-2}^M}{K_{n_k-1}^R}.\end{align*}
We return to $D_{14}(k)+D_{24}(k)t_1^{-1}D_{13}(k)=D_{14}(k)+D_{24}(k)\frac{K_{n_k-2}^M}{K_{n_k-1}^R}$, apply the recursions from  Lemma~\ref{lem: partial disk recursion} to $D_{14}(k)$ and $D_{24}(k)$ and use the fact that $D_{34}(i)=0$ for all $i\in 2\Z$ by Lemma~\ref{lem: Aug Vty Recursion} to obtain the following equivalences:
\begin{align*}
 D_{14}(k)+D_{24}(k)\frac{K_{n_k-2}^M}{K_{n_k-1}^R}
    &=K_{n_k-1}^LK_{n_{(k-1)}-2}^MD_{24}(k-2)+K^M_{n_k-2}D_{14}(k-2)\\&+(K_{n_k-1}^RD_{14}(k-2)+K_{n_k}K_{n_{(k-1)}-2}^MD_{24}(k-2))\frac{K_{n_k-2}^M}{K_{n_k-1}^R}\\
    &=K_{n_{(k-1)}-2}^MD_{24}(k-2)[K_{n_k-1}^L+\frac{K_{n_k}K_{n_k-2}^M}{K_{n_k-1}^R}].\\
\end{align*}
 We then apply the identity $K_{n_k}K_{n_k-2}^M=K_{n_k-1}^RK_{n_k-1}^L+1$ to the last term and obtain

\begin{align*}
    K_{n_{(k-1)}-2}^MD_{24}(k-2)[K_{n_k-1}^L+\frac{K_{n_k}K_{n_k-2}^M}{K_{n_k-1}^R}]&= K_{n_{(k-1)}-2}^MD_{24}(k-2)[K_{n_k-1}^L+\frac{K_{n_k-1}^RK_{n_k-1}^L+1}{K_{n_k-1}^R}]
    \\
    &=\frac{K_{n_{(k-1)}-2}^M}{K_{n_k-1}^R}D_{24}(k-2).
\end{align*}
Observe that $D_{24}(k-2)\neq 0$ by Lemma~\ref{lem: Aug Vty Recursion}, and that by Lemma~\ref{lem: Kni-1R=0}, $K_{n_{(k-1)}-1}^R=0$ and $K_{n_k-1}^R\neq 0$. Finally, $K_{n_{(k-1)}-1}^R=0$ implies that $K_{n_{(k-1)}-2}^M\neq 0$ by Lemma~\ref{lem: Braid Matrix Det}. Thus, we can conclude that $D_{14}(k)+D_{24}(k)t_1^{-1}D_{13}(k)=\frac{K_{n_{(k-1)}-2}^M}{K_{n_k-1}^R}D_{24}(k-2)$ is nonzero and the claim follows.
\end{proof}

\begin{prop}\label{prop:dgahtpy}
The ungraded augmentation $\Aug_u(\Lambda[n_1, \ldots, n_k])$ can be identified with the affine variety given by
\begin{align*}
    \left\{\left(\begin{array}{l}a_1,a_2,\dots, a_{m_1},\\
    a_{m_{(i-1)}+2},\dots, a_{m_i} \\
    \text{for $1<i<k$,} \\
    a_{m_{k-1}+2},\dots, a_{m_k} \end{array}\right)\in \F^{m_k-k+1}~\middle|~\begin{array}{l} K_{n_1}(a_1,\ldots, a_{m_1})=0 \\
    K_{n_i-1}(a_{m_{(i-1)}+2}, \ldots, a_{m_i})=0 \quad \text{for $1<i<k$},\\
    K_{n_k-1}(a_{m_{k-1}+2}, \ldots, a_{m_k})\neq 0
    \end{array} \right\}.
\end{align*}
\end{prop}

\begin{proof}
 Any ungraded augmentation of $\Lambda[n_1, \ldots, n_k]$ is given by $\e\circ \partial a_{m_i+1}=0,$ and $\e\circ \partial b_j=0$ for $i=1, \ldots,k-1$ and $j=1,2.$ In Lemma~\ref{lem: Kni-1R=0}, we obtain the equations $\e(K_{n_1}(a_1, \ldots, a_{n_1}))=0$ and $\e(K^R_{n_{(i-1)}}(a_{m_{(i-1)}+2}, \ldots, a_{m_{i}}))=0$ for $1< i<k$, as well as an inequality $\e(K_{n_k-1}^R(a_{m_{k-1}+2}, \ldots, a_{m_k}))\neq0$, which uniquely determines the value of $\e(t_1)$. Proposition \ref{prop: partial b_2 redundant} then shows that the value of $\e(t_2)$ is also uniquely determined by the equations and inequality above. Thus, the only remaining degrees of freedom in defining an augmentation lie in $\e(a_{m_i+1})$ for $i=1, \ldots, k-1$. 

Recall that two ungraded augmentations are considered to be the same point in the ungraded augmentation vareity if they are dg-algebra homotopic.
Let $B$ denote a subset of the Reeb chords of $\Lambda[n_1, \ldots, n_k]$, and let $\e_1^B$ and $\e_2^B$ be two ungraded augmentations of $\Lambda[n_1, \ldots, n_k]$ such that $\e_1^B$ and $\e_2^B$ are equal except for possibly when evaluated on the Reeb chords in the set $B.$ If $a$ is a Reeb chord such that $\partial(a)=0$, then $\e^{\{a\}}_1(a)\neq \e^{\{a\}}_2(a)$ implies that $\e^{\{a\}}_1$ and $\e^{\{a\}}_2$ are not dg-algebra homotopic and therefore not in the same equivalence class. As a result, we consider all Reeb chords $a$ of $\Lambda[n_1, \ldots, n_k]$ such that $\partial(a)\neq 0$.

 Let $B:=\{a_{m_1+1},\ldots, a_{m_{k-1}+1},b_1,b_2\}$ denote the collection of such Reeb chords. We will now show that $\e^B_1$ and $\e_2^B$ are dg-algebra homotopic. Let $h:\mathcal{A}(\Lambda[n_1, \ldots, n_k])\longrightarrow \F$ be the chain operator given by $h(a_j)=0$ if $j\neq m_i$ for $i=1, \ldots, k-1$. 

Then, let 

\begin{align*}
    h(a_{m_1})&=(\e^B_1(K_{n_1-1}^L))^{-1}(\e^B_1(a_{n_1+1})-\e^B_2(a_{n_1+1}))\\
    h(a_{m_i})&=(\e^B_1(K_{n_1-1}^LK_{n_2-2}^M\cdots K_{n_{(i-2)}-2}^MK_{n_{(i-1)}-2}^M)\e^B_1(K_{n_i-2}^M))^{-1}(\e^B_1(a_{m_i+1})-\e^B_2(a_{m_i+1}))\\
    h(t_1)&= \e^B_1(b_1)+\e^B_2(b_1)+h(K_{n_1-1}^LK_{n_2-2}^M\dots K_{n_{(k-1)}-2}K_{n_k-1}^R)\\
    h(t_2)&=\begin{cases} \e^B_1(b_1)+\e^B_2(b_1)+h( K_{n_k}D_{34}(k-1)+K_{n_k-1}^RD_{24}(k-1))  & k \text{ odd}\\
\e^B_1(b_2)+\e^B_2(b_2)+h(D_{14}(k)+D_{34}(k)b_1t_1D_{13}(k)+D_{24}(k)t_1D_{13}(k))  
& k \text{ even}\end{cases}
\end{align*}

From the continuant recursion relation $K_{n+1}=K_nx_{n+1}+K_{n-1}$, the Leibniz relation from Definition ~\ref{dg-algebra-homotopy}, and the definition of $h$, we have that
\begin{align*}
    h(K_{n_{i}-1}^R)&=h(K^M_{n_{i}-2}a_{m_i}+K_{n_i-3}(a_{m_{(i-1)}+2}, \dots,  a_{m_{i-2}}))\\
    &=h(K^M_{n_{i}-2})\e_2^B(a_{m_i})+h(a_{m_i})\e_1^B(K^M_{n_{i}-2})\\
    &=h(a_{m_i})\e_1^B(K^M_{n_{i}-2}).
\end{align*}

Recall that $\partial(a_{m_1+1})=K_{n_1}$, and $\partial(a_{m_i+1})=K_{n_{1-1}}^LK_{n_2-2}^M\cdots K_{n_{(i-1)}-2}^M K_{n_{i}-1}^R$ for $1< i<k$. As argued in the proof of Lemma~\ref{lem: Kni-1R=0}, $\e(K_{n_1})=0$ and $\e(K_{n_i-1}^R)=0$ for any ungraded augmentation $\e$ of $\SA(\Lambda[n_1, \ldots, n_k])$. So, 
\begin{align*}
\e^B_1(a_{m_i+1})-\e^B_2(a_{m_i+1})&= h(a_{m_i})\e_1^B(K^M_{n_{i}-2})\e^B_1(K_{n_{1-1}}^LK_{n_2-2}^M\cdots K_{n_{(i-1)}-2}^M)\\
&=
h(K_{n_{i}-1}^R)\e^B_1(K_{n_{1-1}}^LK_{n_2-2}^M\cdots K_{n_{(i-1)}-2}^M)\\
&=h(K_{n_{i}-1}^R)\e^B_1(K_{n_{1-1}}^LK_{n_2-2}^M\cdots K_{n_{(i-1)}-2}^M)+\e^B_2(K_{n_{i}-1}^R)h(K_{n_1-1}^LK_{n_2-2}^M\cdots K_{n_{(i-1)}-2}^M)\\
&=h(K_{n_1-1}^LK_{n_2-2}^M\cdots K_{n_{(i-1)}-2}^M K_{n_{i}-1}^R)\\
&=h(\partial(a_{m_i+1})).
\end{align*}

Analogously, one can show that the chain map $h$ we have defined satisfies
$\e^B_1(a_{n_1+1})-\e^B_2(a_{n_1+1})=h(\partial (a_{n_1+1}))$. It is straightforward to check that $\e^B_1(b_i)-\e^B_2(b_i)=h(\partial (b_i))$ for $i=1,2$. Thus, $h$ is a chain homotopy, and the augmentations $\e_1^B$ and $\e_2^B$ are dg-algebra homotopic and therefore in the same equivalence class. 
\end{proof}

\begin{cor}\label{cor: dgahtpy}The ungraded augmentation $\Aug_u(\Lambda[n_1, \ldots, n_k])$ can be identified with the affine variety given by
\begin{align*}
    \left\{\left(\begin{array}{l}a_1,a_2,\dots, a_{m_1},\\
    a_{m_{(i-1)}+2},\dots, a_{m_i} \\
    \text{for $1<i<k$,} \\
    a_{m_{k-1}+1},\dots, a_{m_k} \end{array}\right)\in \F^{m_k-k+2}~\middle|~\begin{array}{l} K_{n_1}(a_1,\ldots, a_{m_1})=0 \\
    K_{n_i-1}(a_{m_{(i-1)}+2}, \ldots, a_{m_i})=0 \quad \text{for $1<i<k$},\\
    K_{n_k}(a_{m_{k-1}+1}, \ldots, a_{m_k})= 0
    \end{array} \right\}.
\end{align*}
\end{cor}
\begin{proof} Note that the only difference of this statement from Proposition \ref{prop:dgahtpy} is the last block $n_k$, where we introduce one extra variable $a_{m_{k-1}+1}$ and change the inequality $K_{n_k-1}(a_{m_{k-1}+2},\dots, a_{m_k})\neq 0$ to an equation $K_{n_k}(a_{m_{k-1}+1},\dots, a_{m_k})=0$. To see that they are equivalent, we first observe that if $K_{n_k-1}(a_{m_{k-1}+2},\dots, a_{m_k})\neq 0$, then there is a unique value for $a_{m_{k-1}+1}$ so that the equation
\[
0=K_{n_k}(a_{m_{k-1}+1},\dots, a_{m_k})=a_{m_{k-1}+1}K_{n_k-1}(a_{m_{k-1}+2},\dots, a_{m_k})+K_{n_k-2}(a_{m_{k-1}+3},\dots, a_{m_k})
\]
holds (Definition \ref{def of K_n}). Conversely, if $K_{n_k}(a_{m_{k-1}+1},\dots, a_{m_k})=0$, then by Lemma \ref{lem: Braid Matrix Det}, we must have $K_{n_k-1}(a_{m_{k-1}+2},\dots, a_{m_k})\neq 0$. This shows that the claimed affine variety is isomorphic to that in Proposition \ref{prop:dgahtpy}
\end{proof}

\begin{cor}\label{cor: topological invariant} If $\La$ and $\La'$ admit Legendrian rational forms and are smoothly isotopic as unoriented links, then $\Aug_u(\Lambda)\cong \Aug_u(\Lambda')$ as affine varieties.
\end{cor}

\begin{proof} By Proposition \ref{prop: isotopic rational links}, the Legendrian rational forms of $\Lambda$ and $\Lambda'$ are related by compositions of three types of moves: 
\begin{enumerate}
    \item $[n_1,\dots, n_{k-1},n_k]\sim [n_1,\dots, n_{k-1},n_k+1,1]$,
    \item $[n_1,n_2,\dots, n_k]\sim [1,n_1+1,n_2,\dots, n_k]$, and
    \item $[n_1,n_2,\dots, n_k]\sim [n_k,\dots, n_2,n_1]$.
\end{enumerate}
The statement then follows from the fact that the affine varieties described in Corollary \ref{cor: dgahtpy} are isomorphic under these changes. Note that Move $(3)$ uses Lemma \ref{lem: palindromic}.
\end{proof}

\begin{rmk} 

Among the three moves listed in Proposition \ref{prop: isotopic rational links} and Corollary \ref{cor: topological invariant}, the first can be realized by a single Reidemeister II move, while the other two moves are generally not Legendrian isotopies.

\end{rmk}

\section{Background on Cluster Varieties}\label{sec:clusterbackground}

In this section, we review the basics of cluster varieties and their cluster localizations. We also review the example of the Grassmannian $\Gr_{2,n}^{\circ}$ and its cluster localizations, which we will use in Section~\ref{sec:cluster} to define the cluster structure on $\Aug_u(\Lambda[n_1, \ldots, n_k])$. Within this section, we assume $\bF$ to be an algebraically closed field of \textbf{any characteristic}.

\subsection{Cluster Varieties and Cluster Localization}

\begin{defn} A \emph{cluster seed} $\mathbf{s}$ is a split algebraic torus $T_\seed\cong (\bF^\times)^n$ together with a system of coordinates $\bfA_\seed=\{A_1,A_2,\dots, A_n\}$ called \emph{cluster variables} and a quiver $Q_\seed$ whose vertices are in bijection with the cluster variables. The quiver $Q_\mathbf{s}$ is required to not contain any 1-cycles or 2-cycles. We sometimes select a subset of the vertices of $Q_\seed$ and declare them to be \emph{frozen}, and any vertex that is not frozen is said to be \emph{mutable}. A cluster variable is said to be \emph{frozen} (resp. \emph{mutable}) if its corresponding quiver vertex is frozen (resp. mutable).
\end{defn}

\begin{defn}\label{defn: quiver mutation} The data of a quiver $Q=Q_\seed$ without 1-cycles and 2-cycles can be encoded by a skew-symmetric matrix $\epsilon$ called the \emph{exchange matrix} of $Q$. Its entries are given by
\[
\epsilon_{ij}:=\#(i\rightarrow j)-\#(j\rightarrow i).
\]
Let $k$ be a mutable vertex of $Q$. The \emph{quiver mutation} of $Q$ in the direction of $k$ produces a new quiver $Q'$ whose exchange matrix $\epsilon'$ is
\[
\epsilon'_{ij}=\left\{\begin{array}{ll}-\epsilon_{ij} & \text{if $k\in \{i,j\}$},\\
\epsilon_{ij}+[\epsilon_{ik}]_+[\epsilon_{kj}]_+-[-\epsilon_{ik}]_+[-\epsilon_{kj}]_+ & \text{if $k\notin\{i,j\}$},
\end{array}\right.
\]
where $[x]_+:=\max\{x,0\}$. Note that quiver mutation is an involution.
\end{defn}

\begin{defn}\label{defn: cluster mutation} Let $\seed$ be a cluster seed and let $k$ be a mutable vertex in $Q_\seed$. The \emph{cluster mutation} of $\seed$ in the direction of $k$ produces a new cluster seed $\seed'=(T_{\seed'},\{A'_1,A'_2,\dots, A'_n\}, Q_{\seed'})$ together with a birational map $\mu_k:T_\seed\dashrightarrow T_{\seed'}$ such that
\begin{itemize}
    \item $Q_{\seed'}$ is obtained from $Q_\seed$ via a quiver mutation in the direction of $k$;
    \item the pull-back map $\mu_k^*$ is 
    \[
    \mu_k^*(A'_i)=\left\{\begin{array}{ll}
        A^{-1}_k\left(\prod_{j}A^{[\epsilon_{kj}]_+}_j+\prod_{j}A^{[-\epsilon_{kj}]_+}_j\right) & \text{if $i=k$}, \\
        A_i & \text{if $i\neq k$}.\end{array}\right.
    \]
\end{itemize}
Note that cluster mutation is also an involution, and frozen cluster variables never change under cluster mutation. Two cluster seeds are said to be \emph{mutation equivalent} to each other if one can be obtained from the other via a sequence of cluster mutations.
\end{defn}

\begin{defn}\label{defn: full-rank} Let $\epsilon$ be the exchange matrix for the quiver $Q_\seed$. The seed $\seed$ is said to be \emph{full-rank} if the submatrix formed by the mutable rows in $\epsilon$ is full-rank. The seed $\seed$ is said to be \emph{really full-rank} if the column vectors of the submatrix formed by the mutable rows in $\epsilon$ span $\mathbb{Z}^{\# \text{mutable vertices}}$. Note that full-rank-ness and really full-rank-ness are invariant under cluster mutations.
\end{defn} 
The following theorem is a fundamental result in cluster theory.

\begin{thm}[Laurent phenomenon {\cite[Theorem 3.1]{FZI}}]\label{thm: LP} Suppose $\seed$ and $\seed'$ are mutation equivalent. Then under the composition of the cluster mutation maps, any cluster variable in $\seed'$ can be expressed as a Laurent polynomial in terms of the cluster variables in $\seed$.
\end{thm}

Cluster varieties are introduced by Fock and Goncharov~\cite{FockGoncharov_ensemble} as a geometric enhancement of Fomin and Zelevinsky's cluster algebras. Below is a definition of cluster varieties\footnote{Following the convention of Fock and Goncharov, there are two versions of cluster varieties: cluster $K_2$ varieties and cluster Poisson varieties. In this article, we only deal with the former.}.

\begin{defn}\label{defn: cluster variety} A \emph{cluster variety} $\mathscr{A}$ is an affine variety together with a collection of cluster seeds $\mathbf{S}=\{\seed=(T_\seed,\bfA_\seed,Q_\seed)\}$, satisfying
\begin{itemize}
    \item each cluster seed $\seed$ is equipped with and inclusion $T_\seed\hookrightarrow \mathscr{A}$ such that $\bigcup_{\seed\in \mathbf{S}}T_\seed$ covers $\mathscr{A}$ up to codimension 2; each $T_\seed$ is called a \emph{cluster torus chart} on $\mathscr{A}$.
     \item for any cluster seed $\seed\in \mathbf{S}$ and any mutable vertex $k$ in $Q_\seed$, the cluster seed produced by cluster mutation of $\seed$ in the direction of $k$ is also in $\mathbf{S}$;
    \item any two cluster seeds $\seed, \seed'\in\mathbf{S}$ are mutation equivalent to each other, such that the transition map between $T_\seed$ and $T_{\seed'}$ is birationally equivalent to the composition of the cluster mutation maps.
\end{itemize} 
The coordinate ring $\mathcal{O}(\mathscr{A})$ is called an \emph{upper cluster algebra}. The subalgebra in $\mathcal{O}(\mathscr{A})$ generated by cluster variables and the inverses of the frozen cluster variables is called a \emph{cluster algebra}.
\end{defn}

 From Definition \ref{defn: cluster variety} we see that in order to describe the cluster structure on a cluster variety, we only need to specify a single cluster seed $\seed$. We often call such a cluster seed an \emph{initial cluster seed} and its cluster variable the \emph{initial cluster variables}. As a consequence of the Laurent phenomenon (Theorem \ref{thm: LP}), every cluster variable in every cluster seed is a global function on the whole cluster variety $\mathscr{A}$. If $A_i$ is a frozen cluster variable, then it is a $\bF^\times$-valued function on $\mathscr{A}$ and hence a unit in $\mathcal{O}(\mathscr{A})$. On the other hand, if $A_i$ is a mutable cluster variable on $\mathscr{A}$, the \emph{cluster localization} of $\mathscr{A}$ at $A_i$ is the non-vanishing locus $\{A_i\neq 0\}$ in $\mathscr{A}$. The coordinate ring of a cluster localization is equal to the localization of $\mathcal{O}(\mathscr{A})$ at $A_i$ as algebras.

Below are some well-known results on cluster varieties.

\begin{lem}\label{lem: product of cluster varieties} If $\mathscr{A}$ and $\mathscr{A}'$ are both cluster varieties with initial cluster seeds $\seed$ and $\seed'$, then $\mathscr{A}\times \mathscr{A}'$ is also a cluster variety, with an initial cluster seed $\seed\times \seed':=(T_\seed\times T_{\seed'}, \bfA_\seed\sqcup \bfA_{\seed'}, Q_\seed \sqcup Q_{\seed'})$. Moreover, cluster seeds in $\mathscr{A}\times \mathscr{A}'$ are naturally products of cluster seeds in $\mathscr{A}$ and $\mathscr{A}'$. If $\mathscr{A}$ and $\mathscr{A}'$ are both full-rank (resp. really full-rank), then so is $\mathscr{A}\times \mathscr{A}'$.
\end{lem}
\begin{proof} This follows from the definition of cluster varieties (Definition \ref{defn: cluster variety}).
\end{proof}

\begin{lem}\label{lem: deleting frozen vertex} Let $\mathscr{A}$ be a cluster variety and let $A_i$ be a frozen cluster variable. Let $\seed$ be a cluster seed for $\mathscr{A}$, which must contain $A_i$ as a cluster variable since it is frozen (c.f. Definition \ref{defn: cluster mutation}). Then the locus $\{A_i=1\}$ is also a cluster variety, with an initial seed $\seed'$ obtained from $\seed$ by setting $T_{\seed'}:=\{A_i=1\}\subset T_\seed$ and $\bfA_{\seed'}:=\bfA_\seed\backslash \{A_i\}$, and deleting the vertex $i$ from $Q_\seed$ to obtain $Q_{\seed'}$. 
\end{lem}
\begin{proof} This follows from the fact that setting a frozen cluster variable $A_i=1$ commutes with quiver mutations (Definition \ref{defn: quiver mutation}) and cluster mutations (Definition \ref{defn: cluster mutation}).
\end{proof}

\begin{defn} A cluster seed $\seed$ is said to be \emph{acyclic} if the quiver $Q_\seed$ does not contain any oriented cycles with only mutable vertices. A cluster variety $\mathscr{A}$ is said to be \emph{acyclic} if it has an acyclic cluster seed $\seed$.
\end{defn}

There is a broader family of well-behaving cluster varieties/algebras called ``locally acyclic cluster varieties/algebras'', introduced by Muller in~\cite{Muller}. We do not need such generality in this article, but we would like to point out that acyclic cluster varieties are locally acyclic, and therefore the following result of Muller applies.

\begin{thm} [{\cite[Lemma 3.4, Theorem 4.1]{Muller}}]\label{thm: cluster localization} Suppose $\mathscr{A}$ is locally acyclic and $A_i$ is a cluster variable in a seed $\seed$ for $\mathscr{A}$. Then the cluster localization of $\mathscr{A}$ at $A_i$ is also a locally acyclic cluster variety, with an initial cluster seed $\seed'$ obtained from $\seed$ by freezing the vertex $i$ in the quiver $Q_\seed$.
\end{thm}

\subsection{\texorpdfstring{Example: Grassmannian $\Gr_{2,n}^\circ$ and its Cluster Localization}{}} \label{sub: Gr2n}

In this subsection, we review a famous example of an acyclic cluster variety, which is a subset of the affine cone of Grassmannian $(2,n)$. We then consider a particular family of cluster localizations of this cluster variety. We will see later that this cluster localization is closely related to the dg-algebra homomorphism induced by a saddle cobordism.

\begin{defn} We define $\Gr_{2,n}^\circ$ to be the quotient space 
\[
\SL_2\backslash\{M\in \mathrm{Mat}_{2\times n}\mid \text{$\Delta_{i,i+1}(M)\neq 0$ for all $1\leq i<n$ and $\Delta_{1,n}(M)\neq 0$}\},
\]
where $\Delta_{i,j}$ denotes the $(i,j)$th Pl\"{u}cker coordinate, which is the determinant of the submatrix formed by the $i$th and $j$th columns of the matrix $M$. As an affine variety, the coordinate ring of $\Gr_{2,n}^\circ$ is
\[
\mathcal{O}(\Gr_{2,n}^\circ)=\frac{\bF[\Delta_{i,j} \text{ for } 1\leq i<j\leq n, \  \Delta_{i,i+1}^{-1} \text{ for } 1\leq i<n, \text{ and } \Delta_{1,n}^{-1}]}{(\Delta_{i,j}\Delta_{k,l}+\Delta_{j,k}\Delta_{i,l}=\Delta_{i,k}\Delta_{j,l} \ \text{for $1\leq i<j<k<l\leq n$})}.
\]
\end{defn}

In~\cite{FZII}, Fomin and Zelevinsky describe the cluster structure on $\Gr_{2,n}^\circ$ combinatorially using triangulations of a labeled $n$-gon. Let us review the cluster structure in detail below. We label the vertices of an $n$-gon clockwise by $1,2,\dots, n$. Then each diagonal/side $e$ of the $n$-gon connects two distinct vertices $i,j$, to which we associate the Pl\"{u}cker coordinate $\Delta_e:=\Delta_{i,j}$ (in the order $i<j$). Given a triangulation $\tau$ of the $n$-gon, we cut out an open subset $U_\tau$ of $\Gr_{2,n}^\circ$ by imposing the conditions $\Delta_e\neq 0$ for all diagonals $e$ present in the triangulation $\tau$ (note that for each side $e$, $\Delta_e\neq 0$ is already satisfied by the definition of $\Gr_{2,n}^\circ$). Each $U_\tau$ is a cluster torus chart on $\Gr_{2,n}^\circ$, with the cluster
\[
\{\Delta_e\mid \text{$e$ is a side or a diagonal present in $\tau$}\}.
\]
The quiver $Q_\tau$ associated with a triangulation $\tau$ is constructed by placing a mutable quiver vertex at each diagonal in $\tau$, a frozen quiver vertex at each side of the $n$-gon, and then drawing a counterclockwise $3$-cycle inside each triangle in $\tau$ to link these quiver vertices together. Under this construction, a seed mutation corresponds to a flip of diagonal, which changes one triangulation $\tau$ to another triangulation $\tau'$. Note that the two triangulations only differ by one diagonal, and the corresponding quivers are related precisely by a quiver mutation at the differing diagonal. Moreover, the Pl\"{u}cker relation on $\Gr_{2,n}^\circ$ can be rearranged into the cluster mutation formula, showing that these seeds are indeed related by a cluster mutation.

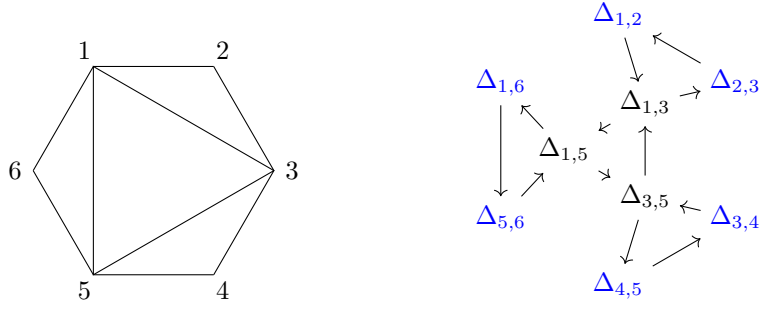
\begin{figure}[H]
    \centering
    \begin{tikzpicture}[scale=0.8]
    \foreach \i in {1,...,6}
    {
    \draw (\i*60:2) -- (\i*60+60:2);
    \node at (180-60*\i:2.3) [] {$\i$};
    }
    \draw (120:2) -- (-120:2) -- (2,0) -- cycle;
    \end{tikzpicture}\hspace{2cm}
    \begin{tikzpicture}[yscale=0.9, xscale=1.1]
        \node [blue] (12) at (0,2) [] {$\Delta_{1,2}$};
        \node [blue] (23) at (30:2) [] {$\Delta_{2,3}$};
        \node [blue] (34) at (-30:2) [] {$\Delta_{3,4}$};
        \node [blue] (45) at (0,-2) [] {$\Delta_{4,5}$};
        \node [blue] (56) at (-150:2) [] {$\Delta_{5,6}$};
        \node [blue] (16) at (150:2) [] {$\Delta_{1,6}$};
        \node (13) at (60:0.8) [] {$\Delta_{1,3}$};
        \node (35) at (-60:0.8) [] {$\Delta_{3,5}$};
        \node (15) at (-0.8,0) [] {$\Delta_{1,5}$};
        \draw [->] (12) -- (13);
        \draw [->] (13) -- (23);
        \draw [->] (23) -- (12);
        \draw [->] (34) -- (35);
        \draw [->] (35) -- (45);
        \draw [->] (45) -- (34);
        \draw [->] (15) -- (16);
        \draw [->] (16) -- (56);
        \draw [->] (56) -- (15);
        \draw [->] (15) -- (35);
        \draw [->] (35) -- (13);
        \draw [->] (13) -- (15);
    \end{tikzpicture}
    \caption{An example of a triangulation and its corresponding cluster seed in $\Gr_{2,6}^\circ$.}
\end{figure}

\begin{lem}\label{lem: Gr(2,n) is acyclic} $\Gr_{2,n}^\circ$ is an acyclic cluster variety.
\end{lem}
\begin{proof} Note that the diagonals connecting $1$ with $i=3,\dots, n-1$ form a triangulation of the $n$-gon, whose quiver has an acyclic mutable part.
\end{proof}

\begin{rmk} Because the underlying undirected graph of the mutable part of the acyclic quiver defined in the proof of Lemma \ref{lem: Gr(2,n) is acyclic} is isomorphic to a type $\mathrm{A}_{n-3}$ Dynkin diagram, we call $\Gr_{2,n}^\circ$ a cluster variety of type $\mathrm{A}_{n-3}$.
\end{rmk}

Next, we would like to give two particular parametrizations of $\Gr_{2,n}^\circ$ as affine varieties that will be closely related to the computation of augmentation varieties of $2$-bridge links. Let us first introduce two special matrices:
\[
B(x):=\begin{pmatrix} x & -1 \\ 1 & 0\end{pmatrix} \quad \text{and} \quad D(x):=\begin{pmatrix} x & 0 \\ 0 & x^{-1}\end{pmatrix}.
\]
For a matrix $M$, we denote the $(i,j)$th entry of $M$ by $(M)_{ij}$.

\begin{prop}\label{prop: affine variety} $\Gr_{2,n}^\circ$ is an affine subvariety in $\bF^{n-1}_{a_1,\dots, a_{n-1}}\times (\bF^\times)^{n-1}_{p_1,\dots, p_{n-1}}$ cut out by
\begin{equation}\label{eq: defining equation for Gr_2,n}
\left(B(a_1)D(p_1)B(a_2)D(p_2)\cdots B(a_{n-1})D(p_{n-1})\right)_{11}=0,
\end{equation}
\end{prop}
\begin{proof} Let $M$ be a matrix representative of a point in $\Gr_{2,n}^\circ$ and let $v_1,v_2,\dots, v_n$ be the column vectors of $M$. By using the $\SL_2$ action, we can choose $M$ so that the submatrix $(v_1,v_n)$ is of the form $\begin{pmatrix}
    1 & 0 \\ 0 & p_n
\end{pmatrix}$ for some $p_n\in \bF^\times$. 

Note that for $a_1\in \bF$ and $p_1\in \bF^\times$, the product $B(a_1)D(p_1)=\begin{pmatrix} a_1p_1 & p_1^{-1} \\ p_1 & 0 \end{pmatrix}$, whose first column vector is linearly independent from $v_1=\begin{pmatrix} 1 \\ 0 \end{pmatrix}$; on the other hand, any vector linearly independent from $v_1$ can be obtained as the first column vector of $B(a_1)D(p_1)$ for some $a_1\in \bF$ and $p_1\in \bF^\times$. Since $\Delta_{1,2}$ is required to be non-zero, the vector $v_2$ can be parametrized as $B(a_1)D(p_1)v_1$. Inductively, we see that there exist parameters $a_i\in \bF$ and $p_i\in \bF^\times$ for $1\leq i<n$ such that
\begin{equation}\label{eq: v_{j+1}}
v_{j+1}=B(a_1)D(p_1)B(a_2)D(p_2)\cdots B(a_j)D(p_j)v_1 \quad \text{for all $2\leq j<n$.}
\end{equation}
Note that in the end, we need $v_n=\begin{pmatrix} 0 \\ p_n\end{pmatrix}$ with $p_n\neq 0$, so we require the parameters $a_i$'s and $p_i$'s to satisfy Equation \ref{eq: defining equation for Gr_2,n}. Note that since $\det B(x)=\det D(x)=1$ for any $x$, Equation \eqref{eq: defining equation for Gr_2,n} automatically implies that 
\begin{equation}\label{eq: p_n}
p_n=\left(B(a_1)D(p_1)B(a_2)D(p_2)\cdots B(a_{n-1})D(p_{n-1})\right)_{21}\neq 0.
\end{equation}
This finishes the proof of the proposition.
\end{proof}

\begin{prop}\label{prop: alternative affine variety} $\Gr_{2,n}^\circ$ is also an open affine subvariety in $\bF^{n-2}_{a_1,\dots, a_{n-2}}\times (\bF^\times)^{n-1}_{p_1,\dots, p_{n-2},p_n}$ cut out by the inequality
\begin{equation}\label{eq: defining inequality for Gr(2,n)}
\left(B(a_1)D(p_1)B(a_2)D(p_2)\cdots B(a_{n-2})D(p_{n-2})\right)_{11}\neq 0.
\end{equation}
\end{prop}
\begin{proof} We set up the matrix representative $M$ of a point in $\Gr_{2,n}^\circ$ the same way as in the proof of Proposition \ref{prop: affine variety}, i.e., such that the submatrix $(v_1,v_n)$ is of the form $\begin{pmatrix} 1 & 0 \\ 0 & p_n\end{pmatrix}$ for some $p_n\in \bF^\times$. Again, we obtain \eqref{eq: v_{j+1}} by the same argument as in that proof.

Now instead of looking at $v_n$, we consider $v_{n-1}$; in order to define a point in $\Gr_{2,n}^\circ$, we just need $\Delta_{n-1,n}=\det(v_{n-1},v_n)\neq 0$. Since $v_n=\begin{pmatrix} 0 \\ p_n\end{pmatrix}$ and $p_n\neq 0$, it follows that all we need is the inequality \eqref{eq: defining inequality for Gr(2,n)}. Note that in this setup, the parameter $p_n$ is a free parameter.
\end{proof}

The next proposition describes the relation between the parameters in Proposition \ref{prop: affine variety} and Pl\"{u}cker coordinates.

\begin{prop}\label{prop: a_i and p_i in terms of Pluckers} Let $a_1, \ldots, a_{n-1}$ and $p_1, \ldots, p_n$ be the parameters appearing in Propositions \ref{prop: affine variety} and \ref{prop: alternative affine variety}. Then 
\begin{enumerate}[label=$(\roman*)$]
    \item $\Delta_{j,j+1}=p_j$ for $1\leq j< n$ and $\Delta_{1,n}=p_n$;  
    \item $\Delta_{j-1,j+1}=a_jp_{j-1}p_j$ for $1<j< n$ and $\Delta_{2,n}=a_1p_1p_n$.
    \item Let $i,j$ be indices such that $1\leq i<j\leq n$. If $p_k=1$ for all $i\leq k<j$, then 
    \[
    \Delta_{i,j}=K_{j-i-1}(a_{i+1},a_{i+2},\dots, a_{j-1}),
    \]
    where $K_{j-i-1}$ is the continuant polynomial defined in Definition \ref{def of K_n}.
\end{enumerate}
Note that $p_n$ can be expressed in terms of $a_i$'s and $p_i$'s as in Equation \eqref{eq: p_n}.
\end{prop}
\begin{proof} For (i), we first note that $\Delta_{1,n}=\det\begin{pmatrix} 1 & 0 \\0 & p_n\end{pmatrix}=p_n$. On the other hand, for any $1\leq j<n$, we can act on $M$ by $\left(B(a_1)D(p_1)\cdots B(a_{j-1}D(p_{j-1})\right)^{-1}$ so that $v_j$ becomes $\begin{pmatrix} 1 \\ 0 \end{pmatrix}$. Then $v_{j+1}=B(a_j)D(p_j)v_j=\begin{pmatrix} a_jp_j \\ p_j\end{pmatrix}$ and hence 
\[
\Delta_{j,j+1}=\det\begin{pmatrix} 1 & a_jp_j \\ 0 & p_j\end{pmatrix}=p_j.
\]

For (ii), we note that if $j<n-1$, we have $v_{j+2}=B(a_j)D(p_j)B(a_{j+1})D(p_{j+1})\begin{pmatrix} 1 \\ 0 \end{pmatrix}=\begin{pmatrix} a_ja_{j+1}p_jp_{j+1}-p_j^{-1}p_{j+1}\\ a_{j+1}p_jp_{j+1}\end{pmatrix}$, which implies that
\[
\Delta_{j,j+2}=\det\begin{pmatrix} 1 & a_ja_{j+1}p_jp_{j+1}-p_j^{-1}p_{j+1} \\ 0 & a_{j+1}p_jp_{j+1}\end{pmatrix}=a_{j+1}p_jp_{j+1}.
\]
This proves (ii) except for the case where $j=1$. When $j=1$, $v_2=\begin{pmatrix} a_1p_1 \\ p_1\end{pmatrix}$ and hence
\[
\Delta_{2,n}=\det\begin{pmatrix} a_1p_1 & 0 \\ p_1 & p_n\end{pmatrix}=a_1p_1p_n.
\]

For (iii), similar to (i), we can act by an $\mathrm{SL}_2$-matrix on the left and assume without loss of generality that $i=1$. Then by \eqref{eq: v_{j+1}}, we see that
\[
\Delta_{1,j}=\det(v_1,v_j)=(B(a_1)B(a_2)\cdots B(a_{j-1}))_{21}=K_{j-2}(a_2,\dots, a_{j-1}),
\]
where the last equality follows from Lemma \ref{lem: braid matrix prod}.
\end{proof}

Proposition \ref{prop: a_i and p_i in terms of Pluckers} leads us to give the following alternative definition for the parameters $a_i$'s and $p_i$'s.

\begin{defn} We define two families of regular functions on $\Gr_{2,n}^\circ$, both indexed by $1\leq i\leq n$:
\[
a_1:=\frac{\Delta_{2,n}}{\Delta_{1,n}\Delta_{1,2}}, \quad  a_n:=\frac{\Delta_{1,n-1}}{\Delta_{n-1,n}\Delta_{1,n}},\quad \text{and} \quad a_i:=\frac{\Delta_{i-1,i+1}}{\Delta_{i-1,i}\Delta_{i,i+1}}\quad  \text{ for $1<i<n$},
\]
and
\[
p_n:= \Delta_{1,n} \quad \text{and} \quad p_i:=\Delta_{i,i+1} \quad \text{for $1\leq i<n$}.
\]
We associate the function $a_i$ with the vertex $i$ of the labeled $n$-gon, and the function $p_i$ with the edge between $i$ and $i+1$ of the labeled $n$-gon (in particular, $p_n$ is associated with the edge between $1$ and $n$).
\end{defn}

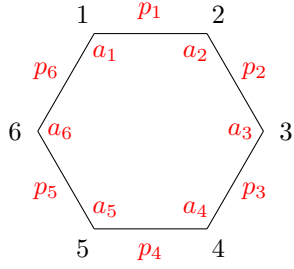
\begin{figure}[H]
    \centering
    \begin{tikzpicture}
    \foreach \i in {1,...,6}
    {
    \draw (\i*60:1.5) -- (\i*60+60:1.5);
    \node at (180-60*\i:1.8) [] {$\i$};
    \node [red] at (180-60*\i:1.2) [] {$a_\i$};
    \node [red] at (150-60*\i:1.6) [] {$p_\i$};
    }
    \end{tikzpicture}
    \caption{Example of the association of the functions $a_i$'s and $p_i$'s.}
\end{figure}

\begin{rmk} The function $a_n$ is not among the parameters in Equation \eqref{eq: defining equation for Gr_2,n}. Similar to $p_n$ in Equation \eqref{eq: p_n}, $a_n$ can be expressed in terms of $a_i$'s and $p_i$'s for $1\leq i<n$. However, we do not make use of $a_n$ in any significant way.
\end{rmk}

\begin{rmk} By connecting $\Gr_{2,n}^\circ$ to the decorated Teichm\"{u}ller space on an $n$-gon~\cite{FT}, the functions $a_i$ can be interpreted geometrically as lengths of horocycles. The functions $a_i$ can also be used as \emph{potential functions} in constructing Landau-Ginzburg models in mirror symmetry~\cite{GS}.
\end{rmk}

Let $\tau$ be a triangulation and let $T_{abc}$ be the triangle inside a triangulation $\tau$ with sides $a,b,c$. Let $\alpha$, $\beta$, $\gamma$ be the angles in $T_{abc}$ opposite to the sides $a,b,c$ respectively. We define a function $W_\alpha$ for the angle $\alpha$ by
\[
W_\alpha:=\frac{\Delta_a}{\Delta_b\Delta_c},
\]
and define functions $W_\beta$ and $W_\gamma$ in an analogous fashion. The following identity relating the functions $a_i$'s and $W_\alpha$'s is well-known; we include a short proof below for completeness.

\begin{lem}\label{lem: potential function} If $\alpha_1,\dots, \alpha_k$ are all the angles in triangles in $\tau$ at the vertex $i$, then
\[
a_i=\sum_{j=1}^k W_{\alpha_j}.
\]
\end{lem}
\begin{proof} It suffices to prove the case where there are only two angles $\alpha$ and $\beta$ at the vertex $i$. In this case, we have
\[
W_\alpha+W_\beta=\frac{\Delta_{i-1,k}}{\Delta_{i-1,i}\Delta_{i,k}}+\frac{\Delta_{i+1,k}}{\Delta_{i,k}\Delta_{i,i+1}}=\frac{\Delta_{i-1,k}\Delta_{i,i+1}+\Delta_{i+1,k}\Delta_{i-1,i}}{\Delta_{i-1,i}\Delta_{i,k}\Delta_{i,i+1}}=\frac{\Delta_{i-1,i+1}}{\Delta_{i-1,i}\Delta_{i,i+1}}=a_i. \qedhere
\]
\end{proof}
\begin{figure}[H]
    \centering
    \begin{tikzpicture}
        \draw (-1.5,0.5) node [above left] {$i-1$} -- (0,1) node [above] {$i$} -- (1.5,0.5) node [above right] {$i+1$};
        \draw (0,-1) node [below] {$k$};
        \draw (-1.5,0.5) -- (0,-1) -- (1.5,0.5);
        \draw [dashed] (-1.5,0.5) -- (1.5,0.5);
        \draw (0,-1) -- (0,1);
        \node at (-0.2,0.7) [] {$\alpha$};
        \node at (0.2,0.7) [] {$\beta$};
    \end{tikzpicture}
    \caption{Local Picture near the vertex $i$.}
    \label{fig:}
\end{figure}
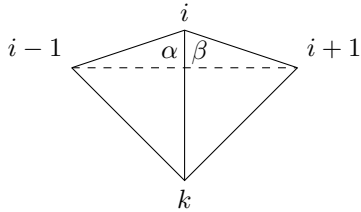

In the remaining portion of this subsection, we consider the cluster localization of $\Gr_{2,n}^\circ$ at the cluster variable $\Delta_{2,n}$ and $\Delta_{i-1,i+1}$ for $1< i<n$. Note that by Theorem \ref{thm: cluster localization} and Lemma \ref{lem: Gr(2,n) is acyclic}, we know that these cluster localizations of $\Gr_{2,n}^\circ$ are themselves cluster varieties.

\begin{prop}\label{prop: cluster localization} The cluster localization of $\Gr_{2,n}^\circ$ at $\Delta_{2,n}$ coincides with the non-vanishing locus $\{a_1\neq 0\}$, and for $1< i<n$, the cluster localization of $\Gr_{2,n}^\circ$ at $\Delta_{i-1,i+1}$ coincides with the non-vanishing locus $\{a_i\neq 0\}$.
\end{prop}
\begin{proof} This follows directly from (2) of Proposition \ref{prop: a_i and p_i in terms of Pluckers}.
\end{proof}

\begin{prop} Consider $\Gr_{2,n-1}^\circ$ as the vanishing locus of $\left(B(b_1)D(r_1)\cdots B(b_{n-2})D(r_{n-2})\right)_{11}$. Set $r_0:=r_{n-1}$, which is non-zero and can be determined from $b_1,r_1,\dots, b_{n-2},r_{n-2}$ in a manner similar to Equation \eqref{eq: p_n}. For each $1\leq i<n$, we define a map $\phi_i:\Gr_{2,n-1}^\circ\times (\bF^\times)^2_{u,v}\longrightarrow \Gr_{2,n}^\circ$ by
\[
\phi_i^*(a_j)=\left\{\begin{array}{ll}
    b_j & \text{if $j<i-1$,} \\
    b_{i-1}+\dfrac{v}{ur_{i-1}} & \text{if $i>1$ and $j=i-1$,} \\
    \dfrac{r_{i-1}}{uv} & \text{if $j=i$,} \\ 
    b_i+\dfrac{u}{vr_{i-1}} & \text{if $i<n-1$ and $j=i+1$,} \\
    b_{j-1} & \text{if $j>i+1$}.
\end{array}\right. \quad \quad \quad 
\phi_i^*(p_j)=\left\{\begin{array}{ll}
    r_j & \text{if $j<i-1$,} \\
    u & \text{if $j=i-1$,} \\
    v & \text{if $j=i$,} \\
    r_{j-1} & \text{if $j>i$.}
\end{array}\right.
\]
After adding a $3$-cycle $u\longrightarrow r_{i-1}\longrightarrow v\longrightarrow u$ of frozen vertices to each cluster seed of $\Gr_{2,n-1}^\circ$, $\phi_i$ becomes an isomorphism of cluster varieties onto its image, which is the cluster localization $\{a_i\neq 0\}$.
\end{prop}
\begin{proof} Recall that the cluster structure on $\Gr_{2,n}^\circ$ is captured by triangulations of a labeled $n$-gon. Following this combinatorial recipe, we see that the cluster localization $\{a_i\neq 0\}$ is the union of cluster charts corresponding to triangulations that contain the diagonal connecting vertices $i-1$ and $i+1$ in the $n$-gon (or the diagonal connecting $2$ and $n$ in the case when $i=1$). Such a diagonal cuts the $n$-gon into two parts, a triangle and an $(n-1)$-gon, and any triangulation of the $n$-gon containing this diagonal must induce a triangulation on the $(n-1)$-gon as well. This shows that the localization $\{a_i\neq 0\}$ is isomorphic, as a cluster variety, to a product of $\Gr_{2,n-1}^\circ$ (the $(n-1)$-gon part) and a $2$-dimensional frozen algebraic torus $(\bF^\times)^2_{p_{i-1},p_i}$. Now to see that $\phi_i$ is an isomorphism of cluster varieties onto its image, it suffices to consider the following local picture near the vertex $i$ (Figure \ref{fig: local picture for localization}) of the labeled $n$-gon, and observe that the pull-back map $\phi_i^*$ agrees with the identity in Lemma \ref{lem: potential function}.
\end{proof}

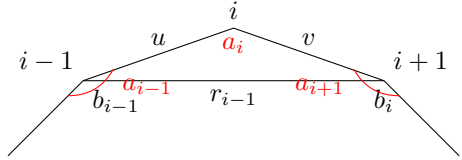
\begin{figure}[H]
    \centering
    \begin{tikzpicture}
        \draw (-3,0) -- (-2,1) -- node [above] {$u$} (0,1.7) node [below, red] {$a_i$} -- node [above] {$v$} (2,1) -- (3,0);
        \node at (-2,1) [above left] {$i-1$};
        \node at (0,1.7) [above] {$i$};
        \node at (2,1) [above right] {$i+1$};
        \draw (-2,1) node [below right] {$b_{i-1}$} -- node [below] {$r_{i-1}$} (2,1) node [below] {$b_i$};
        \draw [red] (-1.6,1.15) node [below right] {$a_{i-1}$} to [bend left] (-2.2,0.8);
        \draw [red] (1.6,1.15) node [below left] {$a_{i+1}$} to [bend right] (2.2,0.8);
    \end{tikzpicture}
    \caption{Local picture for the cluster localization $\{a_i\neq 0\}$.}
    \label{fig: local picture for localization}
\end{figure}

\section{Cluster Structure on the Ungraded Augmentation Variety \texorpdfstring{$\Aug_u(\Lambda[n_1,\dots, n_k])$}{}}\label{sec:cluster}

In this section we define a cluster structure on the ungraded augmentation variety $\Aug_u(\Lambda[n_1,\dots, n_k])$ and describe its relation to decomposable exact Lagrangian fillings of $\Lambda([n_1, \ldots, n_k])$. We first define the initial cluster seed and show that it is a really full-rank acyclic cluster variety. We then construct a set of exact Lagrangian fillings of $\Lambda([n_1, \ldots, n_k])$ so that each filling determines a cluster seed, and the dg-algebra map induced by said cobordisms agrees with cluster localization.

\subsection{Initial Cluster Seed}\label{sub:initial_seed}

\begin{thm}\label{thm: cluster structure on augmentation variety} The ungraded augmentation variety $\Aug_u(\Lambda[n_1,\dots, n_k])$ for $k\geq 2$ is a cluster variety with an initial cluster seed given by the disjoint union of the following quivers and initial cluster variables (the boxed cluster variables are frozen):
\[
K_1(a_1)\rightarrow K_2(a_1,a_2)\rightarrow K_3(a_1,a_2,a_3)\rightarrow \cdots \rightarrow K_{n_1-2}(a_1,\dots, a_{n_1-2})\rightarrow \boxed{K_{n_1-1}(a_1,\dots, a_{n_1-1})},
\]
\[
K_1(a_{m_{(i-1)}+2})\rightarrow K_2(a_{m_{(i-1)}+2},a_{m_{(i-1)}+3})\rightarrow  \cdots \rightarrow K_{n_i-3}(a_{m_{(i-1)}+2},\dots, a_{m_i-2})\rightarrow \boxed{K_{n_i-2}(a_{m_{(i-1)}+2},\dots, a_{m_i-1})}
\]
for each $1<i<k$, and
\[
K_1(a_{m_{k-1}+2})\rightarrow K_2(a_{m_{k-1}+2},a_{m_{k-1}+3})\rightarrow  \cdots \rightarrow K_{n_k-2}(a_{m_{k-1}+2},\dots, a_{m_k-1})\rightarrow \boxed{K_{n_k-1}(a_{m_{k-1}+2},\dots, a_{m_k})}.
\]
\end{thm}

\begin{rmk}
    Note that in the $k=1$, case, the theorem is true, replacing $n_1-1$ index in the first quiver by $n_1$~\cite{Gao_shen_weng_2}.
\end{rmk}

\begin{proof} By Proposition \ref{prop:dgahtpy}, we see that $\Aug_u(\Lambda[n_1,\dots, n_k])$ is a product of affine varieties given by the equations $K_{n_1}(a_1,\dots, a_{n_1})=0$, $K_{n_i-1}(a_{m_{(i-1)}+2},\dots, a_{m_i})=0$ for each $1<i<k$, and the inequality $K_{n_k-1}(a_{m_{k-1}+2},\dots, a_{m_k})\neq 0$. Since a product of cluster varieties is also a cluster variety (Lemma \ref{lem: product of cluster varieties}), it suffices to show that each of the equations and the inequality above defines a cluster variety with the prescribed initial cluster seed.

Let us consider the equation $K_{n_1}(a_1,\dots, a_{n_1})=0$ first. Set $n=n_1+1=m_1+1$ and consider the subvariety $V\subset\Gr_{2,n_1+1}^\circ$ cut out by $p_1=p_2=\cdots =p_{n_1-1}=p_{n_1+1}=1$. By Proposition \ref{prop: a_i and p_i in terms of Pluckers} and Lemma \ref{lem: deleting frozen vertex} we know that this subvariety $V$ is also a cluster variety. We claim that $V$ is precisely given by the vanishing locus of $K_{n_1}(a_1,\dots, a_{n_1})=0$. To see this, recall from Proposition \ref{prop: affine variety} that $\Gr_{2,n_1+1}^\circ$ is defined by Equation \eqref{eq: defining equation for Gr_2,n}. After setting $p_i=1$ for all $i\neq n_1$, the equation becomes 
\[
\left(B(a_1)B(a_2)\cdots B(a_{n_1})D(p_{n_1})\right)_{11}=K_{n_1}(a_1,\dots, a_{n_1})p_{n_1}=0.
\] 
Since $p_{n_1}$ is non-zero, the above equation holds if and only if $K_{n_1}(a_1,\dots, a_{n_1})=0$, which proves our claim. Note that although we did not set $p_{n_1}$ to be $1$, it is not algebraically independent from the $a_i$'s because we it appears as a nonzero scalar so that $v_{n_1+1}=B(a_1)B(a_2)\cdots B(a_{n_1})D(p_{n_1})\begin{pmatrix} 1  \\ 0 \end{pmatrix}=\begin{pmatrix} 0 \\ 1 \end{pmatrix}$, which is required by $p_{n_1+1}=\Delta_{1,n_1+1}=1$. To obtain the prescribed cluster seed, we take the triangulation of the $n$-gon consisting of diagonals connecting the vertex $n_1+1$ with vertices $2,3,\dots, n_1-1$. By following the setup in the proof of Proposition \ref{prop: affine variety}, we see that the initial cluster variable associated with the diagonal connecting the vertex $n_1+1$ with vertex $j=2,3,\dots, n_1-1$ is (c.f. Lemma \ref{lem: braid matrix prod})
\[
\Delta_{j,n_1+1}=\det(v_j,v_{n_1+1})=\left(B(a_1)B(a_2)\cdots B(a_{j-1})\right)_{11}=K_{j-1}(a_1,\dots, a_{j-1}).
\]
There is one frozen cluster variable remaining in the initial cluster seed, namely
\[
p_{n_1}=\Delta_{n_1,n_1+1}=\left(B(a_1)B(a_2)\cdots B(a_{n_1-1})\right)_{11}=K_{n_1-1}(a_1,\dots, a_{n_1-1}).
\]

The equations $K_{n_i-1}(a_{m_{(i-1)}+2},\dots, a_{m_i})=0$ for $1<i<k$ can be dealt with in a similar manner by relating it to the Grassmannian $\Gr_{2,n_i}^\circ$, and we get the prescribed cluster seeds as stated in the proposition.

Lastly, for the inequality $K_{n_k-1}(a_{m_{k-1}+2},\dots, a_{m_k})\neq 0$, let us first shift the indices by $-(m_{k-1}+1)$ to make the notation simpler, which yields $K_{n_k-1}(a_1,\dots, a_{n_k-1})\neq 0$. We then apply Proposition \ref{prop: alternative affine variety} with $n=n_k+1$ and see that this non-vanishing locus is isomorphic to the subvariety of $\Gr_{2,n_k+1}^\circ$ cut out by $\{p_1=p_2=\cdots=p_{n_k-1}=p_{n_k+1}=1\}$. Thus, we may again conclude from Lemma \ref{lem: deleting frozen vertex} that this inequality defines a cluster variety. To get an initial cluster seed, we again consider the triangulation of the $(n_k+1)$-gon consisting of diagonals connecting the vertex $n_k+1$ with vertices $2,\dots, n_k-1$; similar to the $n_1$ case, we again obtain the initial mutable cluster variables
\[
\Delta_{j,n_k+1}=\det(v_j,v_{n_k+1})=\left(B(a_1)B(a_2)\cdots B(a_{j-1})\right)_{11}=K_{j-1}(a_1,\dots, a_{j-1})
\]
for $j=2,3,\dots, n_k-1$. There is again one frozen cluster variable remaining in the initial cluster seed, and it is again
\[
\Delta_{n_k,n_k+1}=\left(B(a_1)B(a_2)\cdots B(a_{n_k-1})\right)_{11}=K_{n_k-1}(a_1,\dots, a_{n_k-1}).
\]
The prescribed initial cluster seed for the $n_k$ block can then be recovered with a shift by $m_{k-1}+1$.
\end{proof}

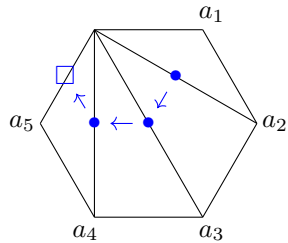
\begin{figure}[H]
    \centering
    \begin{tikzpicture}[scale=0.8]
    \foreach \i in {1,...,6}
    {
    \draw (\i*60:1.8) -- (\i*60+60:1.8);
    }
    \foreach \i in {1,...,5}
    {
    \node at (120-60*\i:2.1) [] {$a_\i$};
    }
    \node [blue] (0) at (150:1.6) [] {$\square$};
    \foreach \i in {1,2,3}
    {
    \draw (120:1.8) -- node (\i) [blue] {$\bullet$} (-180+60*\i:1.8);
    \pgfmathtruncatemacro{\j}{\i-1};
    \draw [->,blue] (\i) -- (\j);
    }
    \end{tikzpicture}
    \caption{Examples of the initial cluster seed.}
\end{figure}

\begin{cor} \label{cor: augmentation variety is really full-rank}$\Aug_u[n_1,\dots, n_k]$ is a really full-rank acyclic cluster variety.
\end{cor}
\begin{proof} The acyclicity follows from the initial cluster seed in Theorem \ref{thm: cluster structure on augmentation variety}. For really full-rank-ness (Definition \ref{defn: full-rank}), it suffices to prove that the cluster seed for each block is really full-rank (Lemma \ref{lem: product of cluster varieties}). Suppose the quiver for a block is of the form $1 \rightarrow 2 \rightarrow \cdots \rightarrow n\rightarrow\boxed{n+1}$ (the boxed vertex is frozen). Then the column vectors in the $n\times (n+1)$ submatrix of interest are $v_1=-e_2$, $v_i=e_{i-1}-e_{i+1}$ for $1<i<n$, $v_n=e_{n-1}$, and $v_{n+1}=-e_n$. We can do a backward induction to show that all $e_i$ with $1\leq i\leq n$ are in the span of these column vectors: for the base case, $e_n$ and $e_{n-1}$ are clearly in the span; but then $e_i=v_{i+1}+e_{i+2}$ would also be in the span of $v_i$'s by induction. Therefore we can conclude that the quiver of each block is really full-rank, and hence $\Aug_u[n_1,\dots, n_k]$ is a really full-rank cluster variety.
\end{proof}

\begin{rmk} Note that we could have used the equations in Corollary \ref{cor: dgahtpy} to obtain a cluster structure on $\Aug_u(\Lambda[n_1,\dots, n_k])$, and the resulting cluster structure will be of the same cluster type and quasi-cluster equivalent to the one we stated in Theorem \ref{thm: cluster structure on augmentation variety} (see \cite{FraserQuasi} and \cite[Appendix]{Casals_weng} for more on quasi-cluster equivalences). We choose the one stated in Theorem \ref{thm: cluster structure on augmentation variety} because of its connection to the ruling stratification, which we will discuss in Section \ref{sec:ruling}.
\end{rmk}

\subsection{Exact Lagrangian Fillings and Cluster Seeds}

In this section, we construct and distinguish (typically non-orientable) exact Lagrangian fillings of $\Lambda[n_1, \ldots, n_k]$ for $n_i\geq 2.$ We begin by first establishing which Reeb chords in $\Lambda[n_1, \ldots, n_k]$ can be pinched. To aid us in this process, we introduce ``dips" near certain crossings to isolate them and ensure that they are proper.
This technique, first introduced by Fuchs in \cite{Fuchs03}, is often used to simplify dg-algebra computations involving particularly complicated Reeb chords; see e.g. \cite{Sabloff_2005, FuchsRutherford}. For a more detailed discussion involving a computation similar to the case we consider here, see \cite[Section 3.2]{Gao_shen_weng_2}.

\begin{lem}\label{lem:contractible}

The Reeb chords in the set $\{a_i~|~1\leq i\leq m_k\}\backslash\{a_{m_i+1}~|~1\leq i\leq k-1\}$ 
of $\Lambda[n_1, \ldots, n_k]$ are contractible and proper for 
$n_i\geq 2$, $1\leq i\leq k$ after adding a dip to the right of each Reeb chord $a_{m_i}$.
\end{lem}
\begin{proof}

Note that we do not consider the Reeb chords $a_{m_i}+1$ for $2\leq i \leq k-1$, as they have nonzero differential and are not contractible nor proper in general. For the remaining Reeb chords, we can use the local model of the pinch move given in \cite[Section 3.1]{Hughes2023}, to see that they are all contractible. The local model described in loc. cit. gives a description of the pinch move in the front projection -- see Figure \ref{fig:D4minus} -- and its Lagrangian resolution,  verifying that any Reeb chord appearing to the right of another in a block can be made arbitrarily small by a Legendrian isotopy. For the case of $a_1,$ there is an additional local model verifying contractibility.

To see that the specified Reeb chords are proper, we can observe that for Reeb chord $a_j$ appearing between the first and last of any block, the only disks with positive corners at $a_j$ are the disks with positive corners $a_j$ and $a_{j\pm 1}$ filling in the bigons formed by the two adjacent crossings. For the case of $a_1,$ the only such bigon has positive corners $a_1$ and $a_2$. For the crossings at the right of each block, $a_{m_i}$ for $1\leq i\leq k$ 
we add a dip to right of $a_{m_i}$; that is, we perform a Reidemeister two move between the two strands crossing at $a_{m_i}$ that introduces two crossings to the right of $a_{m_i}$. This then reduces the crossing $a_{m_i}$ to the case of the middle crossings considered above at the cost of adding two new generators $x_i, y_i$ to the dg-algebra for $1\leq i\leq k$. After pinching at $a_{m_i}$ and removing the dip, this has the effect of modifying the differential of $b_1, b_2,$ and $a_{m_j+1}$ for $i<j\leq k$. As we show in \ref{sub: equations}, none of these Reeb chords appear in the equations cutting out the ungraded augmentation variety; in particular, there exists a dg-algebra homotopy for each crossing $a_{m_i+1}$ for $2\leq i \leq k$.
\end{proof}

\begin{rmk}
For the purposes of this paper, we only apply Lemma~\ref{lem:contractible} when $n_1\geq2, n_2\geq 3, \ldots, n_{k-1}\geq 3,$ and $n_k\geq 2$. If $n_i=2$ for $2\geq i\geq k-1$ and we pinch one of the chords, we will obtain a stabilized Legendrian which does not admit any fillings.
\end{rmk}\label{rem:stabilized}

\begin{defn}\label{defn:admissible}
Let $\Lambda[n_1, \ldots, n_k]$ be in Legendrian rational form. Let $P=(i_1, \ldots, i_l)$ for $l\geq 1$ and $i_1, \ldots, i_l\in\{1, \ldots, m_k\}$ denote a pinching sequence where the Reeb chords $a_{i_1}, \ldots, a_{i_l}$ of $\Lambda[n_1, \ldots, n_k]$ are pinched in the order they are listed. Let $\Lambda_j[n_1, \ldots, n_k]$ denote the Legendrian knot obtained after pinching the Reeb chords $a_{i_1}, \ldots, a_{i_j}$ of $\Lambda[n_1, \ldots, n_k]$. An \emph{admissible pinching sequence} $P=(i_1, \ldots, i_l)$ is a pinching sequence such that for each $j=0, \ldots, l-1$,~$\Lambda_{j}[n_1, \ldots, n_k]$ is in Legendrian rational form, and $a_{i_{j+1}}$, is a contractible and proper Reeb chord of $\Lambda_j[n_1, \ldots, n_k]$. 

Let $l=m_k-2k+2$, and let $P$ be an admissible pinching sequence for $\La[n_1, \ldots, n_k]$. After applying the sequence of saddle cobordisms $L_P$ corresponding to $P$, we obtain a Legendrian $2$-bridge link with base points that is of type $[1,2,2,\dots, 2,1]$, which is Legendrian isotopic to a split union $\Lambda_-$ of two max-tb unknots with base points. Note that in the $k=1$ case, $l=n_1$ and the resulting of pinching all $n_1$ Reeb chords of $\La[n_1]$ is a split union of unknots. Similarly, in the $k=2$ case, any admissible pinching sequence yields $\La[1, 1]$, which is again Legendrian isotopic to a split union of unknots. Let $L_0$ denote the exact Lagrangian filling given by concatenating the trace of the Legendrian isotopy taking $\Lambda[1, 2, \ldots, 2,1]$ to the split union of two Legendrian max-tb unknots with the two minimum cobordisms filling in said unknots. We abuse notation and denote the concatenation of $L_P$ and $L_0$ by $L_P$. 

We say that an exact Lagrangian filling of $\La[n_1, \ldots, n_k]$ is \emph{constructed by an admissible pinching sequence} $P$ if it is the concatenation of $L_P$ and $L_0$. See Figure~\ref{fig:pinchingsequence} for an example of an exact Lagrangian filling of $\Lambda[3,3,3,3,2]$ constructed by an admissible pinching sequence.
\end{defn}

\begin{prop}\label{prop:admissiblefilling} 

Any exact Lagrangian filling of $\Lambda[n_1, \ldots, n_k]$ constructed by an admissible pinching sequence imposes the relations
$t_1=w_1$ and $t_2=w_2$, where $w_1$ and $w_2$ are Laurent monomials in $s_1, \ldots, s_l$. As a result, $\Aug_u(\Lambda_-)$ is isomorphic to $(\bF^\times)^l$ for $l=m_k-2k+2$. Moreover, all such exact Lagrangian fillings are smoothly isotopic.
\end{prop}
\begin{proof} 
The Legendrian $\Lambda[1, 2, \ldots, 2,1]$ is Legendrian isotopic to split union of two Legendrian max-tb unknots which we will denote by $\Lambda^{(1)}$ and $\Lambda^{(2)}.$
 Observe that all of the Reeb chords of $\La[n_1, \ldots, n_k]$ except for $b_1$ and $b_2$ begin on $\Lambda_-^{(1)}$ and end on $\Lambda_-^{(2)}$ or vice versa. Thus, each pinch move at any of the Reeb chords $a_i$ that are proper and contractible (see Lemma~\ref{lem:contractible}) contributes one co-oriented base point $s_i$ for $i=1, \ldots, l$ to each link component $\Lambda^{(1)}$ and $\Lambda^{(2)}$. Therefore, $\partial(b_1)=t_1+w_1$ and $\partial(b_2)=t_2+w_2$, where $w_1$ and $w_2$ are products of $s_1^{\pm1}, \ldots, s_l^{\pm1}$. Both unknots are filled with a minimum cobordism, so we then have the relations $\e_L(t_1)=w_1$ and $\e_L(t_2)=w_2.$

Every exact Lagrangian filling of $\La[n_1, \ldots, n_k]$ constructed by an admissible pinching sequence $P$ has the same Euler characteristic and normal Euler number (the latter follows from~\cite[Proposition 1.1]{CCPRSY}). Resolving a crossing is equivalent to attaching a 1-handle. For any two pinching sequences $P_1$ and $P_2$ that only differ in the order of Reeb chords pinched, the exact Lagrangian fillings are clearly smoothly isotopic as the order in which 1-handles are attached does not change the smooth type of the filling. Finally, suppose that both $a_j$ and $a_k$ are crossings contained in one of the blocks $n_i$ of $\La[n_1, \ldots, n_k]$. One can smoothly isotope the $1$-handle attached at a crossing $a_j$ so that it resolves the crossing $a_k$ instead. Thus, all exact Lagrangian fillings of $\La[n_1, \ldots, n_k]$ constructed by an admissible pinching sequence $P$ are smoothly isotopic.
\end{proof}

\begin{figure}[h]
	\centering
	\begin{tikzpicture}[scale=1.5]
		\node[inner sep=0] at (0,0) {\includegraphics[width=15 cm]{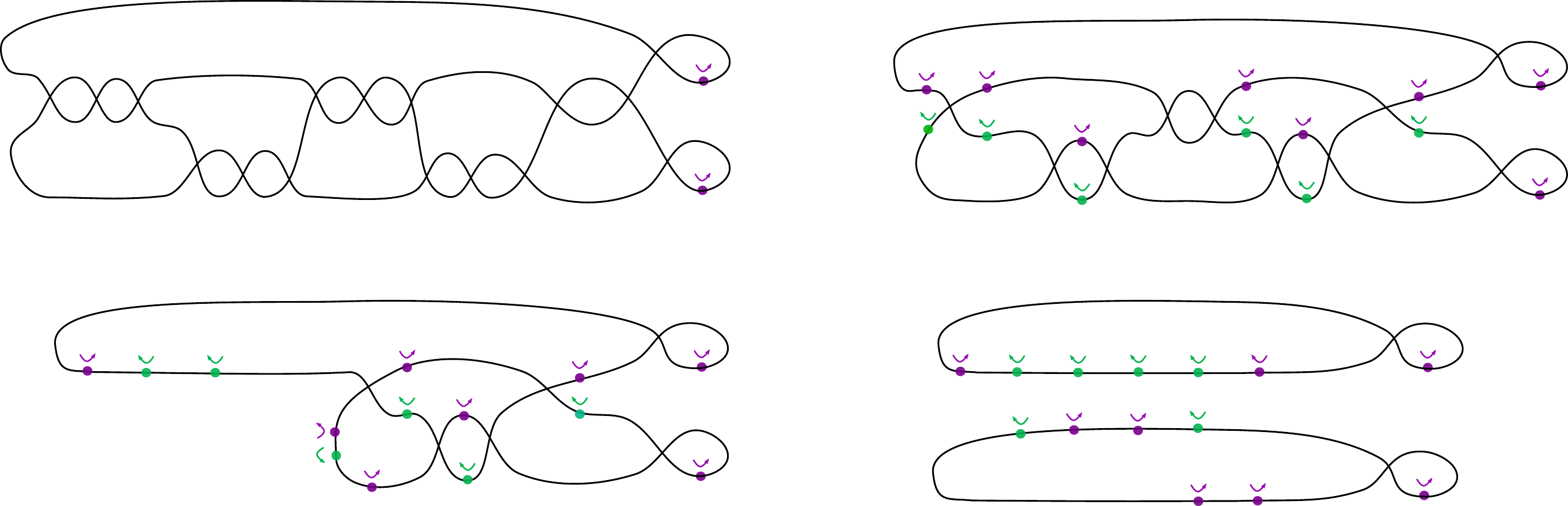}};
        \node at (-0.5,0.9){$t_2$};
		\node at (-0.5,0.2){$t_1$};
        \node at (4.85,0.9){$t_2$};
		\node at (4.85,0.2){$t_1$};
        \node at (4.15,-1.8){$t_1$};
		\node at (4.15,-1){$t_2$};
        \node at (0.7,0.9){$s_1$};
		\node at (0.7,0.7){$s_1$};
        \node at (1.3,1.3){$s_2$};
		\node at (1.35,0.55){$s_2$};
        \node at (2.9,1.3){$s_4$};
		\node at (2.9,0.55){$s_4$};
        \node at (4.05,1.3){$s_6$};
		\node at (4.05,0.55){$s_6$};
        \node at (3.3,1){$s_5$};
		\node at (3.35,0.2){$s_5$};
        \node at (1.9,0.95){$s_3$};
		\node at (1.9,0.2){$s_3$};
        \node at (-4.5,-0.9){$s_1$};
		\node at (-3.15,-1.4){$s_1$};
        \node at (-4.1,-0.9){$s_2$};
		\node at (-3.15,-1.1){$s_2$};
        \node at (-2.4,-0.5){$s_4$};
		\node at (-2.4,-1.2){$s_4$};
        \node at (-1.3,-0.55){$s_6$};
		\node at (-1.3,-1.2){$s_6$};
        \node at (-2.05,-0.805){$s_5$};
		\node at (-2.05,-1.65){$s_5$};
        \node at (-3.6,-0.9){$s_3$};
		\node at (-2.6,-1.65){$s_3$};
        \node at (-0.5,-1.6){$t_1$};
		\node at (-0.5,-0.9){$t_2$};
        \node at (3,-1.8){$s_5$};
        \node at (2.6,-1.8){$s_3$};
        \node at (1.5,-1.3){$s_1$};
        \node at (1.9,-1.3){$s_2$};
        \node at (2.2,-1.3){$s_4$};
        \node at (2.6,-1.3){$s_6$};
        \node at (3,-0.9){$s_6$};
        \node at (2.6,-0.9){$s_5$};
        \node at (2.2,-0.9){$s_4$};
        \node at (1.9,-0.9){$s_3$};
        \node at (1.5,-0.9){$s_2$};
        \node at (1.2,-0.9){$s_1$};
		\node at (0.2,1){$\boldsymbol{\rightarrow}$};
        \node at (0.2,-1){$\boldsymbol{\rightarrow}$};
         \node at (0,0){$\boldsymbol{\swarrow}$};
  \end{tikzpicture}
	\caption{An exact Lagrangian filling constructed with an admissible pinching sequence for $\Lambda[3,3,3,3,2].$ Pinch in order $a_1,a_3,a_6,a_8,a_{11},a_{14}$, then perform a sequence of Reidemeister II moves and fill both unknots with minimum cobordisms from which we obtain the relations $t_1=s_1s_3s_5s_6(s_2s_4)^{-1}$, and $t_2=s_1s_6(s_2s_3s_4s_5)^{-1}.$} 
	\label{fig:pinchingsequence}
\end{figure}

\begin{rmk}Note that for the case of $(2,n)$ torus links, $\Lambda[n]$ has oriented exact Lagrangian fillings so one can match the co-orientation on $t_1, t_2$ to the orientation of the knot and $w_1=w_2$. One then obtains the relation $t_1t_2^{-1}=1.$
\end{rmk}

\begin{thm}\label{thm: cluster torus chart from pinching sequence} The induced morphism $\Phi_{L_P}:\Aug_u(\Lambda_-)\longrightarrow \Aug_u(\Lambda_+)$ is an open embedding of an algebraic torus onto a cluster torus chart on $\Aug_u(\Lambda_+)$ (c.f. Definition \ref{defn: cluster variety}).
\end{thm}
\begin{proof} 

Let $\Phi_{L_i}$ denote the dg-algebra morphism induced by the saddle cobordism $L_i$ given by pinching a contractible and proper Reeb chord $a_i$. It suffices to only consider the Reeb chords that appear in the defining equations and inequality in Proposition \ref{prop:dgahtpy}, so we can ignore the Reeb chords $b_1, b_2, a_{m_i+1}$ for $1\leq i<k-1$. Then, analogously to the case shown in Example~\ref{ex:pinch}, $\Phi_{L_i}$ the identity for all Reeb chords except for itself and the adjacent ones: $a_{i-1}$ (if $i>1$), $a_i$, and $a_{i+1}$ (if $i\neq m_j$ for any $1\leq j\leq k$). They are mapped according to 
\[
\Phi_{L_i}(a_{i-1})=a_{i-1}+s_i^{-1}, \quad \Phi_{L_i}(a_i)=s_i, \quad \Phi_{L_i}(a_{i+1})=a_{i+1}+s_i^{-1}.
\]
Note that these assignments coincide with the cluster localization in Proposition \ref{prop: cluster localization}, with $u$ (resp. $v$) being replaced by the product of base points (frozen cluster variables) that are already sandwiched between $a_{i-1}$ and $a_i$ (resp. $a_i$ and $a_{i+1}$) before the pinching, and with $r_{i-1}$ being replaced by $us_iv$. This shows that each pinching induces a cluster localization of the cluster variety $\Aug_u(\Lambda_+)$. After cluster localizing at all the $l$ many mutable vertices, we obtain a cluster torus chart with local coordinates $s_i$'s, which can then be naturally identified with the ungraded augmentation variety $\Aug_u(\Lambda_-)$. 
\end{proof}

\begin{prop}\label{prop: every cluster chart is achieved} Every cluster torus chart in $\Aug_u(\Lambda_+)$ can be obtained as the image of $\Phi_{L_P}$ for some admissible pinching sequence $P$.
\end{prop}
\begin{proof} Recall from Lemma \ref{lem: product of cluster varieties} that every cluster seed in a product of cluster varieties is a product of cluster seeds. Thus, it suffices to show that every cluster seed in each block of the $2$-bridge Legendrian link can be obtained through an admissible pinching sequence. To do so, we can homotope the $n$-gon described in the proof of Theorem \ref{thm: cluster structure on augmentation variety} so that all its vertices line up in a straight line, with the unique ``frozen'' side represented by a large semicircular arc below the line. More specifically, in the case of the first block, we introduce a blank vertex as the leftmost vertex of the $n_1+1$-gon and the remaining vertices are labeled by $a_1,\ldots, a_{n_1}$. Similarly, the vertices of the $n_i$-gon corresponding to the $i$th block are labeled by $a_{m_{i-1}+1}, \ldots, a_{m_i}$ and the vertices of the $n_k$-gon corresponding to the $k$th block are labeled by $a_{m_{k-1}+1},\ldots, a_{m_k}$. Note that, while the introduction of the blank vertex is necessary to match the continuant polynomial coming from the first block to the equation cutting out $\Gr^\circ_{2, n_1+1}$, as described in the proof of Theorem \ref{thm: cluster structure on augmentation variety}, it is always one of the vertices on either side of the ``frozen" edge and thus never appears as part of a pinching sequence; likewise for $a_{m_{i-1}+1}$, $1< i< k$. Following this homotopy, we represent all of the remaining sides of the $n$-gon by small semicircular arcs below the line connecting vertices $a_j$ and $a_{j+1}$. For any given triangulation $\tau$, we similarly draw the diagonals present in $\tau$ with semicircular arcs. We then go from top to bottom through the diagonals present in $\tau$, and read off the vertex opposite to each of these diagonals. (If there is ambiguity in determining which one among a few diagonals is higher than the others, then any choice of ordering these diagonals will work.) The resulting list gives an admissible pinching sequence for that block. The total admissible pinching sequence can be any shuffling of these pinching sequences for a given block. See Figure \ref{fig: triangulation} for an example for the first block with $n_1=6$ as well as~\cite[Section 2.3.1]{Hughes2023} for an explicit method for relating pinching sequences of $(2, n)$ torus links to triangulations. 
\end{proof}

\begin{figure}[H]
    \centering
    \begin{tikzpicture}[scale=0.8]
        \foreach \i in {1,...,6}
        {
        \draw (-3+\i,0) node [above] {$a_\i$} arc (0:-180:0.5);
        }
        \draw (-3,0) arc (-180:0:3);
        \draw (-2,0) arc (-180:0:1);
        \draw (0,0) arc (-180:0:1);
        \draw (-2,0) arc (-180:0:2);
        \draw (-3,0) arc (-180:0:2.5);
        \node [blue] at (0,-3) [] {$\square$};
    \end{tikzpicture}
    \caption{By following the procedure described in the proof of Proposition \ref{prop: every cluster chart is achieved}, the corresponding cluster seed can be induced by the admissible pinching sequence $(2,4,3,1,5)$ or $(4,2,3,1,5)$.}
    \label{fig: triangulation}
\end{figure}

\begin{exmp} The initial cluster seed described in Theorem \ref{thm: cluster structure on augmentation variety} can be obtained from the admissible pinching sequence $(1,\dots, {m_1-1},{m_1+2},\dots, {m_2-1},\dots, m_{k-2}+2,\dots, m_{k-1}-1,m_{k-1}+2,\dots, m_k)$. We call this the \emph{left-to-right pinching sequence}. However, this is not the only admissible pinching sequence that gives rise to this cluster seed. The phenomenon of different pinching sequences giving rise to the same cluster seed has been observed in previous works, e.g., \cite{Pan, Gao_shen_weng_2, CGGLSS}. 
\end{exmp}

\begin{rmk} The procedure described in the proof of Proposition \ref{prop: every cluster chart is achieved} does not produce all pinching sequences that give rise to the same triangulation. Given an admissible pinching sequence $(i_1,i_2,\dots, i_l)$ of Reeb chords in the same block, we can draw the corresponding triangulation by ``chopping off'' corners of the polygon in the order specified by the pinching sequence; see \cite[Section 2.3.1]{Hughes2023} for a more explicit bijection. As an example, the pinching sequence $(2,4,6,3,1)$ produces the same triangulation shown in Figure \ref{fig: triangulation}. Note that in this example, we consider a different set of Reeb chords but retain our labelings of the vertices of the polygon, so that the only component that changes in this recipe is which corners we are allowed to remove.  
\end{rmk}

In the case of Legendrian $(2, n)$-torus links, Ekholm, Honda, and K\'alm\'an show that either order of pinching at any two contractible Reeb chords $a_i$ and $a_j$ yield Hamiltonian isotopic cobordisms if $|j-i|\geq 2$ \cite[Claim 8.1]{ehk}. They show this by observing that such pairs of Reeb chords are \emph{simultaneously contractible}, i.e. that the lengths of both Reeb chords can simultaneously be made arbitrarily small by some Legendrian isotopy. We adapt their argument to the 2-bridge link setting with the following lemma.
\begin{lem}\label{lem: simultaneously contractible}
 Any two contractible Reeb chords $a_i, a_j$ of $\La[n_1, \dots, n_k]$ are simultaneously contractible if and only if $|j-i|\geq 2$.
\end{lem}

\begin{proof}
Let $a_i, a_j$ be two Reeb chords of $\La[n_1, \dots, n_k]$ satisfying $|j-i|\geq 2$ and consider the front projection pictured in Figure \ref{fig:front}. By \cite[Proposition 3.1]{Hughes2023}, the local model pictured in Figure \ref{fig:D4minus} represents a one-handle attachment that is Hamiltonian isotopic to a pinching cobordism. Since $|j-i|\geq 2$, we can apply this local model to the two crossings simultaneously. The local model resolves the rightmost crossing of an adjacent pair so that we can always apply the local model to the chosen contractible Reeb chords $a_i$ and $a_j$. The Ng resolutions of the resulting family of fronts then gives an exact Lagrangian cobordism realizing the simultaneous pinching at $a_i$ and $a_j$, as desired. Note that we make no assumptions on the orientation of the strands before or after the pinch move so as to be able to use this local model for both orientable and non-orientable pinch moves.   

For the converse, note that when $a_i$ and $a_j$ are both contractible and $|j-i|=1$, they belong to the same block. In this case, the proof of Proposition \ref{prop: every cluster chart is achieved} implies that in the case $|j-i|=1$, swapping the order of pinching at $a_i$ or $a_j$ in the admissible pinching sequence induces distinct cluster charts. In particular, it follows from \cite[Section 5.3]{Hughes2023} that the two cluster charts induced by the pair of admissible pinching sequence fillings are related by a mutation.
\end{proof}

\begin{figure}[H]
\centering
\includegraphics[scale=0.2]{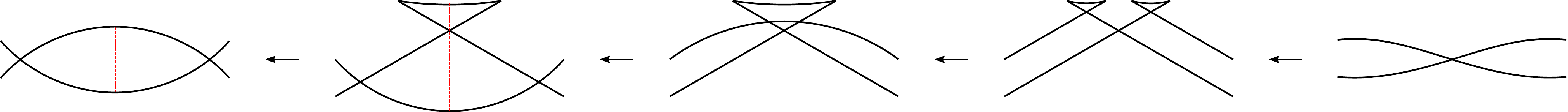}
\caption{A local model of a generic perturbation of the $D_4^-$ cobordism. The cobordism is Hamiltonian isotopic fixing the boundary to a pinching cobordism in the Lagrangian projection. The direction of the arrow indicates a cobordism from the concave end to the convex end.}	\label{fig:D4minus}\end{figure}

In the following proposition, we say that a set of Reeb chords is consecutive if the indices labeling them are consecutive; e.g. $a_1, a_2, a_3$ is consecutive, but $a_1, a_3, a_4$ is not.

\begin{prop}\label{prop: product of Catalan many fillings}
    Let $\mathcal{R}$ be a set of 
    $n_1-1$ consecutive proper and contractible  Reeb chords from the first block, $n_i-2$ consecutive proper and contractible Reeb chords from the $i$th block for $1<i<k$, and all $n_k-1$ proper and contractible Reeb chords from the last block.  
    There are precisely $C_{n_1-1}C_{n_2-2}\dots C_{n_{(k-1)}-2}C_{n_k-1}$ Hamiltonian isotopy classes of exact Lagrangian fillings of $\La[n_1, \dots, n_k]$ constructed by admissible pinching sequences of the Reeb chords in $\mathcal{R}$.
\end{prop}

\begin{proof}
   Fix an arbitrary set $\mathcal{R}$ of
   $n_1-1$ consecutive Reeb chords in the first block (these are all proper and contractible), 
   $n_i-2$ consecutive proper and contractible Reeb chords in the middle blocks, and the rightmost $n_k-1$ Reeb chords in the last block. Then any admissible pinching sequence of the Reeb chords in $\mathcal{R}$
   can be expressed by some ordering of these Reeb chords. Since we have chosen consecutive Reeb chords in each block, we can produce any triangulation of the $n$-gon, allowing us to apply  Proposition \ref{prop: every cluster chart is achieved} to conclude that there are at least $C_{n_1-1}C_{n_2-2}\dots C_{n_{(k-1)}-2}C_{n_k-1}$ distinct Hamiltonian isotopy classes of fillings. In the construction of the proof of Proposition \ref{prop: every cluster chart is achieved}, different pinching sequences inducing the same cluster chart are precisely those that differ by swapping the orders of pinching at pairs of contractible Reeb chords $a_i$ and $a_j$ with $|j-i|\geq 2$. Note that swapping the order of pinching any two simultaneously contractible Reeb chords is an exact Lagrangian isotopy, or equivalently, a Hamiltonian isotopy. Thus, by Lemma \ref{lem: simultaneously contractible}, every admissible pinching sequence of the set $\mathcal{R}$  inducing the same cluster chart is Hamiltonian isotopic.
\end{proof}

\begin{rmk}
In this proof, the fact that we can produce any triangulation of the $n_i$-gon crucially relies on the fact that we choose consecutive Reeb chords. If we choose non-consecutive Reeb chords, say we leave out $a_j$ within the block, then the diagonal connecting $a_{n_{i-1}+1}$ and $a_j$ is fixed in the triangulation, consequently yielding a smaller number of triangulations (which is still a product of Catalan numbers).
\end{rmk}

\begin{exmp}\label{ex:[n,2]} The numbers of orientable vs.\ non-orientable pinch moves are \textbf{not} Legendrian invariants. To see this, consider $2$-bridge knots of the form $\Lambda[n,2]$ (which are topologically twist knots). For all admissible pinching sequences for $\Lambda[n,2]$, the first pinch move is always non-orientable. If we choose to pinch $a_{n+2}$ first, then all remaining pinches are orientable pinch moves; in fact, after pinching $a_{n+2}$, the remaining link is Legendrian isotopic to a $(2,n-1)$ torus link, which has many orientable fillings. In contrast, if we choose to pinch $a_{n+2}$ last, then all pinch moves in the admissible pinching sequence are non-orientable. Moreover, admissible pinching sequences with different numbers of orientable vs. non-orientable pinch moves can correspond to the same cluster seed. For example, the pinching sequences $(1,2,\dots, n-1, n+2)$ and $(n+2,1,2,\dots, n-1)$ correspond to the same cluster seed, although they have different numbers of orientable vs. non-orientable pinch moves. Since the pair of Reeb chords $a_{n+2}$ and $a_i$ are simultaneously contractible for all $1\leq i \leq n-1$, the two pinching sequences yield Hamiltonian isotopic fillings. This Hamiltonian isotopy can be seen as an example of exchanging crosscaps with genera in the exact Lagrangian setting.
\end{exmp}

\begin{exmp}\label{ex:[4,4]}
As stated in Question~\ref{ques:non-orientable} it is not known whether any non-orientable exact Lagrangian filling is Hamiltonian isotopic to an exact Lagrangian filling constructed via a single non-orientable 1-handle attachment.
In a similar vein, one could ask the following far more restrictive question. Let $D$ denote a fixed front of $\Lambda$ and $L$ a decomposable exact non-orientable Lagrangian filling of $\Lambda$ constructed via a pinching sequence on $D$. Does there always then exist a decomposable  exact Lagrangian filling $L'$ Hamiltonian isotopic to $L$, constructed via a pinching sequence on $D$ with only one non-orientable pinch move? We can answer this latter question negatively. Consider the 2-bridge knot $\Lambda[4,4]$. Observe that pinching any of the crossings $a_1,\ldots, a_8$ of $\Lambda[4,4]$ is a non-orientable pinch move. After performing this first non-orientable pinch move we obtain either of $\Lambda[3,4]$ or $\Lambda[4,3]$ which are both the max-tb Legendrian $m(6_2)$ knot with Thurston-Bennequin number $-1$. Since the smooth $4$ ball genus of $m(6_2)$ is $1$, then by~\cite[Theorem 1.3]{chantraine}, the max-tb Legendrian $m(6_2)$ does not admit orientable exact Lagrangian fillings.  
\end{exmp}

\section{Ruling and Anticlique Stratifications of \texorpdfstring{$\Aug_u(\Lambda[n_1, \ldots, n_k]).$}{}}\label{sec:ruling}

In this section we describe two stratifications on $\Aug_u(\Lambda[n_1, \ldots, n_k])$ and show that they coincide. The first stratification arises from normal rulings of $\Lambda[n_1, \ldots, n_k]$, a Legendrian invariant related to augmentations, while the second arises from the viewing $\Aug_u(\Lambda[n_1, \ldots, n_k])$ as a really full-rank acyclic cluster variety. Finally, we use this result and a result of Rutherford~\cite{Rutherford06} relating normal rulings and a term of the Kauffman polynomial to show that the $\F_q$ point count of the cluster variety $\Aug_u(\Lambda[n_1, \ldots, n_k])$ encodes a term of the Kauffman polynomial of the topological link type of $\Lambda[n_1, \ldots, n_k]$. We begin by reviewing the normal ruling stratification on $\Aug_u(\Lambda[n_1, \ldots, n_k]).$

\begin{defn}\label{defn: ruling} Let $\Lambda\subset \R^3_{\st}$ be a Legendrian link and let us fix a front projection $D$ of $\Lambda$. A \emph{normal ruling} $R(D)$ of $\Lambda$ is a set of crossings called \emph{switches} such that resolving these switches yields a split union of unknots $\Lambda_1,\dots, \Lambda_m$, satisfying
\begin{enumerate}
    \item each $\Lambda_i$ for $i=1,\dots, m$ is a Legendrian max-tb unknot with zero crossings and two cusps, bounding an embedded ruling disk $D_i$;
    \item exactly two link components are incident to any switches after resolving;
    \item near each switch, the incident ruling disks $D_i$ bounded by $\Lambda_i$ are either nested or disjoint as shown in the first row of Figure \ref{fig:rul}.
\end{enumerate}
At a crossing with no switches, the crossing is called a \emph{departure} or \emph{return} (see Figure~\ref{fig:rul}).
\end{defn}

\begin{figure}[h]
\centering
\includegraphics[width=.6\textwidth]{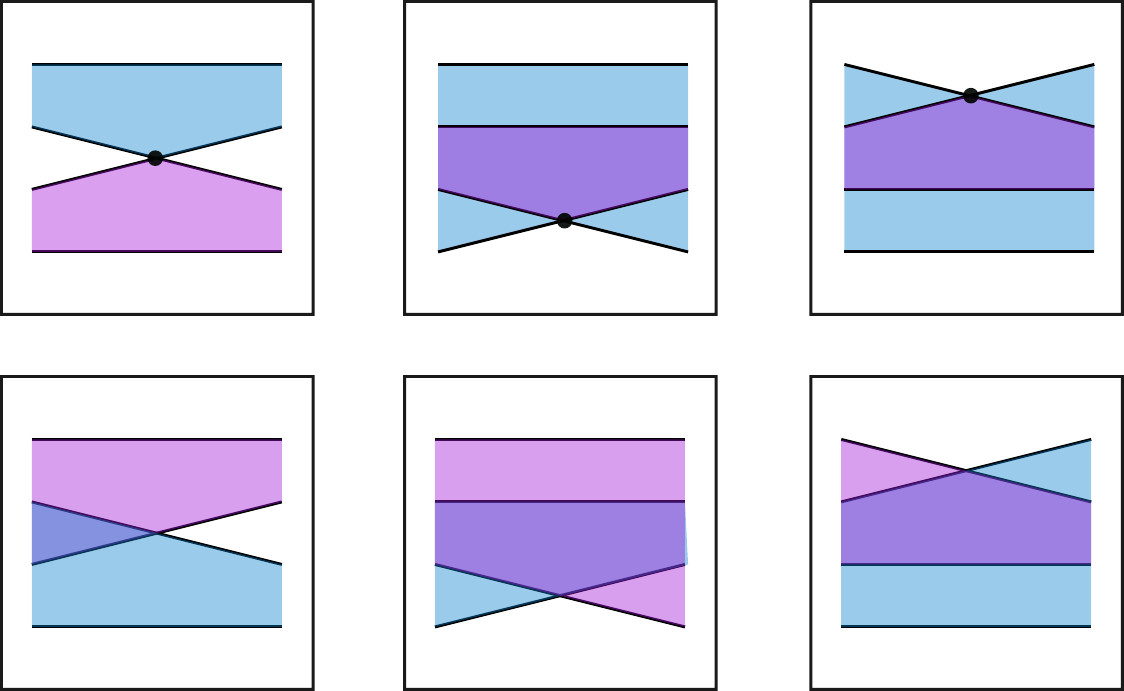}
\caption{The upper row shows all possible switches for a graded normal ruling. The configurations on the bottom row and any of their vertical reflections show the possible configurations of a graded normal ruling at a crossing with no switches. The bottom row shows all possible returns and the vertical reflections of the configurations on the bottom row show all possible departures.}
			\label{fig:rul}\end{figure}

Note that we consider only ungraded normal rulings, so the switches of a ruling $R(D)$  may occur at any crossing of the front $D$. For any Legendrian link $\Lambda \subset \R^3_{\st}$, there exists an ungraded normal ruling $R$ of $\Lambda$ if and only $\Lambda$ has an ungraded augmentation over $\bF$ ~\cite{Fuchs_Ishkhanov_2004,Sabloff_2005, henry_rutherford, Leverson_2016}. We briefly review how Henry and Rutherford obtain this result using Morse complex sequences which can be thought of as combinatorial objects motivated by generating families. A \emph{Morse complex sequence} of a Legendrian $\Lambda\subset \R^3_{\st}$ with a front $D$ is a finite collection of chain complexes and handleslide marks on $D$ which satisfy the same set restrictions 
as Morse complexes in a generic one-parameter family of functions. \emph{Handleslide marks} are line segments on the
front diagram $D$ of $\Lambda$  with an associated parameter $a\in \bF$. See~\cite[Definition 4.1]{henry_rutherford} for a complete definition of Morse complex sequences. Henry showed that there is a set of equivalence moves on Morse complex sequences in~\cite[Proposition 3.8]{Henry_2011}. A Legendrian front is in \emph{plat} position if every left cusp has the same $x$ coordinate, every crossing is transverse, and every right cups has the same $x$ coordinate. A \emph{nearly plat} front is the result of taking a plat front and perturbing it so that every crossing and cusp has a unique $x$ coordinate. Given a Legendrian link with a nearly plat front diagram $D$, any ungraded normal ruling of $D$ is in bijection with an ungraded Morse complex sequence (in SR form)~\cite[Definition 4.4, Lemma 4.5]{henry_rutherford}. This bijection is given as follows. 
Start with an ungraded normal ruling. For every switch, place two handleslide marks with $\mathbb{F}^{\times}$-parameters $u_s^{\pm 1}$ near each switch $s$ in $R$ as shown in Figure~\ref{fig: handle slides}. If the switch is nested, place an additional handleside mark between the top of the two ruling disks with $\alpha u^{-1}$, where $\alpha$ is a $\mathbb{F}^{\times}$ coefficient relating the boundary strands of the two disks prior to the handleslides we are placing. For every return place one handleslide mark with an $\mathbb{F}$-parameter $z_r$ on the left of each return as shown in the middle column of Figure~\ref{fig: handle slides}. If the return is nested, we add an additional handleslide mark with the parameter $\alpha z_r$. Again $\alpha$ is a coefficient relating the boundary strands of the two disks prior to the handleslides we are placing. Finally, for every right cusp, add one handleslide mark with a $\mathbb{F}$ parameter $v_R$.


Now slide the handleslide marks to the right as far as possible using equivalence moves on Morse complex sequences, before they are stopped by a crossing (see~\cite[Section 6]{henry_rutherford} for how to combine pairs of handleslide marks into one, and for other moves). At the end of this process, we obtain a Morse complex sequence with handleslide marks only to the left of each crossing and right cusp parametrized by sums of the parameters $u_s$'s, $\alpha'$s, and $z_r$'s. Note that only $b_2$ will have a handleslide mark parametrized by sums including the parameters $\alpha'$s. Such a Morse complex sequence with handleslide marks to the left of each crossing and right cusp is said to be in $A$-form. There is a bijection between ungraded Morse complex sequences in $A$ form and a family of ungraded augmentations~\cite[Theorem 1.1]{Henry_Rutherford_2015}: the parameters of each handleslide mark to the left of a crossing or right cups are assigned to the Reeb chord corresponding to the crossing or right cusp in Ng's resolution of the front. We define $\widetilde{W}_R$ to be the family of augmentations associated with a normal ruling $R$ through this process.

\begin{figure}
\begin{tikzpicture}[scale=1.5]
		\node[inner sep=0] at (0,0) {\includegraphics[width=0.7\textwidth]{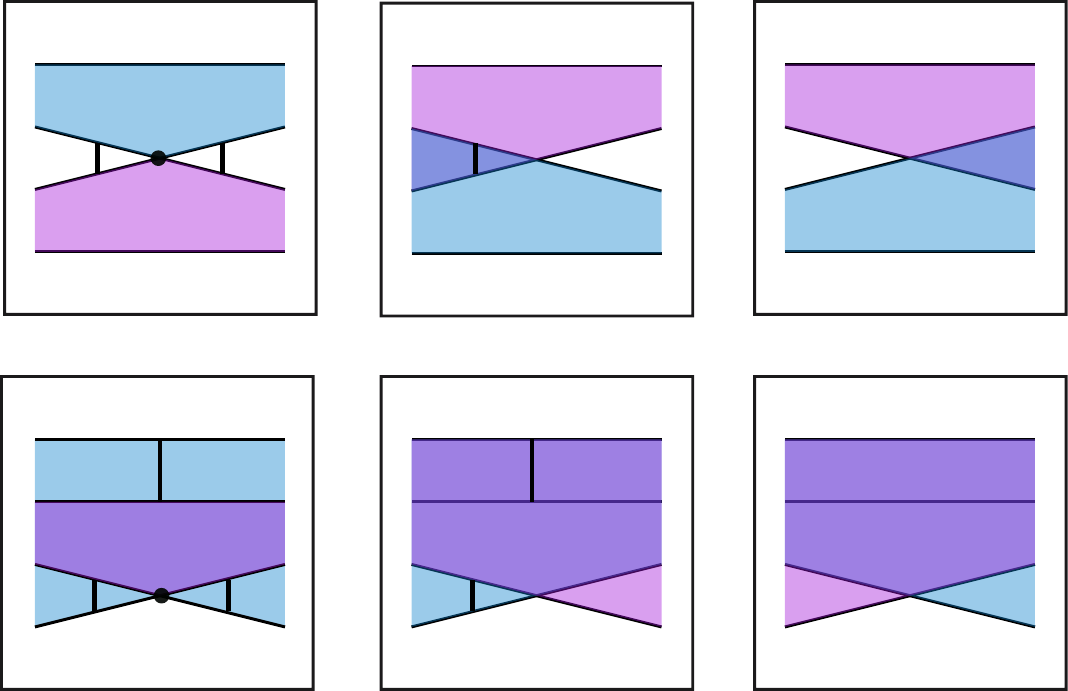}};
        \node at (-3.4,1.35){$u_s$};
        \node at (-1.9,1.35){$-u_s^{-1}$};
        \node at (-2.2,-0.9){$\alpha u_s^{-1}$};
        \node at (-2.1,-2.2){$-u_s^{-1}$};
        \node at (-3.2,-2.2){$u_s$};
        \node at (0.3,-0.9){$\alpha z_r$};
        \node at (-0.45,-2.2){$z_r$};
        \node at (-0.7,1.35){$z_r$};
    \end{tikzpicture}
    \caption{Assignment of handleslide marks near crossings to obtain a Morse complex sequence in SR form from an normal ruling.}
    \label{fig: handle slides}
\end{figure}

\begin{thm}[{\cite[Theorem 3.4]{henry_rutherford}}]\label{thm: ruling stratification} Suppose $D$ is the nearly plat front diagram of a Legendrian link $\Lambda$. Then, 
we can decompose the augmentation variety $\wAug_u(\Lambda)$ (with respect to the front projection $D$) into a disjoint union
$$\wAug_u(\Lambda)=\coprod_R \widetilde{W}_R$$
where the union is over all ungraded normal rulings of $D.$ Each subset $\widetilde{W}_R$ is the image of an injective regular map
$$\varphi_R:(\F^{\times})^{s(R)-\delta(R)+c} \times \F^{r(R)+\delta(R)}\hookrightarrow \wAug_u(\Lambda)$$
where $s(R)$ is the number of switches in $R$, $\delta(R)$ is the number of disks in $R$, $r(R)$ is the number of returns of $R$, and $c$ is the number of base points on $\Lambda.$
\end{thm}

For the Legendrian $2$-bridge links that we are interested in, the fronts shown in Figure~\ref{fig:front} can be easily placed into nearly plat position: for $k\in 2\Z$ perform a Reidemeister II move so that $\Lambda[n_1, \ldots, n_k]$ is now $\Lambda[n_1, \ldots, n_k+1, 1]$; for $k$ odd perturb the front so that each right cusp and crossing has their own $x$-coordinate.

\begin{prop}\label{prop: Reeb chords that must be departures and returns} Let $\Lambda=\Lambda[n_1,\dots, n_k]$ be a Legendrian $2$-bridge link in Legendrian rational form and let $R$ be a normal ruling of $\Lambda$. Then for each $1\leq j<k$, there is a departure at $a_{m_j}$ and a return at $a_{m_j+1}$.
\end{prop}
\begin{proof} Every block $n_i$ in $\Lambda$ has at least two crossings for $1<i<k$, so in between any of the crossings $a_{m_i}$ and $a_{m_i+1}$, any ruling disk cannot pair strands $1$ and $2$, and also cannot pair strands $2$ and $3.$ A switch or departure at $a_{m_i+1}$ implies that there is a ruling disk with boundary on strands $2$ and $3$ to the left of $a_{m_i+1}$ for $i$ odd while for $i$ even it implies that there is a ruling disk with boundary on strands $1$ and $2$. A switch or return at $a_{m_i}$ implies that there is a ruling disk with boundary on strands 1 and 2 to the right of $a_{m_i}$ for $i$ odd, while for $i$ even it implies that there is a ruling disk with boundary on strands 2 and 3. Thus, $a_{m_i}$ must be a departure and $a_{m_i+1}$ must be a return.
\end{proof}

\begin{cor}\label{cor: descend of ruling stratification} For a Legendrian $2$-bridge link $\Lambda=\Lambda[n_1,\dots, n_k]$ in Legendrian rational form, the stratification $\wAug_u(\Lambda)=\bigsqcup_R \widetilde{W}_R$ descends to a stratification of $\Aug_u(\Lambda)=\bigsqcup_RW_R$. In particular, each stratum 
\[
W_R\cong (\bF^\times)^{s(R)}\times \bF^{r(R)-k+1},
\]
where $s(R)$ and $r(R)$ are the numbers of switches and returns in the normal ruling $R$, respectively.
\end{cor}

\begin{defn} We call the stratification $\Aug_u(\Lambda)=\bigsqcup_RW_R$ the \emph{ruling stratification} of $\Aug_u(\Lambda)$ for a Legendrian link $\Lambda\subset \R^3_{\st}$.
\end{defn}

\begin{proof} Recall from Proposition \ref{prop:dgahtpy} that for a Legendrian $2$-bridge link $\Lambda$, two augmentations are homotopic if and only if they agree on all Reeb chords except at $b_1$, $b_2$, and $a_{m_j+1}$ for $1\leq j<k$. But these are all returns or right cusps in $R$ by Proposition \ref{prop: Reeb chords that must be departures and returns}, and therefore they can take any augmentation values under the parametrization. This shows that homotopic augmentations must belong in the same stratum $\widetilde{W}_R$, and hence the stratification on $\wAug(\Lambda)$ descends to a stratification on $\Aug_u(\Lambda)$.

The claim that $W_R\cong (\bF^\times)^{s(R)}\times \bF^{r(R)-k+1} $ follows from Theorem \ref{thm: ruling stratification} and the fact that we need to quotient $k+1$ many $\bF$ factors corresponding to $b_1$, $b_2$, and $a_{m_j+1}$ for all $1\leq j<k$.
\end{proof}

On the other hand, following~\cite{LS}, each acyclic cluster seed $\seed$ gives rise to a stratification of the cluster variety. We will prove that for the acyclic initial cluster seed described in Theorem \ref{thm: cluster structure on augmentation variety}, this stratification coincides with the ruling stratification on the ungraded augmentation varieties in the case of $2$-bridge links.

\begin{defn} Let $\seed$ be an acyclic cluster seed. Let $Q_\seed$ be the quiver associated with $\seed$ and let $\epsilon$ be the exchange matrix of $Q_\seed$ (c.f. Definition \ref{defn: quiver mutation}). A collection $I$ of mutable vertices in $Q_\seed$ is called an \emph{anticlique} if $\epsilon_{ij}=0$ for any two $i,j\in I$. In particular, we consider the empty set $\emptyset$ as an anticlique as well.
\end{defn}

\begin{defn} Let $\seed$ be an acyclic seed of an acyclic cluster variety $\mathscr{A}$. Given an anticlique $I$ of mutable vertices in $Q_\seed$, we define a subvariety $\mathscr{O}_I\subset \mathscr{A}$ by
\[
\mathscr{O}_I=\left(\bigcap_{i\in I}\{A_i=0\}\right)\cap\left(\bigcap_{i\notin I}\{A_i\neq 0\}\right).
\]
\end{defn}

\begin{thm}[{\cite[Corollary 3.2, Proposition 3.5]{LS}}]\label{thm: anticlique stratification} Let $\seed$ be a really full-rank acyclic seed of an acyclic cluster variety $\mathscr{A}$ and let $\mathcal{I}$ be the set of anticliques for $Q_\seed$. Then $\mathscr{A}=\bigsqcup_{I\in \mathcal{I}}\mathscr{O}_I$ and each stratum $\mathscr{O}_I$ is isomorphic to $\bF^{|I|}\times (\bF^\times)^{\dim \mathscr{A}-2|I|}$.
\end{thm}

\begin{cor} The ungraded augmentation variety $\Aug_u(\Lambda[n_1,\dots, n_k])$ admits an anticlique stratification where each stratum $\mathscr{O}_I$ is isomorphic to $\bF^{|I|}\times (\bF^\times)^{\dim \mathscr{A}-2|I|}$.
\end{cor}
\begin{proof} This follows from Corollary \ref{cor: augmentation variety is really full-rank} and Theorem \ref{thm: anticlique stratification}.
\end{proof}

By Proposition \ref{prop: Reeb chords that must be departures and returns}, we know that the crossing types at the crossings that are closest to the adjacent block are fixed for all normal rulings. Thus, each normal ruling of $\Lambda$ can be viewed as a product of ``normal rulings'' at each block, in the sense that we can choose the crossing types for the remaining undetermined crossings in each block, and the choices we make in one block do not affect any other block. 

\begin{defn} We define the following crossings in the front projection of $\Lambda[n_1,\dots, n_k]$ to be \emph{undetermined}:
\begin{itemize}
    \item in the first block, $a_1,\dots, a_{m_1-1}$ ($n_1-1$ many of them);
    \item for the middle blocks with $1<j<k$, $a_{m_{(j-1)}+2},\dots, a_{m_j-1}$ ($n_j-2$ many of them for each $j$);
    \item in the last block, $a_{m_{(k-1)}+2},\dots, a_{m_k}$ ($n_k-1$ many of them).
\end{itemize}
\end{defn}

Moreover, based on Definition \ref{defn: ruling} and the fact that the crossings in each block come as a $2$-stranded braid, we see that departures and returns must appear as adjacent pairs within the remaining crossings in each block, with a departure on the left and a return on the right. Note that the number of adjacent pairs of crossings among the remaining crossings in each block happens to be equal to the number of mutable vertices in the initial quiver $Q_\seed$ ($n_1-2$ for the first block, $n_j-3$ for the middle blocks, and $n_k-2$ for the last block). 

\begin{defn}\label{defn: mutable vertices corresponding to pairs} We associate each mutable quiver vertex in $Q_\seed$ with an adjacent pair of undetermined crossings in the front projection of $\Lambda=\Lambda[n_1,\dots, n_k]$, according to the same order from left to right within each block. Now given an anticlique $I$ in $Q_\seed$, we construct a corresponding normal ruling $R_I$ by assigning a departure-return pair to the pairs of adjacent crossings corresponding to elements in $I$, and assign switches to the remaining undetermined crossings. It is not hard to see that these assignments are compatible and do give rise to a normal ruling $R$ of $\Lambda$. See Figure~\ref{fig:strat} for an example.
\end{defn}

\begin{figure}[H]
	\centering
	\begin{tikzpicture}[scale=1.5]
		\node[inner sep=0] at (0,0) {\includegraphics[width=8 cm]{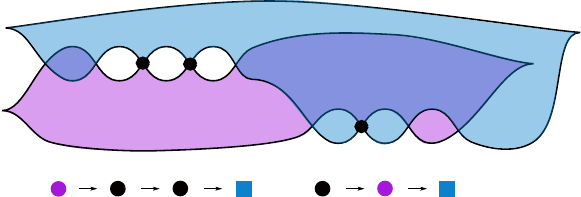}};
  \end{tikzpicture}
	\caption{An ungraded normal ruling of $\La[5,4]$ and the corresponding anticlique in the initial quiver $Q_\seed$ where the blue squares represent frozen cluster variables, the dots represent mutable cluster variables and the purple dots are the cluster variables contained in the anticlique. }
	\label{fig:strat}
\end{figure}

\begin{prop}\label{prop: bijection} The assignment $I\longmapsto R_I$ is a bijection between anticliques in the initial quiver $Q_\seed$ in Theorem \ref{thm: cluster structure on augmentation variety} and normal rulings of $\Lambda=\Lambda[n_1,\dots, n_k]$.
\end{prop}
\begin{proof} It suffices to prove this statement block by block. Consider a single block with $n$ undetermined crossings (so the block quiver has $n-1$ mutable vertices). Let $r(n)$ be the number of normal rulings for that block and let $i(n)$ be the number of anticliques in the block quiver. We do a proof by induction to show that $r(n)=i(n)$ for all $n$. For the base case $n=1$, there is one anticlique, namely the empty set $\emptyset$; on the front projection of the block, there is a single undetermined crossing, with a return to its left and a departure to its right\footnote{Similar patterns hold for the first and the last block, so we do not need to separate them into different cases.}, so the undetermined crossing must be a switch, and hence $r(1)=i(1)=1$. For the base case $n=2$, there are two anticliques: the empty set and the singleton set; in the front projection of the block, there are two undetermined crossings sandwiched between a return on the left and a departure on the right, so the undetermined crossings can be assigned in two ways: either both become switches, or become a departure-return pair, proving that $r(2)=n(2)=2$.

Inductively, suppose we are at case $n$. Let us consider the right-most mutable vertex $v$: it is either in an anticlique or not in an anticlique. The anticliques that contain $v$ must not contain its adjacent mutable quiver vertex, and hence there are $i(n-2)$ many such anticliques; the anticliques that do not contain $v$ can contain its adjacent mutable quiver vertex, and hence there are $i(n-1)$ such anticliques. Summing them up we can conclude that $i(n)=i(n-2)+i(n-1)$. Correspondingly, since the right-most undetermined crossing $a$ has a departure on its right, itself can only be either a return or a switch. If $a$ is a return, then the crossing before it must be a departure, which reduces to the case of two fewer crossings with $r(n-2)$ many rulings; if $a$ is a switch, then the crossing before it can still only be a return or a switch, which reduces to the case of one fewer crossings with $r(n-1)$ many rulings. Thus in total, there are $r(n)=r(n-2)+r(n-1)$ many normal rulings. Now by the inductive hypothesis, we can conclude that
\[
r(n)=r(n-2)+r(n-1)=i(n-2)+i(n-1)=i(n),
\]
and we do see that the crossings that are selected to form an anticlique correspond precisely to departure-return pairs.
\end{proof}

\begin{cor} Let $F_n$ denote the $n$th Fibonacci number (under the convention $F_1=1$, $F_2=1$, and $F_n=F_{n-2}+F_{n-1}$). Then there are $F_{n_1}F_{n_2-1}\cdots F_{n_{(k-1)}-1}F_{n_k}$ many strata in the ruling stratification as well as in the anticlique stratification $\Aug_u(\Lambda[n_1,\dots, n_k])$.
\end{cor}
\begin{proof} This follows from the recursion $r(n)=r(n-2)+r(n-1)$ with $r(1)=1$ and $r(2)=2$ in the proof of Proposition \ref{prop: bijection}.
\end{proof}

\begin{thm}\label{thm: stratification coincide} $W_{R_I}=\mathscr{O}_I$ as subvarieties in $\Aug_u(\Lambda[n_1,\dots, n_k])$.
\end{thm}
\begin{proof} We prove this statement block by block. Consider a block of $n$ undetermined crossings $a_1, a_2, \dots, a_n$. By following the construction of augmentations from normal rulings, the block factor for the ruling stratum is isomorphic to $(\bF^\times)^\sigma\times \bF^\rho$, where $\sigma$ is the number of switches among these undetermined crossings and $\rho$ is the number of returns among these undetermined crossings, and each point in this stratum uniquely determines an augmentation value $\e(a_i)$ for each of these undetermined crossings $a_i$. On the other hand, let $A_1\rightarrow A_2\rightarrow \cdots \rightarrow A_{n-1}\rightarrow \boxed{A_n}$ be the block quiver and the initial cluster variables in this block. Then the statement follows from the following two claims (note that $a_1$ cannot be a return by construction):
\begin{enumerate}
    \item if $a_i$ is a return, then $A_{i-1}=0$;
    \item if $a_i$ is not a return, then $A_{i-1}=\displaystyle\prod_{j \in S_i} u_j\neq 0$ where $S_i$ is a set consisting of all indices $j<i$ for which $a_j$ is a switch.
\end{enumerate}
In particular, claims (1) and (2) imply that $W_{R_I}\subset \mathscr{O}_I$. Since $\Aug_u(\Lambda)=\bigsqcup_{R_I}W_{R_I}=\bigsqcup_I \mathscr{O}_I$, it follows that $W_{R_I}=\mathscr{O}_I$.

So it remains to prove these two claims, and we do so by an induction on $i$. For the base case $i=2$, if $a_2$ is a return, then $a_1$ must be a departure. As there is no handle slide from the left arriving at $a_1$, $A_1=a_1=0$. If $a_2$ is not a return, then $a_1$ is not a departure; but $a_1$ cannot be a return, either, so $a_1$ must be a switch, which has a handle slide with an invertible parameter $u_1$ on the left. This shows that $A_1=a_1=u_1\neq 0$.

Inductively, suppose these two claims are true up to $a_{i-1}$. For (1), let us suppose that $a_i$ is a return. Then $a_{i-1}$ is a departure, and $a_{i-2}$ is a switch or a return. If $a_{i-2}$ is a switch, then $a_{i-1}$ would have a handle slide coming from the left with an invertible parameter $u_{i-2}^{-1}$. By the inductive hypothesis and the recurrence of continuants (Definition \ref{def of K_n}), we have
\begin{align*}
A_{i-1}=&K_{i-1}(a_1,\dots, a_{i-1})\\
=&K_{i-3}(a_1,\dots, a_{i-3})+a_{i-1}K_{i-2}(a_1,\dots, a_{i-2})\\
=&A_{i-3}+a_{i-1}A_{i-2}\\
=&\prod_{j\in S_i} u_j+u_{i-2}^{-1}\left(\prod_{j\in S_i} u_j\right)u_{i-2}\\
=&0.
\end{align*}
If $a_{i-2}$ is a return, then $a_{i-1}$ would not have any handle slide coming from the left, and hence $a_{i-1}=0$. But by the inductive hypothesis, $A_{i-3}=0$ as well; thus,
\[
A_{i-1}=K_{i-1}(a_1,\dots, a_{i-1})=K_{i-3}(a_1,\dots, a_{i-3})+a_{i-1}K_{i-2}(a_1,\dots, a_{i-2})=A_{i-3}+a_{i-1}A_{i-2}=0+0=0.
\]

For (2), let us suppose that $a_i$ is not a return. The same recurrence relation $A_{i-1}=a_{i-1}A_{i-2}+A_{i-3}$ still holds. Note that the right-hand side does not care whether $a_i$ is a departure or a switch, so we only need to consider the following three possibilities:
\begin{center}
\begin{tabular}{|c|c|c|c|c|} \hline
    $i-1$ & $i-2$ & $a_{i-1}$ & $A_{i-2}$ &  $A_{i-3}$ \\ \hline
     return &  departure & $z_{i-1}$ & $0$ & $\prod_{j\in S_{i-2}} u_j$ \\ \hline
     switch & switch & $u_{i-2}^{-1}+u_{i-1}$ &  $\prod_{j\in S_{i-1}} u_j$ & $\prod_{j\in S_{i-2}}u_j$\\ \hline
    switch & return & $u_{i-1}$ & $\prod_{j\in S_{i-1}} u_j$ & 0 \\ \hline
\end{tabular}
\end{center}
The computation for each of the three rows is as follows:
\[
A_{i-1}=a_{i-1}A_{i-2}+A_{i-3}=z_{i-1}\cdot 0+\prod_{j\in S_{i-2}} u_j=\prod_{j\in S_{i-2}} u_j=\prod_{j\in S_i} u_j,
\]
\[
A_{i-1}=a_{i-1}A_{i-2}+A_{i-3} = \left(u_{i-2}^{-1}+u_{i-1}\right)\prod_{j\in S_{i-1}}u_j+\prod_{j\in S_{i-2}}u_j=u_{i-1}\prod_{j\in S_{i-1}}u_j=\prod_{j\in S_i}u_j,
\]
\[
A_{i-1}=a_{i-1}A_{i-2}+A_{i-3} = u_{i-1}\prod_{j\in S_{i-1}}u_j +0=u_{i-1}\prod_{j\in S_{i-1}}u_j=\prod_{j\in S_i}u_j.
\]
Note that (2) is true for all three cases. This finishes the proof.
\end{proof}

In the last part of this section, we relax the characteristic $2$ assumption and consider a finite field $\bF_q$ of $q$ elements. We would like to prove a result relating the $\bF_q$-point counts of certain type A cluster varieties and a leading coefficient of Kauffman polynomials for $2$-bridge links.

\begin{lem}\label{lem: point count for a single block} The $\bF_q$-point count of a cluster variety\footnote{Since cluster varieties are defined over $\Z$, one can base change them to any characteristic.} with $A_1\rightarrow A_2\rightarrow \cdots \rightarrow A_{n-1}\rightarrow \boxed{A_n}$ as an initial cluster seed is 
\[
f_n(q):=\sum_{k=0}^n (-1)^{n-k} q^k.
\]
\end{lem}
\begin{proof} Let us do an induction on $n$. For the base case $n=1$, there is only one frozen vertex and hence the cluster variety is isomorphic to $\bF_q^\times$, which has $q-1$ points. Inductively, vertex $1$ may or may not be in the anticlique. If
vertex $1$ is in an anticlique, then vertex $2$ cannot be in the anticlique; this reduces to the case with $n-2$ many vertices times an $\bF$ factor (Theorem \ref{thm: anticlique stratification}), and by the inductive hypothesis we have $qf_{n-2}(q)$ many points. If vertex $1$ is not in an anticlique, then vertex $1$ itself becomes a $\bF_q^\times$ factor, and again by the inductive hypothesis we have $(q-1)f_{n-1}(q)$ many points. In total, we have
\[
qf_{n-2}(q)+(q-1)f_{n-1}(q)=q\sum_{k=0}^{n-2}(-1)^{n-k-2}q^k+(q-1)\sum_{k=0}^{n-1}(-1)^{n-k-1}q^{k}=\sum_{k=0}^n(-1)^{n-k}q^k
\]
many points in the cluster variety. This finishes the induction.
\end{proof}

\begin{defn}\label{defn: a certain product of A varieties} Let $\mathscr{A}[n_1,\dots, n_k]$ be the cluster variety with an initial cluster seed whose quiver is a disjoint union of $k$ many type A quivers as in Lemma \ref{lem: point count for a single block}, with lengths $n_1-1, n_2-2,n_3-2,\dots, n_{k-1}-2, n_k-1$, respectively. Note that the ungraded augmentation variety $\Aug_u(\Lambda[n_1,\dots, n_k])$ is isomorphic to $\mathscr{A}[n_1, \dots, n_k]$ over the chosen algebraically closed field $\bF$ of characteristic $2$.
\end{defn}

\begin{cor}\label{cor: point count of product of type A cluster varieties} Let $\mathscr{A}[n_1,\dots, n_k]$ be defined as in Definition \ref{defn: a certain product of A varieties}. Then its $\bF_q$-point count is
\[
|\mathscr{A}[n_1,\dots, n_k](\bF_q)|=f_{n_1-1}(q)f_{n_2-2}(q)\cdots f_{n_{k-1}-2}(q) f_{n_k-1}(q).
\]
\end{cor}
\begin{proof} This follows directly from Lemmas \ref{lem: product of cluster varieties} and \ref{lem: point count for a single block}.
\end{proof}

The Kauffman polynomial $F_{\kappa}(a,z)$ is a $2$-variable polynomial invariant associated with a topological link $\kappa$.  When $\Lambda$ is a max-tb Legendrian representative of $\kappa$, we define $B_\Lambda(z)$ to be the coefficient of $a^{-\tb(\Lambda)-1}$ of $F_{\kappa}(a,z)$. Rutherford proved the following result that relates this coefficient with normal rulings of $\Lambda$.

\begin{thm}[{\cite[Theorem 3.1]{Rutherford06}}]\label{thm:ruling_kauff} Fix a front diagram $D$ of $\Lambda$. Then
$B_\Lambda(z)=\sum_R z^{s(R)-\lambda+1}$, where the sum is taken over all normal rulings $R$ of $\Lambda$, $s(R)$ denotes the number of switches in $R$, and $\lambda$ is the number of left cusps in the front diagram $D$.
\end{thm}

\begin{rmk} The polynomial $\sum_R z^{s(R)-\lambda+1}$ is also known as the \emph{ruling polynomial} of $\Lambda$.  More precisely, the polynomial $B_{\La}(z)$ is a particular term of the \emph{Dubrovnik version} of the Kauffman polynomial, an equivalent formulation discovered by the eponymous mathematician while he was staying in the city of Dubrovnik. See \cite[Section 7]{Kauffman90} for a short anecdote concerning its origins and a simple substitution for obtaining the more commonly known expression of the Kauffman polynomial. We note that the coefficients $B_{\La}(z)$ are all non-negative integers, a fact that does not hold for the other version of the polynomial. This non-negativity is crucial for our purposes of obtaining it as the point count of an algebraic variety. 
\end{rmk}

\begin{thm}\label{thm: Kauffman polynomial and point count} Let $\mathscr{A}[n_1,\dots, n_k]$ be defined as in Definition \ref{defn: a certain product of A varieties}. Then
\[
|\mathscr{A}[n_1,\dots, n_k](\bF_q)|=q^{\frac{m_k}{2}-k+1}(q^{\frac{1}{2}}-q^{-\frac{1}{2}})B_{\Lambda[n_1,\dots, n_k]}(q^{\frac{1}{2}}-q^{-\frac{1}{2}}).
\]
\end{thm}
\begin{proof} Although we have only shown in Theorem \ref{thm: stratification coincide} that the ruling stratification and the anticlique stratification coincide on $\Aug_u(\Lambda[n_1,\dots, n_k])$, which is an affine variety in characteristic $2$, we see that each stratum is a product of an affine space with an algebraic torus, and hence over any finite field $\bF_q$, we can still match the anticlique stratification strata with the description in terms of switches and returns as in Corollary \ref{cor: descend of ruling stratification}. This implies that 
\[
|\mathscr{A}[n_1,\dots, n_k](\bF_q)|=\sum_R(q-1)^{s(R)}q^{r(R)-k+1}.
\]
On the other hand, we know that $s(R)+r(R)+d(R)=m_k$ (here $d(R)$ stands for the number of departures). Since returns and departures come in pairs, we can deduce that $r(R)=\frac{m_k-s(R)}{2}$. Plugging this into the equation above yields
\begin{align*}
|\mathscr{A}[n_1,\dots, n_k](\bF_q)|=&\sum_R(q-1)^{s(R)}q^{\frac{m_k-s(R)}{2}-k+1}\\
=&\sum_R(q^{\frac{1}{2}}-q^{-\frac{1}{2}})^{s(R)}q^{\frac{m_k}{2}-k+1}\\
=&q^{\frac{m_k}{2}-k+1}(q^{\frac{1}{2}}-q^{-\frac{1}{2}})B_{\Lambda[n_1,\dots, n_k]}(q^{\frac{1}{2}}-q^{-\frac{1}{2}}),
\end{align*}
where the second to last equality follows from Theorem~\ref{thm:ruling_kauff}, which says $B_{\Lambda[n_1,\dots, n_k]}(z)=\sum_Rz^{s(R)-1}$ in the case of Legendrian $2$-bridge links.\end{proof}

\begin{cor} For a Legendrian $2$-bridge link $\Lambda[n_1,\dots, n_k]$, 
\[
B_{\Lambda[n_1,\dots, n_k]}(q^{\frac{1}{2}}-q^{-\frac{1}{2}})=\frac{f_{n_1-1}(q)f_{n_2-2}(q)\cdots f_{n_{k-1}-2}(q)f_{n_k-1}(q)}{q^{\frac{m_k}{2}-k+1}(q^{\frac{1}{2}}-q^{-\frac{1}{2}})},
\]
where $f_n(q)$ is defined in Lemma \ref{lem: point count for a single block} and $m_k:=\sum_i n_i$.
\end{cor}
\begin{proof} It follows from Corollary \ref{cor: point count of product of type A cluster varieties} and Theorem \ref{thm: Kauffman polynomial and point count}.
\end{proof}

\bibliographystyle{alphaurl-a}
\bibliography{main}
\end{document}